\documentclass{amsart}
\usepackage{amsmath,amsthm}
\usepackage{amsfonts,amssymb}
\usepackage{accents}
\usepackage{enumerate}
\usepackage{accents,color}
\usepackage{graphicx}
\usepackage{wrapfig}
\newcommand{\lvt}{\left|\kern-1.35pt\left|\kern-1.3pt\left|}
\newcommand{\rvt}{\right|\kern-1.3pt\right|\kern-1.35pt\right|}

\newtheorem{thm}{Theorem}[section]
\newtheorem{cor}[thm]{Corollary}
\newtheorem{lem}[thm]{Lemma}
\newtheorem{prop}[thm]{Proposition}

\newtheorem{defn}[thm]{Definition}

\theoremstyle{remark}
\newtheorem{rem}{Remark}[section]

 \def\la{{\langle}}
 \def\ra{{\rangle}}

 \def\sph{{\mathbb{S}^{d-1}}}

 \def\sC{{\mathsf C}}

 \def\sH{{\mathsf H}}
 \def\sJ{{\mathsf J}}
 \def\sK{{\mathsf K}}
 \def\sL{{\mathsf L}}
 
 \def\sP{{\mathsf P}}
 \def\sQ{{\mathsf Q}}
 \def\sS{{\mathsf S}}

 \def\d{{\mathrm{d}}}
  
 \def\a{{\alpha}}
 \def\b{{\beta}}
 \def\g{{\gamma}}
 \def\k{{\kappa}}
 \def\t{{\theta}}
 \def\l{{\lambda}}
 
 \def\s{\sigma}
 \def\la{{\langle}}
 \def\ra{{\rangle}}

 \def\kb{{\mathbf k}}

 \def\ub{{\mathbf u}}
 
 \def\xb{{\mathbf x}}
 \def\yb{{\mathbf y}}

 \def\Cb{{\mathbf C}}
 \def\Hb{{\mathbf H}}
 \def\Jb{{\mathbf J}}
 \def\Qb{{\mathbf Q}}
 \def\Pb{{\mathbf P}}
 \def\Kb{{\mathbf K}}
 \def\Lb{{\mathbf L}}
 \def\Sb{{\mathbf S}}

 \def\CH{{\mathcal H}}

 \def\CO{{\mathcal O}}

 \def\CV{{\mathcal V}}

 \def\BB{{\mathbb B}}

 \def\NN{{\mathbb N}}

 \def\RR{{\mathbb R}}
 \def\SS{{\mathbb S}}
 
 \def\VV{{\mathbb V}}
 \def\UU{{\mathbb U}}
 \def\ZZ{{\mathbb Z}}
      \def\proj{\operatorname{proj}}

\def\lla{\langle{\kern-2.5pt}\langle}      
\def\rra{\rangle{\kern-2.5pt}\rangle}

\newcommand{\wh}{\widehat}

\def\f{\frac}

\graphicspath{{./}}
\begin{document}

\title[Orthogonal structure and orthogonal series on hyperboloid]
{Orthogonal structure and orthogonal series in and on a double cone or a hyperboloid}
 
\author{Yuan Xu}
\address{Department of Mathematics\\ University of Oregon\\
    Eugene, Oregon 97403-1222.}\email{yuan@uoregon.edu}

 
\date{\today}

\begin{abstract} 
We consider orthogonal polynomials on the surface of a double cone or a hyperboloid of revolution, either
finite or infinite in axis direction, and on the solid domain bounded by such a surface and, when the surface 
is finite, by hyperplanes at the two ends. On each domain a family of orthogonal polynomials, 
related to the Gegebauer polynomials, is study and shown to share two characteristic properties 
of spherical harmonics: they are eigenfunctions of a second order linear differential operator with eigenvalues
depending only on the polynomial degree, and they satisfy an addition formula that provides a closed form 
formula for the reproducing kernel of the orthogonal projection operator. The addition formula leads to a 
convolution structure, which provides a powerful tool for studying the Fourier orthogonal series on these 
domains. Furthermore, another family of orthogonal polynomials, related to the Hermite polynomials, is 
defined and shown to be the limit of the first family, and their properties are derived accordingly. 
\end{abstract}

\maketitle

\section{Introduction}
\setcounter{equation}{0}

The study of orthogonal polynomials and the Fourier orthogonal series in several variables has seen 
substantial progress in recent years (cf. \cite{DX}). The most useful and the most well studied families 
of multivariable orthogonal polynomials are those on regular domains, such as cubes and other tensor 
product domains, spheres, balls and simplexes, especially those that can be regarded as analogues of 
classical orthogonal polynomials of one variable. There have been, however, few works beyond the 
regular domains. Recently we start to examine orthogonal structure on a quadratic surface of revolution
or in the domain bounded by such a surface, taking a cue from spherical harmonics on the unit sphere 
and classical orthogonal polynomials on the unit ball. 

Spherical harmonics serve as our quintessential example on quadratic surfaces. They are orthogonal with
respect to the surface measure on the unit sphere. Among their numerous properties, we single out two
characteristics ones. Let $\CH_n^d$ be the space of spherical harmonics of degree $n$ in $d$ variables. 
Then
\begin{enumerate}[\quad    (I)]
 \item Spherical harmonics are eigenfunctions of the Laplace--Beltrami operator $\Delta_0$ for the unit 
 sphere $\sph$, 
\begin{equation*} 
   \Delta_0 Y = - n (n+d-2) Y, \qquad \forall Y \in \CH_n^d, \quad n =0,1,2,\ldots.
\end{equation*}
\item Spherical harmonics satisfy an addition formula:  Let $\{Y_\ell\}$ be an orthonormal basis of $\CH_n^d$;
then  
\begin{equation*} 
   \sum_{\ell} Y_\ell (\xi) Y_\ell (\eta) = Z_n^{\f{d-2}{2}}(\la \xi, \eta\ra), \qquad \xi,\eta \in \sph, 
\end{equation*}
where $Z_n^\l(t)= \frac{n+\l}{\l} C_n^\l(t)$ and $C_n^\l$ is the Gegenbauer polynomial of degree $n$.
\end{enumerate}
These two properties are fundamental for approximation theory and harmonic analysis on the unit sphere;
see, for example, \cite{DaiX, DX, FS, SW} and references therein. While the structure of eigenfunctions can 
be used to describe smoothness of functions on the sphere, the addition formula provides a closed form 
formula for the reproducing kernel of $\CH_n^d$ that possesses a structure of one-dimension. 

For the unit ball $\BB^d$, bounded by the surface $\sph$, the classical orthogonal polynomials are orthogonal
with respect to the weight function $\varpi_\mu(x)=(1-\|x\|^2)^{\mu-\f12}$, $\mu \ge 0$; they also possess 
analogues of the two characteristic properties. Let $\CV_n^d(\varpi_\mu)$ denote the space of orthogonal 
polynomials of degree $n$ on the ball. The polynomials in this space are eigenfunctions of a second order 
linear differential operator,
\begin{equation*} 
  \left( \Delta  - \la x,\nabla \ra^2  - (2\mu+d-1) \la x ,\nabla \ra \right)u = - n(n+2\mu+ d-1) u, \quad u \in \CV_n^d(\varpi_\mu), 
\end{equation*}
which plays the role of the Laplace--Beltrami operator. Furthermore, let $\{P_\kb: |\kb| = n\}$ be an 
orthonormal basis of $\CV_n^d(\varpi_\mu)$; then the addition formula for $\CV_n^d(\varpi_\mu)$ on 
the unit ball takes the form
$$
  \sum_{|\kb| = n} P_\kb(x) P_\kb(y) = c_\mu \int_{-1}^1 Z_n^{\mu+\f{d-1}{2}}\left(\la x, y\ra + u\sqrt{1-\|x\|^2} 
      \sqrt{1-\|y\|^2}\right) (1-u^2)^{\mu-1}\d u. 
$$
This provides a closed form formula for the reproducing kernel of $\CV_n(\varpi_\mu)$ on $\BB^d$.
While the differential equation on the ball was known to Hermite, at least for $d =2$ (cf. \cite{AF}), the 
addition formula on the ball was discovered more recently in \cite{X99} and it is instrumental for recent 
advances of analysis on the unit ball; see, for example, \cite{OC, DaiX, DaiX10, KePX, KL, KPX, PX, SS, W, WZ} 
and the references there. 

Recently in \cite{X19} we considered orthogonal structure on the surface of the cone of revolution,  
$\VV_0 = \{(x,t): \|x\| = t, x \in \RR^d, \, t \in [0, b]\}$, where $b$ can be $+\infty$, and in the solid cone 
$\VV^{d+1}$ bounded by $\VV_0^{d+1}$ and by the hyperplane $t = b$ when $b$ is finite, and studied two 
families of orthogonal polynomials, which can be called Jacobi polynomials and Laguerre polynomials, 
on the surface and in the solid cone. In both domains, our main result shows that these two families 
are eigenfunctions of a second order linear differential operator and, furthermore, the Jacobi polynomials
on the cone also satisfy an addition formula. This shows, in particular, that the Jacobi polynomials on the 
cone satisfy both characteristic properties. The study uses an orthogonal basis that is explicit constructed. 
The construction of the basis is shown in \cite{OX1} to be possible for orthogonal polynomials in and on other 
quadratic surfaces of revolution. 

The purpose of this paper is to study orthogonal structure on the quadratic surface 
$$
 {}_\varrho\VV_0^{d+1} = \left \{(x,t): \|x\|^2 = c^2 (t^2 - \varrho^2), \quad x\in \RR^d, \quad \varrho \le |t| \le b\right\},
$$
where $\varrho$ and $b$ are nonnegative real numbers and $b$ could be infinity, and on the solid 
domain $\VV^{d+1}$ bounded by $\VV_0^{d+1}$ and the hyperplanes $t = \pm b$ if $b$ is finite. 
The surface is a double hyperboloid when $\varrho >  0$ and it degenerates to a double cone when 
$\varrho =0$. Examples of these surfaces are depicted in Figure 1. 
\begin{figure}[ht]
  \begin{center}
  \includegraphics[width=0.36\textwidth]{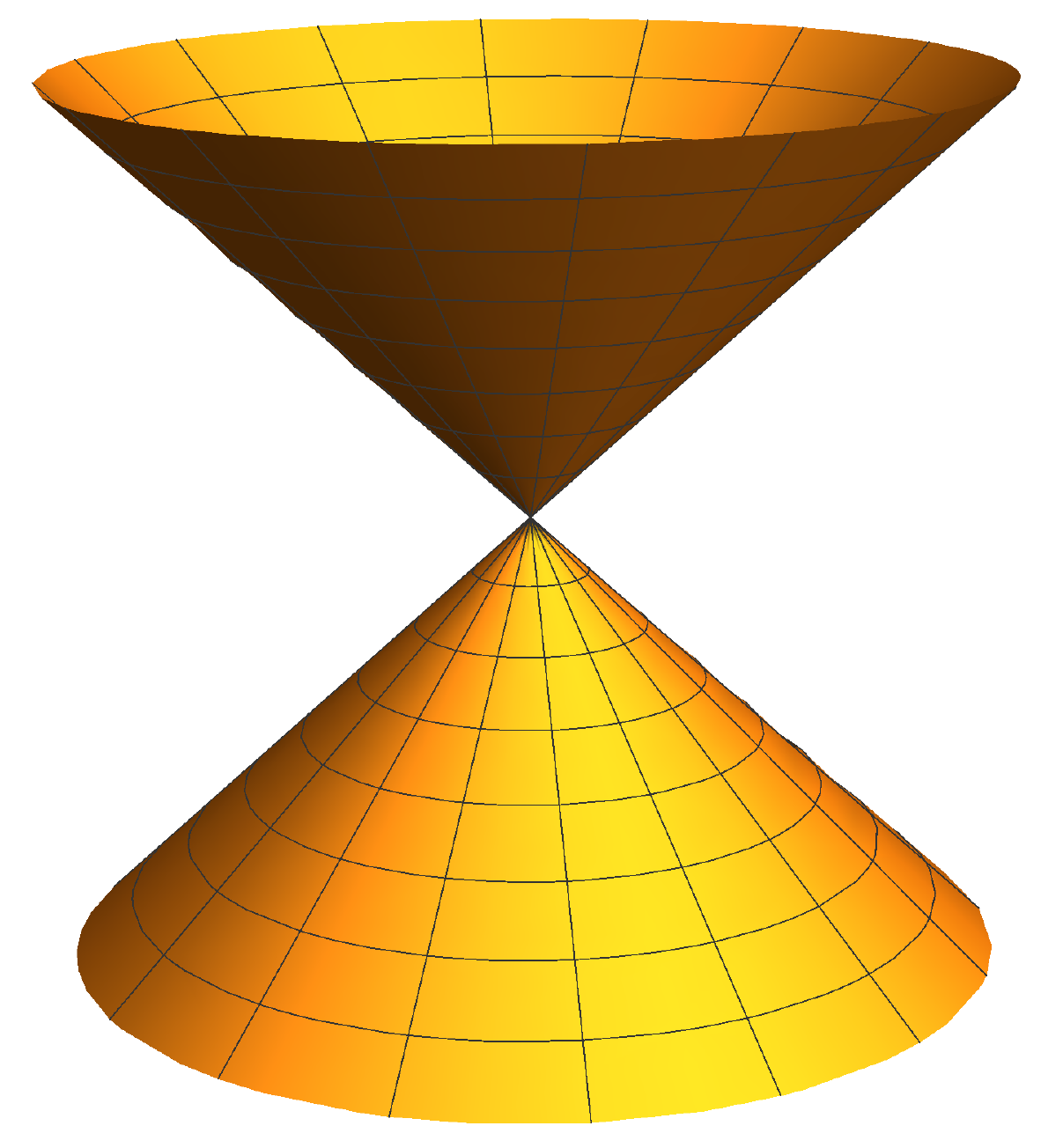}\qquad\quad \includegraphics[width=0.35\textwidth]{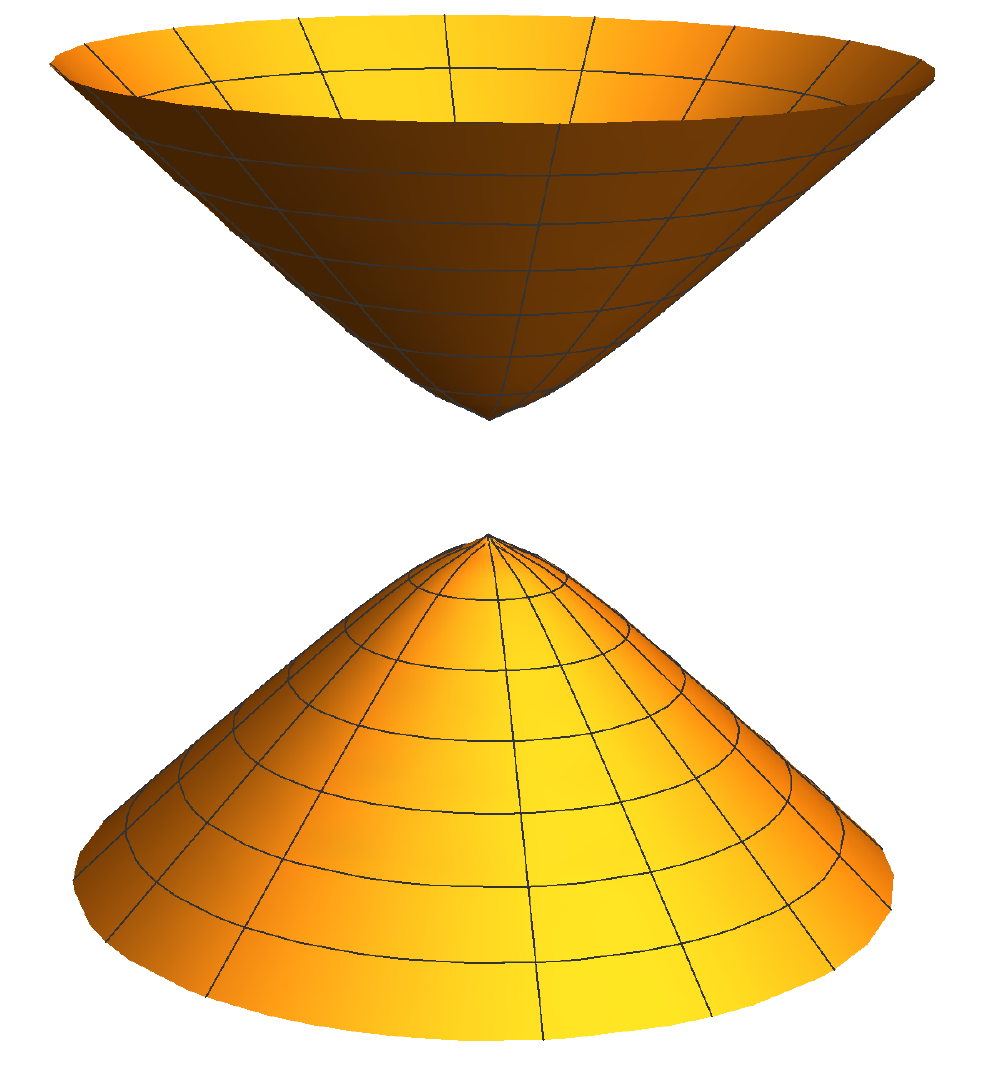} 
  \end{center}
  \caption{Double cone for $\varrho = 0$, double hyperboloid for $\varrho= 1$; $c=1$.}
\end{figure}
The paper can be regarded as a sequel of \cite{X19} and the main goal is to see if it is possible to establish 
the two characteristic properties for orthogonal polynomials on these domains. In this setting it is most 
natural to study orthogonality with respect to weight functions that are even in $t$ variable. For the surface 
$\VV_0^{d+1}$, what we call the Gegenbauler polynomials on the hyperboloid are orthogonal with respect to 
$w_{\b,\g}(t)= |t|(t^2-\varrho^2)^{\b-\f12}(1+\varrho^2 - t^2)^{\g-\f12}$ with $\b =0$, which becomes 
$(1-t^2)^{\g-\f12}$ on the double cone, and the Hermite polynomials on the hyperboloid are orthogonal with 
respect to $w_{\b}(t) = |t|(t^2-\varrho^2)^{\b-\f12}e^{-t^2}$ with $\b =0$, which becomes $e^{-t^2}$ on the double cone, 
whereas for the solid $\VV_0^{d+1}$ we can multiply these weight functions in $t$ by 
$(t^2 - \varrho^2 - \|x\|^2)^{\mu-\f12}$.

Our main effort lies in the studying of the Gegenbauer polynomials in/on the double cone and the 
hyperboloid. It turns out that these polynomials divide naturally into two groups according to the 
parity of the orthogonal polynomials in $t$ variable, and we bestow the names, Gegenbauer, only for 
orthogonal polynomials that are even in $t$ variable when working with hyperboloids. Our main result 
shows that, for each domain, the Gegenbauer polynomials, even in $t$ variable, are eigenfunctions of a 
second order linear differential operator and they also satisfy an addition formula, which holds more generally
for all permissible parameter $\b$. Moreover, for the double cone, $\varrho =0$,  orthogonal polynomials that 
are odd in $t$ variable also satisfy these two characteristic properties, but for different $\b$ in the 
Gegenbauer weight. There is no single differential operator that has all orthogonal polynomials of the 
same degree as eigenfunctions. These results, restricted to polynomials with the same parity, are still 
invaluable for studying the Fourier orthogonal series on the hyperboloid. Indeed, if a function $f \in L^2$ 
on the surface or on the solid hyperboloid is even (or odd) in $t$ variable,  then its Fourier orthogonal series 
contains only orthogonal polynomials that are even (or odd) in $t$ variable, just like the Fourier cosine 
(or sine) series for even (or odd) functions in classical Fourier series. 

The Hermite polynomials in or on the hyperboloid turn out to be limits of the corresponding family of
the Gegenbauer polynomials, which holds more generally for all permissible parameter $\b$. As a 
consequence, we deduce that the Hermite polynomials on the double cone and on the hyperboloid, 
even in $t$ variable, are also eigenfunctions of a second order linear differential operator. However, 
these polynomials no longer satisfy an addition formula, but their Poisson kernel satisfies a closed
form of a Mehler-type formula, just like the product Hermite polynomials on $\RR^d$. Such formulas 
are instrumental for studying Hermtie and Laguerre series in $\RR^d$ \cite{Tha}; see, for example, 
\cite{BRT, Tha2} for more recent work and references. 

For the half cone, our study in this paper and in \cite{X19}, together, provide four families of orthogonal 
polynomials associated with the classical weight functions. Among those,  the Jacobi polynomials and
the Laguerre polynomials on the cone are treated as a whole of all orthogonal polynomials of degree $n$, 
whereas the Hermite and the Gegenbauer polynomials on the cone need to be treated separately according 
to the parity of the polynomials in $t$ variable. The two on infinite domains resemble product 
Lagurerre on $\RR_+^d$ and product Hermite polynomials on $\RR^d$, whereas the two on compact
domains resemble the Jacobi polynomials on the simplex and the Gegenbauer polynomials on the 
unit ball. In this regard, the orthogonal structures on the cone is surprisingly rich and provides fertile 
ground for further analysis on the cones.

The paper is organized as follows. In the next section, we recall necessary background on orthogonal
polynomials and review basics for spherical harmonics and classical orthogonal polynomials on the unit
ball. The orthogonal structure for a generic even weight function on the surface or solid hyperboloid 
and double cone is studied in Section 3, which provides foundation for further study of two families of
orthogonal polynomials in the paper.  The Gegenbauer polynomials on the double cone and on the 
hyperboloid are treated in Section 4, where they are shown to be eigenfunctions of a second order 
differential operator. Addition formulas for the Gegenbauer polynomials on hyperboloids are derived in 
Section 5, where we establish the main result for a more general family of polynomials, called generalized 
Gegenbauer polynomials on hyperboloids, for a generic $\b$ in the weight function. In Section 6, 
we show how addition formula leads to a convolution structure and used it for studying the Fourier 
orthogonal series. In Section 7, we treat the Hermite polynomials and their 
generalizations for a generic $\b \ge 0$ on the hyperboloid as limits of the Gegenbauer polynomials
and their generalizations. Finally, in Section 8, we discuss further extensions of our results to the
Dunkl setting that permits an additional reflection invariant weight function of $x$ variables. In Appendix A 
we list properties of the generalized Gegenbauer polynomials on $[-1,1]$ are the generalized Hermite 
polynomials on $\RR$ that are needed in our development. 

\section{Preliminary}
\setcounter{equation}{0}

In the first subsection, we introduce necessary notations for orthogonal polynomials of one variable that
will be used throughout this paper. We review spherical harmonics and classical orthogonal polynomials 
on the unit ball, as well as the Fourier orthogonal series in terms of them, in the second and the third subsection. 
These will serve as building blocks of our results on hyperboloids and provide a general guideline for our study. 

\subsection{Orthogonal polynomials of one variable}
By a weight function $w$ on $\RR$, we mean a nonnegative function with infinite support and 
finite moments, which warrants the existence of orthogonal polynomials. We denote by
$p_n(w)$ the orthogonal polynomials of degree $n$ with respect to $w$. Then 
$$
    c_w \int_{\RR} p_n(w;x) p_m(w;x) w(x) \d x = h_n(w) \delta_{m,n}, \qquad m,n = 0,1,2,\ldots, 
$$ 
where $h_n(w)$ is the $L^2$ norm of $p_n(w)$ and $c_w$ is the normalization constant of $w$ so that 
$h_0(w) =1$. The Forurier orthogonal series of $f \in L^2(w)$ is defined by 
$$
  f = \sum_{n=0}^\infty \wh f_n p_n(w), \qquad \wh f_n =  \frac{1}{h_n(w)} c_w\int_{\RR} f(x) p_n(w;x) w(x) \d x.
$$
We denote its $n$-th partial sum by $s_n (w;f)$, which can be written as an integral 
$$
  s_n(w;f) = \sum_{k=0}^n \wh f_k p_k(w) =  c_w\int_{\RR} f(y) k_n(w;x,y) w(y) \d y
$$
in terms of the kernel $k_n(w)$ defined by 
\begin{equation} \label{eq:1dkernel}
     k_n(w;x,y) = \sum_{k=0}^n\frac{p_k(w;x) p_k(w;y)}{h_k(w)}. 
\end{equation}

We will need the classical orthogonal polynomials, which can be given explicitly in terms of
hypergeometric functions defined by 
$$
  {}_2 F_1\left(\begin{matrix} a, b \\ c \end{matrix}; x\right) := \sum_{n=0}^\infty \frac{(a)_n (b)_n}{(c)_n n!} x^n,
$$ 
where $(a)_n = a(a+1)\cdots(a+n-1)$ denotes the Pochhammer symbol. 
These are

\medskip\noindent
{\it Hermite polynomials $H_n$}:  
$$
  H_n(x) = (-1)^n e^{x^2} \Big(\frac{\d}{\d x}\Big)^n e^{-x^2} = \sum_{j=0}^{\lfloor \f{n}2 \rfloor} \frac{n!}{j! (n-2j)!} (2x)^{n-2j},
$$ 
orthogonal with respect to $w(x) = e^{-x^2}$ on $\RR$ with $c_w = \pi^{-\f12}$ and $h_n(w) = 2^n n!$. 

\medskip\noindent{\it Laguerre polynomials $L_n^\a$}: 
$$
  L_n^\a (x) =  \frac{(\a+1)_n}{n!} \sum_{j=0}^n \frac{(-n)_j}{j! (\a+1)!} x^j =  {}_1 F_1\left(\begin{matrix} -n \\ \a+1 \end{matrix}; x\right),
$$ 
orthogonal with respect to $w(x) = x^\a e^{-x}$, $\a > -1$, on $\RR_+ := [0,\infty)$ with $c_w = \frac{1}{\Gamma(\a+1)}$
and $h_n(w) = (\a+1)_n /n!$. 

\medskip\noindent
{\it Gegenbauer polynomials $C_n^\l$}: 
$$
 C_n^\l(x) = \frac{(\l)_n2^n}{n!} {}_2 F_1\left(\begin{matrix} - \frac{n}{2}, \frac{1-n}{2} \\ 1-n-\l \end{matrix}; 
       \frac{1}{x^2}\right),
$$
orthogonal with respect to $w(x) = (1-x^2)^{\l-\f12}$, $\l > -\f12$, on $[-1,1]$ with $c_w$ given by 
\begin{equation}\label{eq:c_l}
    c_{\l} := \frac{\Gamma(\l+1)}{\Gamma(\f12)\Gamma(\l+\f12)}
\end{equation}
and $h_n(w) = \frac{\l}{n+\l}C_n^\l(1)$ and $C_n^\l(1) = (2 \l)_n /n!$.
 
\medskip\noindent
{\it Jacobi polynomials $P_n^{(\a,\b)}$}: 
$$
  P_n^{(\a,\b)}(x) = \frac{(\a+1)_n}{n!} {}_2F_1 \left (\begin{matrix} -n, n+\a+\b+1 \\
      \a+1 \end{matrix}; \frac{1-x}{2} \right),
$$
orthogonal with respect to $w(x) = (1-x)^\a(1+x)^b$, $\a,\b > -1$, on $[-1,1]$ with $c_w = 2^{-\a-\b-1} c_{\a,\b}$
and 
\begin{equation}\label{eq:c_ab}
   c_{\a,\b} := \frac{\Gamma(\a+\b+2)}{\Gamma(\a+1)\Gamma(\b+1)}
\end{equation}
and $h_n(w)$ is given in, say, \cite[p. 21]{DX}. 

\medskip

We will need
two more families of orthogonal polynomials, generalized Gegenbauer polynomials orthogonal with respect 
to $|x|^{2\mu} (1-x^2)^{\l-\f12}$ on $[-1,1]$ and generalized Hermite polynomials orthogonal with respect to 
$|x|^{2\mu} e^{-x^2}$ on $\RR$. These polynomials and their properties will be given in Appendix \ref{sect:append}

We reserve the notation $c_\l$ in \eqref{eq:c_l} and $c_{\a,\b}$ in \eqref{eq:c_ab} and will use them in multiple
places throughout this paper. 

\subsection{Spherical harmonics}
Let $\Delta$ be the Laplace operator of $\RR^d$. A polynomial $Y$ of $d$-variables is called harmonic
if $\Delta Y = 0$. Spherical harmonics are homogeneous harmonic polynomials restricted on the unit 
sphere. Let $\CH_n^d$ be the space of spherical harmonics of degree $n$. If $Y \in \CH_n^d$, then 
$Y(x) = \|x\|^n Y(x')$, $x' = x/\|x\| \in \sph$, so that it is completely determined by its restriction on the
unit sphere. It is known that 
\begin{equation} \label{eq:sphHn-dim}
 \dim \CH_n^d =\binom{n+d-1}{n} - \binom{n+d-3}{n-2} = \binom{n+d-2}{n} +\binom{n+d-3}{n-1}.
 \end{equation}   
Spherical harmonics of different degrees are orthogonal on the sphere: If $ n \ne m$, then  
$$
   \frac{1}{\s_d} \int_\sph Y_n (\xi) Y_m (\xi) d\s(\xi) = 0, \qquad Y_n \in \CH_n^d, \quad Y_m \in \CH_m^d, 
$$
where $\s_d$ denotes the surface area $\s_d = 2 \pi^{\f{d}{2}}/\Gamma(\f{d}{2})$ of $\sph$. In spherical
polar coordinates $x = r x'$, $r> 0$ and $x' \in\sph$, the Laplace operator satisfies
$$
  \Delta  = \frac{d}{dr^2} +\frac{d-1}{r}\frac{d}{dr} + \frac{1}{r^2} \Delta_0,
$$
where $\Delta_0$ is the Laplace-Beltrami operator of the unit sphere $\sph$. Spherical harmonics are
eigenfunctions of the latter operator \cite[(1.4.9)]{DaiX}, 
\begin{equation} \label{eq:sph-harmonics}
     \Delta_0 Y(\xi) = -n(n+d-2) Y(\xi), \qquad Y \in \CH_n^d,
\end{equation}
which is the property (I) in the previous section. Let $\{Y_\ell^n: 1 \le \ell \le \dim\CH_n^d\}$ 
be an orthonormal basis of $\CH_n^d$ with respect to the normalized surface measure. Then the
the addition formula of the spherical harmonics, which is the property (II) in the previous section, 
states \cite[(1.2.3) and (1.2.7)]{DaiX},
\begin{equation} \label{eq:additionF}
  \sP_n(x,y):= \sum_{\ell =1}^{\dim \CH_n^d} Y_\ell^n (x) Y_\ell^n(y) = Z_n^{\f{d-2}{2}} (\la x,y\ra), \quad x, y \in \sph, 
\end{equation}
where $Z_n^\l$ is a polynomial of one variable defined by 
\begin{equation} \label{eq:Zn}
   Z_n^\l (t) = \frac{n+\l}{\l} C_n^\l (t), \quad \l > 0, \quad\hbox{and}\quad Z_n^0(t) := \begin{cases} 2 T_n(t), & n \ge 1 \\ 1, & n =0 \end{cases}
\end{equation}
with $C_n^\l$ being the Gegenbauer polynomial and $T_n$ being the Chebyshev polynomial of the first kind.
The function $\sP_n$ is the reproducing kernel of $\CH_n^d$ and  the kernel for the orthogonal projection operator $\proj_n: L^2(\sph) \mapsto \CH_n^d$, 
$$
   \proj_n f(x) = \frac{1}{\s_d}\int_{\sph} f(y) \sP_n(x,y) \d\s(y).
$$
In particular, it is independent of the choices of orthonormal basis of $\CH_n^d$. For $f \in L^2(\sph)$, its 
Fourier orthogonal series in spherical harmonics is defined by 
$$
f = \sum_{n=0}^\infty \proj_n f = \sum_{n=0}^\infty \sum_{\ell=0}^{ \dim \CH_n^d} \wh f_\ell^n Y_\ell^n,    
   \qquad  \wh f_\ell^n = \frac{1}{\s_d} \int_\sph f(y) Y_\ell^n(y) \d \s. 
$$
Thus, the addition formula \eqref{eq:additionF} shows that the kernel possesses a one-dimensional 
structure and satisfies a closed-form formula. Consequently, it plays a fundamentally role for the Fourier
analysis on the sphere. 

\subsection{Orthogonal polynomials on the unit ball} 
On the unit ball $\BB^d$ of $\RR^d$, let $\varpi_\mu$ denote the weight function 
\begin{equation}\label{eq:weightB}
 \varpi_\mu(x):= (1-\|x\|^2)^{\mu-\f12},  \quad \mu > -\tfrac12. 
\end{equation}
Classical orthogonal polynomials on the unit ball are orthogonal with respect to 
$$
  \la f,g\ra_{\mu} =b_\mu^\BB \int_{\BB^d} f(x) g(x) \varpi_\mu(x) dx, \qquad 
     b_{\mu}^\BB =  \frac{\Gamma(\mu+\f{d+1}{2})}{\pi^{\f{d}{2}}\Gamma(\mu+\f12)},
$$
which is an inner product normalized so that $\la 1,1\ra_\mu =1$. These polynomials are closely 
related to spherical harmonics and also possess the two characteristic properties.

Let $\CV_n^d(\varpi_\mu)$ be the space of orthogonal polynomials of degree $n$ with respect to 
$\varpi_\mu$. It is well--known that 
$$
  \dim \CV_n^d(\varpi_\mu) = \binom{n+d-1}{n}. 
$$
Several explicit orthogonal bases of $\CV_n^d(\varpi)$ can be explicitly given in terms of classical orthogonal
polynomials of one variable; see \cite[Chapter 5]{DX}. A basis of $\CV_n^d(\varpi_\mu)$ can be conveniently
indexed by $\{P_\kb^n: |\kb| = n,\, \kb \in \NN_0^d\}$. The classical orthogonal polynomials are eigenfunctions
of a second order linear differential operator, 
\begin{equation}\label{eq:diffBall}
  \left( \Delta  - \la x,\nabla \ra^2  - (2\mu+d-1) \la x ,\nabla \ra \right)u = - n(n+2\mu+ d-1) u, \quad u \in \CV_n^d(\varpi_\mu), 
\end{equation}
which is the analog of the property (I) for spherical harmonics. Furthermore, let $\{P_{\kb}^n: |\kb| =n\}$ be an 
orthonormal basis of $\CV_n^d(\varpi_\mu)$; then an analog of the property II, the addition formula, holds for 
$\mu \ge 0$ in the form of \cite{X99} 
\begin{align}\label{eq:PnBall}
 & \Pb_n(\varpi_\mu;x,y) :=   \sum_{|\kb| = n}P_{\kb}^n(x)P_{\kb}^n(y) \\
   &  \qquad = c_{\mu-\f12} \int_{-1}^1 Z_n^{\mu+\f{d-1}{2}}  \left(\la x,y\ra+ t \sqrt{1-\|x\|^2}\sqrt{1-\|y\|^2} \right) 
        (1-t^2)^{\mu-1} \d t \notag
\end{align}
where $c_{\mu-\f12}$ is defined in \eqref{eq:c_l} and the identity holds when $\mu =0$ under the limit 
\begin{equation}\label{eq:limit-int}
  \lim_{\mu \to 0}  c_{\mu-\f12} \int_{-1}^1 f(t) (1-t^2)^{\mu-1} dt = \frac{f(1) + f(-1)}{2}.  
\end{equation}

As in the case of the spherical harmonics, the kernel $\Pb_n$ is the reproducing kernel of $\CV_n^d(\varpi_\mu)$ 
and the kernel of the projection operator $\proj_n^\mu: L^2(\BB^d,\varpi_\mu) \mapsto \CV_n^d(\varpi_\mu)$, 
$$
   \proj^\mu_n f(x) =  b_\mu^{\BB} \int_{\BB^d} f(y) \Pb_n(\varpi_\mu; x,y) \varpi_\mu (y) \d y.
$$
For $f \in L^2(\BB^d,\varpi_\mu)$, its Fourier orthogonal series is defined by 
$$
f = \sum_{n=0}^\infty \proj_n^\mu f = \sum_{n=0}^\infty \sum_{|\kb|=n} \wh f_\kb^n P_\kb^n,    
     \qquad  \wh f_\kb^n = b_\mu^\BB \int_{\BB^d} f(y) P_\kb^n(y)  \varpi_\mu (y) \d y.
$$
Again, the addition formula \eqref{eq:PnBall} shows that the kernel possess a one-dimensional structure
and satisfies a closed-form formula. Likewise, it plays a fundamentally role for the Fourier analysis on 
the ball as we mentioned in the introduction. 
 
\section{Orthogonal structure in and on a hyperboloid}
\setcounter{equation}{0}

The orthogonal structures on the surface of a hyperboloid will be consider in the first subsection and the
structure in the solid hyperboloid will be considered in the second subsection.  

\subsection{On the surface of a hyperboloid}\label{sec:sfOP} 
Let $\varrho \ge 0$, $c > 0$ and $b > \varrho$. We consider the surface of revolution 
$$
{}_\varrho\VV_0^{d+1}: =\left \{(x,t): \|x\|^2 = c^2 ( t^2- \varrho^2), \,\,  \varrho \le |t| \le b, \,\, x \in \RR^d \right \},
$$ 
where $b$ is either finite or $\infty$. When $\varrho > 0$, this is a double hyperboloid of revolution. 
When $\varrho = 0$, it degenerates to the double cone
$$
 \VV_0^{d+1}: = {}_0 \VV_0^{d+1} = \left \{(x,t): \|x\| = c |t|, \,\,    |t| \le b, \,\, x \in \RR^d \right \}.
$$
To make notations less overwhelming, we shall adopt the convention of using ${}_\varrho\VV_0^{d+1}$ only
when we want to make a distinction between $\varrho \ne 0$ and $\varrho =0$, and will use $\VV_0^{d+1}$
across the board otherwise. The surface consist of two parts,
$$
   \VV_{0,+}^{d+1} = \{(x,t) \in    \VV_0^{d+1}: t \ge \varrho\} \quad \hbox{and}\quad 
     \VV_{0,-}^{d+1} = \{(x,t) \in   \VV_0^{d+1}: t \le -\varrho\},
$$
which we call the upper and the lower part. It is evident that $\VV_0^{d+1} = \VV_{0,+}^{d+1} \cup \VV_{0,-}^{d+1}$. 

\subsubsection{Orthogonal polynomials} 
Let $w$ be a weight function defined on $(-\infty, \varrho ] \cup [\varrho, + \infty)$ in the real line. 
On $\VV_0^{d+1}$ we consider the inner product 
$$
  \la f,g\ra_{w} = b_{w} \int_{\VV_0^{d+1}} f(x,t) g(x,t) w(t) \d\s(x,t),
$$
where $\d \s$ denotes the surface measure on $\VV_0^{d+1}$. Let $\CV_n(\VV_0^{d+1}, w)$ be 
the space of orthogonal polynomials with respect to the inner product $\la \cdot,\cdot \ra_w$, which is well 
defined on the space of polynomials of two variables modulo the polynomial idea generated by 
$\|x\|^2 - t^2+\varrho^2$. The dimension of this space is the same as that of $\CH_n^{d+1}$ of $\SS^d$, 
$$
  \dim \CV_n(\VV_0^{d+1},w) = \binom{n+d-1}{n} +  \binom{n+d-2}{n-1}.
$$  
The integral on the surface of the hyperboloid can be decomposed as 
\begin{align*}
  \int_{\VV_0^{d+1} } f(x,t) \d \s(x,t) &  = \int_{ \varrho \le |t| \le b} \int_{\|x\| = c \sqrt{t^2-\varrho^2}} f(x,t) \d\s(x,t) \\
     & =  \int_{ \varrho \le |t| \le b} \big (c^2 (t^2-  \varrho^2)\big)^{\f{d-1}{2}}
         \int_{\sph} f\left(c \sqrt{t^2- \varrho^2} \, \xi, t\right) \d\s(\xi)\d t.
\end{align*}
An orthogonal basis of $\CV_n(\VV_0^{d+1}, w)$ can be given in terms of orthogonal polynomials of 
one variable and spherical harmonics. 

\begin{prop}\label{prop:sfOPbasis}
For a fixed $m \in \NN_0$, let $q_n^{(m)}$ be the orthogonal polynomial of degree $n$ with respect to the weight 
function $(t^2-\varrho^2)^{m + \f{d-1}{2}} w(t)$ on $(-\infty, -  \varrho \,] \cup [ \varrho,\infty)$. Let 
$\{Y_\ell^m: 1 \le \ell \le \dim \CH_n^d\}$ denote an orthonormal basis of $\CH_m^d$. We define 
\begin{equation} \label{eq:sfOPbasis}
  \sQ_{m, \ell}^n (x,y) = q_{n-m}^{(m)} (t)  Y_\ell^m \left (\frac{x}{c}\right), 
      \quad 0 \le m \le n, \,\, 1 \le \ell \le \dim \CH_m^d.
\end{equation}
Then $\sQ_{m,\ell}^n$ form an orthogonal basis of $\CV_n(\VV_0^{d+1},w)$. 
\end{prop}

\begin{proof}
It is easy to see that the number of $\sQ_{m,\ell}^n$ is equal to the dimension of $\CV_n(\VV_0^{d+1},w)$. Since
$\sQ_{m,\ell}^n$ are evidently polynomials of degree $n$, it is sufficient to show that they are orthogonal
with respect to $\la\cdot,\cdot\ra_w$. Using $Y_\ell^m \left (\frac{x}{c}\right) = (t^2-\varrho^2)^{\frac{m}{2}} 
Y_\ell^m(\xi)$, $\xi\in \sph$, and the orthogonality of $Y_\ell^m$ on the unit sphere, it follows that
$$
\la \sQ_{m,\ell}^n, \sQ_{m',\ell'}^{n'} \ra = \delta_{m,m'} \delta_{\ell,\ell'}  
    c^{d-1} b_w \s_d \int_{ \varrho  \le |t| \le b} \big ( t^2-  \varrho^2)^{m + \f{d-1}{2}}
      q_{n-m}^{(m)}(t)q_{n'-m}^{(m)}(t)w(t) \d t,
$$
which is zero if $n \ne n'$ by the orthogonality of $q_n$. 
\end{proof}

Without loss of generality we can, and will, assume $c = 1$ from now on. Furthermore, by assuming 
that $w$ is supported on the set $(-b, - \varrho\,] \cup [ \varrho, b)$ we can assume the integral over
$\VV_0^{d+1}$ is over $t \in \RR$.  We make the following observation: 

\begin{prop} \label{prop:parity}
If $w$ is an even function, then $\sQ_{m,\ell}^n$ in \eqref{eq:sfOPbasis} is even in $t$ 
if $n-m$ is even, and odd in $t$ if $n-m$ is odd. 
\end{prop}

\begin{proof}
Since $w$ is even, $q_n^{(m)}$ has the same parity as $n$. Since $Y_\ell^m$ is homogeneous, 
$$
Y_\ell^m(x) = (t^2-\varrho^2)^{\tfrac{m}2} Y_\ell^m(\xi), \qquad  \xi \in \sph, \,
    x = \sqrt{t^2-\varrho^2}\, \xi, 
$$
and the factor $ (t^2-\varrho^2)^{\tfrac{m}2}$ becomes $|t|^{m}$ when $\varrho = 0$, so that it is always
even in $t$. Consequently, $\sQ_{m,\ell}^n$ has the same parity as $n-m$. 
\end{proof}

In particular, the proposition prompts the following definition.

\begin{defn}
Let $w$ be an even weight function. We denote by $\CV_n^E(\VV_0^{d+1},w)$ the subspace of 
$\CV_n(\VV_0^{d+1},w)$ that consists of polynomials even in $t$ variable. Similarly, 
$\CV_n^O(\VV_0^{d+1}, w)$ denotes the subspace that consists of polynomials odd in $t$ variable. 
\end{defn}

In terms of the basis in \eqref{eq:sfOPbasis}, Proposition \ref{prop:parity} implies
\begin{align*}
\CV_n^E(\VV_0^{d+1},w) &\, = \mathrm{span} 
    \left \{\sQ_{n-2k,\ell}^n: 1 \le \ell \le \dim \CH_{n-2k}^d, \, 0 \le k \le \tfrac{n}{2} \right\}, \\ 
 \CV_n^O(\VV_0^{d+1},w) &\, = \mathrm{span} 
    \left \{\sQ_{n-2k-1,\ell}^n: 1 \le \ell \le \dim \CH_{n-2k}^d, \, 0 \le k \le \tfrac{n-1}{2} \right\}. 
\end{align*}
Hence, by Proposition \ref{prop:parity}, we immediately deduce the following corollary. 
 
\begin{cor}
Let $w$ be an even weight function. Then for $n= 1,2, 3, \ldots$,
$$
  \CV_n(\VV_0^{d+1}, w) = \CV_n^E(\VV_0^{d+1}, w) \bigoplus \CV_n^O(\VV_0^{d+1}, w).
$$
Furthermore, 
\begin{equation}\label{eq:dimVnE}
  \dim \CV_n^E(\VV_0^{d+1},w) =\binom{n+d-1}{n}, \quad \quad  
    \dim \CV_n^O(\VV_0^{d+1},w) =\binom{n+d-2}{n-1}. 
\end{equation}
\end{cor}

\begin{proof}
The decomposition of the space is obvious. The dimension of $\CV_n^E(\VV_0^{d+1},w)$
is equal to $\sum_{0 \le k \le \lfloor\frac{n}{2}\rfloor} \dim \CH_{n-2k}^d$, which simplifies by
the first equality in \eqref{eq:sphHn-dim}. 
\end{proof}

Assume now that $w$ is of the form $w(t) = |t| w_0(t^2-\varrho^2)$ for some function $w_0$ defined on $\RR_+$. 
Then $w$ is even. The orthogonal polynomial $q_k^{(m)}$ in Proposition \ref{prop:sfOPbasis} can be deduced with
the help of the following proposition. 

\begin{prop} \label{prop:op-1d}
Let $\rho > 0$ and let $w_0$ be a weight function defined on $[0, \infty)$. The orthogonal polynomials 
$q_n$ with respect to the weight function $|t| w_0(t^2-\varrho^2)$, defined on $(-\infty, \varrho ] \cup [ \varrho, b)$, 
are given by 
\begin{align*}
   q_{2k}(t) = \, & p_k \left(w_0; t^2-\varrho^2\right) \\
   q_{2k+1}(t)= \, & \frac{1}{t} \left[p_{k+1} \left(w_0; t^2-\varrho^2\right)
       p_k \left(w_0;-\varrho^2\right) - p_k\left(w_0; t^2-\varrho^2\right) p_{k+1}\left(w_0;-\varrho^2\right) \right].
\end{align*} 
\end{prop}

\begin{proof}
Let $\la f,g\ra = \int_{|t| \ge \rho} f(t) g(t) |t| w_0(t^2-\varrho^2) \d t$. Because the weight function is even, 
the polynomial $q_{2k}$ must be even and $q_{2k+1}$ must be odd. 
In particular, $\la q_{2k}, q_{2j+1}\ra =0$ for all $k, j \in \NN_0$. Moreover, we can assume 
$q_{2k} (t) = p_k (t^2 -\varrho^2)$ for some polynomial $p_k$ or degree $k$. Then, changing variable 
$s = t^2-\varrho^2$ gives
$$
  \la q_{2k}, q_{2j}\ra = 2 \int_{\varrho}^\infty p_k \left(t^2-\varrho^2\right) p_j  \left(t^2-\varrho^2\right)
      t w_0  \left(t^2-\varrho^2\right)\d t 
      = \int_{0}^\infty p_k (s) p_j (s) w_0(s) \d s. 
$$
Consequently, $p_k = p_k(w_0)$. Furthermore, we can assume $q_{2k+1}(t) = t p_k(t^2-\varrho^2)$ 
for some polynomial $p_k$ of degree $k$. It then follows that 
\begin{align*}
\la q_{2k+1}, q_{2j+1}\ra &\, = 2 \int_\varrho^\infty t^2  p_k  \left(t^2-\varrho^2\right) p_j  \left(t^2-\varrho^2\right)
         t w_0  \left(t^2-\varrho^2\right)\d t \\
    &\,  = \int_{0}^\infty p_k(s) p_j (s)\left (s+\varrho^2 \right) w_0(s) \d s. 
\end{align*}
The orthogonal polynomials with respect to $(\varrho+s)w_0(s)$ can be written in terms of orthogonal 
polynomials $p_n(w_0)$, according to a theorem due to Christoffel \cite[p. 29]{Sz}, which shows that
$$
     p_k(t) =  \frac{ p_{k+1}(w_0; t) p_k(w_0;-\varrho^2) - p_k(w_0; t) p_{k+1}(w_0;-\varrho^2)} {t+\varrho^2}. 
$$
Now, $q_{2k+1}(t) = t p_k(t^2-\varrho^2)$ gives the stated result. This completes the proof. 
\end{proof}

Setting $w(t) = |t| w_0(t^2-\varrho^2)$ in Proposition \ref{prop:sfOPbasis}, the above proposition gives
the following corollary:

\begin{cor} \label{cor:sfOPeven}
Let $w(t) = |t| w_0(t^2-\varrho^2)$ and $w_0^{(m)}(t) = |t|^{m+\f{d-1}{2}} w_0(t)$. Then the orthogonal 
polynomials in \eqref{eq:sfOPbasis} that are even in $t$ variable are given by 
\begin{equation} \label{eq:sfOPeven}
  \sQ_{n-2k,\ell}^{n}(x,t) = p_k \big(w_0^{(n-2k)}; t^2-\varrho^2 \big) Y_\ell^{n-2k}(x), \quad 
      1 \le \ell \le \dim \CH_{n-2k}^d,\,0 \le k \le \tfrac{n}{2}.
\end{equation}
In particular, these polynomials consist of an orthogonal basis of $\CV_n^E(\VV_0^{d+1}, w)$. 
\end{cor}

\subsubsection{Fourier orthogonal series}\label{sec:sfOPseries} 
Let $f \in L^2(\VV_0^{d+1}, w)$. With respect to the orthogonal basis $\{\sQ_{m,\ell}^n\}$, its Fourier 
orthogonal series is defined by 
\begin{equation}\label{eq:Fourier}
   f= \sum_{n=0}^\infty  \sum_{m=0}^n \sum_{\ell =1}^{\dim \CH_m^d}  \wh f_{m,\ell}^n \sQ_{m,\ell}^n,
   \qquad \wh f_{m,\ell}^n = \frac{\la f, \sQ_{m,\ell}^n\ra_w}{\la \sQ_{m,\ell}^n, \sQ_{m,\ell}^n\ra_w}. 
\end{equation}
Let $\sP_n(w)$ be the reproducing kernel of $\CV_n(\VV_0^{d+1},w)$. In terms of the basis $\{\sQ_{m,\ell}^n\}$,
\begin{equation}\label{eq:sfKernel}
\sP_n\big(w; (x,t),(y,s) \big) = \sum_{m=0}^n \sum_{\ell =1}^{\dim \CH_m^d} 
     \frac{\sQ_{m,\ell}^n(x,t) \sQ_{m,\ell}^n(y,s)} {\la \sQ_{m,\ell}^n, \sQ_{m,\ell}^n\ra_w}. 
\end{equation}
The kernel, however, is independent of the choice of orthogonal bases. The orthogonal projection operator
$\proj_n: L^2(\VV_0^{d+1}, w) \mapsto \CV_n(\VV_0^{d+1}, w)$ is defined by 
$$
\proj_n f :=  \sum_{k=0}^n  \sum_{m=0}^k \sum_{\ell =1}^{\dim \CH_m^d}  \wh f_{m,\ell}^k \sQ_{m,\ell}^k
  = b_w \int_{\VV_0^{d+1}} f(y,s) \sP_n (w; \cdot, (y,s))w(s) \d y \d s.
 $$
 
If $w$ is an even weight function on $\RR$, we denote by $\sP^E_n$ and $\sP^O_n$ the reproducing kernels
of $\CV_n^E(\VV_0^{d+1},w)$ and  $\CV_n^O(\VV_0^{d+1},w)$, respectively. These kernels can also be written
as sums in terms of their respective orthogonal bases. For example, we have
\begin{equation}\label{eq:sfKernelE}
\sP_n^E\big(w; (x,t),(y,s) \big) = \sum_{k=0}^{\lfloor \frac{n}{2} \rfloor} \sum_{\ell =1}^{\dim \CH_{n-2k}^d} 
     \frac{\sQ_{n-2k,\ell}^n(x,t) \sQ_{n-2k,\ell}^n(y,s)} {\la \sQ_{n-2k,\ell}^n, \sQ_{n-2k,\ell}^n\ra_w}. 
\end{equation}

\begin{lem} \label{lem:sfPbE+PbO}
Let $w$ be an even weight function on $\RR$. Then the reproducing kernel $\sP_n(w)$ can be decomposed as 
\begin{align}
\sP^E_n \big (w; (x,t),(y,s) \big) &\, = \frac12 \left[ \sP_n\big (w; (x,t),(y,s) \big) + \sP_n\big (w; (x,t),(y, -s) \big) \right],     \label{eq:sfPbE} \\
\sP^O_n \big (w; (x,t),(y,s) \big) &\, = \frac12 \left[ \sP_n\big (w; (x,t),(y,s) \big)-\sP_n\big (w; (x,t),(y, -s) \big) \right].    \label{eq:sfPbO}
\end{align}
\end{lem}

\begin{proof}
Let $\sQ_n$ denote the righthand side of \eqref{eq:sfPbE} in this proof. By \eqref{eq:sfKernel} and Proposition \ref{prop:parity},  $\sP_n(w; (x,-t), (y,s))=\sP_n(w; (x,t), (y, -s))$. Hence, $\sQ_n$ is symmetric in $(x,t)$ and 
$(y,s)$. The reproducing property of $\sP_n(w)$ shows that if $P\in \CV_n^E(\VV_0^{d+1}, w)$, then 
$$
  \int_{\VV_0^{d+1}} \sQ_n\big ((x,t),(y,s) \big) P(y,s) w (s) \d s =  \tfrac12 \big(P(x,t)+ P(x,-t)\big) = P(x,t).
$$
Moreover, for each fixed $(x,t)$, $\sQ_n$ is clearly even in $s$ so that $\sQ_n( (x,t), \cdot)$ 
is an element of $\CV_n^E(\VV_0^{d+1}, w)$. By the symmetry of $\sQ_n$ in its variables, the same also 
holds if we fix $(y,s)$ instead. Consequently, $\sQ_n$ is the reproducing kernel of $\CV_n^E(\VV_0^{d+1}, w)$; 
that is, $\sQ_n = \sP^E (w)$. A similar proof works for $\sP_n^O (w)$.
\end{proof}
 
Assume that $w$ is an even weight function. If $f(x,t)$ is even in the $t$ variable, then its Fourier orthogonal 
series in \eqref{eq:Fourier} contains only orthogonal polynomials in $\CV_n^E(\VV_0^{d+1},w)$. Alternatively,
if we are given a function $f$ defined on the upper hyperboloid, we could extend it to the double hyperboloid 
by defining $f(x,-t) = f(x,t)$ so that the extended $f$ is even in $t$, which allows us to use the Fourier 
orthogonal expansion that contains only orthogonal polynomials even in $t$ variable. This is formulated 
in the next theorem. 

\begin{thm} \label{pop:FourierEven} 
Let $w(t) = t w_0(t^2-\varrho^2)$ on $\RR_+$. If $f \in L^2(\VV_{0,+}^{d+1}, w)$, then
\begin{equation}\label{eq:FourierEven}
 f= \sum_{n=0}^\infty  \sum_{k=0}^{\lfloor\frac{n}{2}\rfloor} \sum_{\ell =1}^{\dim \CH_{n-2k}^d} 
      \wh f_{n-2k,\ell}^n \sQ_{n-2k,\ell}^n
\end{equation}
in $L^2$ sense, where $\sQ_{n-2k,\ell}^n$ are defined in \eqref{eq:sfOPeven} and $ \wh f_{n-2k,\ell}^n$
are defined in \eqref{eq:Fourier} for $f$ evenly extended in $t$ variable. 
\end{thm}

\begin{proof}
For $f$ defined on the upper surface $\VV_{0,+}^{d+1}$, we extend it to $\VV_0^{d+1}$ so that 
$f$ is even in $t$-variable; that is, considering $f(x,|t|)$ defined on $\VV_0^{d+1}$. By symmetry,
it follows readily that 
$$
  \int_{\VV_0^{d+1}} |f(x,|t|)|^2 w(t) \d\s(x,t) = 2  \int_{\VV_{0,+}^{d+1}} |f(x,t)|^2 w(t) \d\s(x,t).  
$$
Since $\sQ_{m,\ell}^n$ is odd in $t$ when $n-m$ is an odd integer, we see that 
$$
      \la f, \sQ_{m,\ell}^n\ra_w =0, \qquad  n - m = \mathrm{odd}. 
$$ 
Hence, the Fourier coefficients $\wh f_{m,\ell}^n$ are zero when $n - m =2k+1$, or $m = n-2k-1$. Thus,
the series for $f(x,|t|)$ contains only $\sQ_{n-2k,\ell}^n$ and, in particular, it is of the form \eqref{eq:FourierEven}
on $\VV_0^{d+1}$. The series convergence to $f(x,|t|)$ in $L^2(\VV_0^{d+1}, w)$ and converges, in particular, 
to $f \in L^2(\VV_{0,+}^{d+1}, w)$ by symmetry. 
\end{proof}

A couple of remarks are in order. To be more specific, our remarks below address only the Fourier orthogonal series 
on the cone, although the essence applies to that of hyperboloid as well. The proposition shows that, for 
a given function defined on the upper cone, we could expand it in terms of half as many orthogonal 
polynomials on the surface of the double cone. This should be compared with the Fourier orthogonal 
series on the upper cone studied in \cite{X19}, which uses all orthogonal polynomials on the upper cone 
$\VV_{0,+}^d$ defined with respect to the inner product 
$$
  \la f,g\ra = \int_{\VV_{0,+}^{d+1}} f(x,t) g(x,t) w(t) \d \s(x,t). 
$$
This is of course not surprising, the same phenomenon already appears in the classical Fourier series
when one expands an even function in the Fourier cosine series. 
It should be noted, however, that the orthogonal structure on the two series are different. This is best 
illustrated by considering the classical weight functions for which the orthogonal basis can be given in terms 
of classical orthogonal polynomials and spherical harmonics. For example, in \cite{X19}, such a basis 
is constructed for the Laguerre weight $w(t) = t^\a e^{-t}$, but not for the Hermite weight $w(t) = e^{-t^2}$. 
In the next section, we shall construct an orthogonal basis for the Hermite weight on the double cone.  

\subsection{On a solid hyperboloid}
Let $\varrho$ be a real number, $c > 0$ and $b > \varrho$. We consider the solid hyperboloid 
$$
{}_\varrho\VV^{d+1}: =\left \{(x,t): \|x\|^2 \le c^2 ( t^2- \varrho^2), \,\,  \varrho \le |t| \le b, \,\, x \in \RR^d \right \},
$$ 
where $b$ is either finite or $\infty$, which is bounded by the hyperboloid surface $\VV_0^{d+1}$ and, when
$b$ is finite, by the hyperplanes $t = b$ and $t = -b$. When $\varrho = 0$, it is degenerate to the solid double cone
$$ 
   \VV^{d+1} = \left\{(x,t): \|x\| \le c\, t, \, x \in \RR^d, \, t \in [-b, b]\right\}.
$$
We again write $ \VV^{d+1} = {}_\varrho\VV^{d+1}$ unless when we want to emphasis $\varrho > 0$, and
we define upper and lower parts of these domains as 
$$
\VV_{+}^{d+1} = \left \{(x,t) \in \VV_0^{d+1}: t \ge \varrho \right\} \quad \hbox{and}\quad 
\VV_{-}^{d+1}  = \left \{(x,t) \in \VV_0^{d+1}: t \le -\varrho \right\}.
$$
It is evident that $\VV^{d+1} = \VV_{+}^{d+1} \cup \VV_{-}^{d+1}$. The integral on the domain $\VV^{d+1}$
can be decomposed as 
\begin{align*}
  \int_{\VV^{d+1} } f(x,t) \d x \d t &  = \int_{\varrho \le |t| \le b} \int_{\|x\| \le c \sqrt{t^2-\varrho^2}} f(x,t) \d\s(x,t) \\
     & =  \int_{\varrho \le |t| \le b} \big (c^2 (t^2-  \varrho^2)\big)^{\f{d}{2}}
         \int_{\BB^d} f\left(c \sqrt{t^2- \varrho^2} \, y, t\right) \d y \d t.
\end{align*}

\subsubsection{Orthogonal polynomials}
Let $w(t)$ be a weight function on the real line. For $\mu > -\f12$, let  
$$
   W(x,t) = w(t) \left( c (t^2-\varrho^2) - \|x\|^2 \right)^{\mu-\f12}, \qquad (x,t) \in \VV^{d+1}, 
$$
and define the inner product 
$$
    \la f,g\ra_{W} = b_{W} \int_{\VV^{d+1}} f(x,t) g(x,t) W(x,t) \d x \d t.
$$
Let $\CV_n(\VV^{d+1}, W)$ be the space of orthogonal polynomials with respect to the inner
product $\la \cdot,\cdot \ra_W$. It follows from the general theory of orthogonal polynomials that 
$$
        \dim \CV_n(\VV^{d+1},W) = \binom{n+d}{n}.
$$  

Recall that the classical weight function on the unit ball $\BB^d$ is defined by $\varpi_\mu(x) = (1-\|x\|^2)^{\mu-\f12}$.
The weight function $W_\mu$ can be written as 
\begin{equation}\label{eq:W(x,t)}
 W(x,t) = w(t) \big(c(t^2-\varrho^2)\big)^{\mu-\f12} \varpi_\mu(y), \qquad y = \frac{x}{\sqrt{ c(t^2-\varrho^2)}} \in \BB^d.
\end{equation}
An orthogonal basis of $\CV_n(\VV^{d+1}, W)$ can be given in terms of orthogonal polynomials of 
one variable and orthogonal polynomials on the unit ball.  

\begin{prop}\label{prop:solidOPbasis}
For a fixed $m \in \NN_0$, let $q_n^{(m)}$ be the orthogonal polynomial of degree $n$ with respect to the weight 
function $(t^2-\varrho^2)^{m +\mu+ \frac{d-1}{2}} w(t)$ defined on $(-b, -\varrho] \cup [\varrho,b)$. Let 
$\{P_{\kb}^m: |\kb| = m\}$ denote an orthonormal basis of $\CV_m^d(\varpi_\mu)$. We define
\begin{equation} \label{eq:solidOPbasis}
  \Qb_{m, \kb}^n (x,t) = q_{n-m}^{(m)} (t) \big( c(t^2-\varrho^2) \big)^{\frac{m}{2}}
           P_\kb^m \bigg (\frac{x}{\sqrt{ c(t^2-\varrho^2)} }\bigg).
\end{equation}
Then $\{\Qb_{m,\kb}^n: 0 \le m \le n, \, |\kb| = m, \, \kb \in \NN_0^d\}$ is an orthogonal basis of $\CV_n(\VV^{d+1},W)$. 
\end{prop}

\begin{proof}
Using the orthonormality of $P_{\kb}^m$, it follows readily that 
$$
 \la  \Qb_{m, \kb}^n,  \Qb_{m', \kb'}^{n'} \ra_W = \delta_{\kb,\kb'}\delta_{m,m'}
   \frac{b_W}{b_\mu^\BB} \int_{\varrho \le |t| \le b} 
 q_{n-m}^{(m)} (t) q_{n'-m}^{(m)} (t) \big( c(t^2-\varrho^2) \big)^{m+ \mu+\frac{d-1}{2}} w(t) \d t    
$$ 
from which the orthogonality follows form that of $q_{n-m}^{(m)}$. 
\end{proof}
Without loss of generality, we can and will assume $c = 1$. We also absorb $b$ in the support of $w$ to assume
$b = +\infty$. If $w$ is even, then $q_n^{(m)}$ has the same parity as $n$. This leads to the following observation: 

\begin{prop} \label{prop:paritySolid}
If $w$ in \eqref{eq:W(x,t)} is an even function on $\RR$, then $\Qb_{m,\kb}^n$ in \eqref{eq:solidOPbasis} 
is even in $t$ if $n-m$ is even, and odd in $t$ if $n-m$ is odd. 
\end{prop}

Analogous to the surface of hyperboloid, we give the following definition. 

\begin{defn}
Let $W$ be an even weight function in $t$. We denote by $\CV_n^E(\VV^{d+1},W)$ the subspace of 
$\CV_n(\VV^{d+1},W)$ that consists of polynomials even in $t$ variable. Similarly, 
$\CV_n^O(\VV^{d+1}, W)$ dentoes the subspace that consists of polynomials odd in $t$ variable. 
\end{defn}

In terms of the basis in \eqref{eq:solidOPbasis},   Proposition \ref{prop:paritySolid} implies
\begin{align*}
\CV_n^E(\VV^{d+1},W) &\, = \mathrm{span} 
    \left \{\Qb_{n-2k,\kb}^n: |\kb| = n-2k, \, 0 \le k \le \tfrac{n}{2} \right\}, \\ 
 \CV_n^O(\VV^{d+1},W) &\, = \mathrm{span} 
    \left \{\Qb_{n-2k-1,\kb}^n: |\kb| = n-2k-1, \, 0 \le k \le \tfrac{n-1}{2} \right\}. 
\end{align*}
Hence, by Proposition \ref{prop:paritySolid}, we immediately deduce the following corollary. 
 
\begin{cor}
Let $W$ be an even weight function in $t$. Then for $n= 1,2, 3, \ldots$,
$$
  \CV_n(\VV^{d+1}, W) = \CV_n^E(\VV^{d+1}, W) \bigoplus \CV_n^O(\VV^{d+1}, W).
$$
\end{cor}

For solid domains, the dimensions of the these spaces do not simplify. For example, 
\begin{equation*}
  \dim \CV_n^E(\VV^{d+1},W) = 
     \sum_{k=0}^{\lfloor \frac{n}{2} \rfloor} \dim \CV_{n-2k}^d(\varpi_\mu) = 
      \sum_{k=0}^{\lfloor \frac{n}{2} \rfloor} \binom{n-2k+d-1}{n-2k}.
\end{equation*}

Assume that $w$ is of the form $w(t) = |t| w_0(t^2-\varrho^2)$ for some function $w_0$ defined on $\RR_+$. 
By Proposition \ref{prop:op-1d}, the basis in $\CV_n^E(\VV^{d+1},W)$ can be written as the following: 

\begin{cor} \label{cor:solidOPeven}
Let $w(t) = |t| w_0(t^2-\varrho^2)$ in \eqref{eq:W(x,t)} and let $w_0^{(m)}(t) = |t|^{m+\mu+ \f{d-1}{2}} w_0(t)$. 
Then the orthogonal polynomials in \eqref{eq:solidOPbasis} that are even in $t$ variable are given by 
\begin{equation} \label{eq:solidOPeven}
  \Qb_{n-2k,\kb}^{n}(x,t) = p_k \big(w_0^{(n-2k)}; t^2-\varrho^2 \big) \big( t^2-\varrho^2\big)^{\frac{n-2k}{2}}
      P_\kb^{n-2k} \bigg (\frac{x}{\sqrt{t^2-\varrho^2} }\bigg), 
\end{equation}
where $|\kb| = n-2k$, $0\le k \le n/2$.  In particular, these polynomials consist of an orthogonal basis 
of $\CV_n^E(\VV^{d+1},W)$.
\end{cor}

\subsubsection{Fourier orthogonal series}
Let $f \in L^2(\VV^{d+1}, W)$. With respect to the orthogonal basis $\{\Qb_{m,\kb}^n\}$, its Fourier 
orthogonal series is defined by 
\begin{equation}\label{eq:FourierSolid}
   f= \sum_{n=0}^\infty  \sum_{m=0}^n \sum_{|\kb| =m} \wh f_{m,\kb}^n \Qb_{m,\kb}^n,
   \qquad \wh f_{m,\kb}^n = \frac{\la f, \Qb_{m,\kb}^n\ra_W}{\la \Qb_{m,\kb}^n, \Qb_{m,\kb}^n\ra_W}. 
\end{equation}
Let $\Pb_n(W)$ be the reproducing kernel of $\CV_n(\VV^{d+1},W)$. In terms of the basis $\{\Qb_{m,\kb}^n\}$,
\begin{equation}\label{eq:Kernel}
\Pb_n\big(W; (x,t),(y,s) \big) = \sum_{m=0}^n \sum_{|\kb| = m} 
     \frac{\Qb_{m,\kb}^n(x,t) \Qb_{m,\kb}^n(y,s)} {\la \Qb_{m,\kb}^n, \Qb_{m,\kb}^n\ra_W}. 
\end{equation}
The kernel is independent the choice of orthogonal basis. The orthogonal projection operator
$\proj_n: L^2(\VV^{d+1}, W) \mapsto \CV_n(\VV^{d+1}, W)$ is defined by 
$$
\proj_n f :=  \sum_{k=0}^n  \sum_{m=0}^k \sum_{|\kb| =m} \wh f_{m,\kb}^n \Qb_{m,\kb}^n
       = b_W \int_{\VV^{d+1}} f(y,s) \Pb_n \big(W; \cdot, (y,s)\big) W(y,s) \d y \d s.
 $$
 
If $W$ is an even weight function in $t$, we denote by $\Pb^E_n$ and $\Pb^O_n$ the reproducing kernels
of $\CV_n^E(\VV^{d+1},W)$ and  $\CV_n^O(\VV^{d+1},W)$, respectively. These kernels can also be written
as sums in terms of their respective orthogonal bases. For example,  
\begin{equation}\label{eq:KernelE}
\Pb_n^E\big(W; (x,t),(y,s) \big) = \sum_{k=0}^{\lfloor \frac{n}{2} \rfloor} \sum_{|\kb| = n-2k} 
     \frac{\Qb_{n-2k,\kb}^n(x,t) \Qb_{n-2k,\kb}^n(y,s)} {\la \Qb_{n-2k,\kb}^n, \Qb_{n-2k,\kb}^n\ra_W}. 
\end{equation}
As an analogue of the surface of hyperboloid, these kernels satisfy the following: 

\begin{lem} \label{lem:PbE+PbO}
Let $W$ be an even weight function in $t$ variable. Then the reproducing kernel $\Pb_n(w)$ can be decomposed as 
\begin{align}
\Pb^E_n \big (W; (x,t),(y,s) \big) &\, = \frac12 \left[ \Pb_n\big (W; (x,t),(y,s) \big) + \Pb_n\big (W; (x,t),(y, -s) \big) \right],     \label{eq:PbE} \\
\Pb^O_n \big (W; (x,t),(y,s) \big) &\, = \frac12 \left[ \Pb_n\big (W; (x,t),(y,s) \big)-\Pb_n\big (W; (x,t),(y, -s) \big) \right].    \label{eq:PbO}
\end{align}
\end{lem}

Assume that $W$ is even in $t$ variable. If $f(x,t)$ is even in the $t$ variable, then its Fourier orthogonal 
series in \eqref{eq:Fourier} contains only orthogonal polynomials in $\CV_n^E(\VV^{d+1},W)$. As in the 
case of surface of hyperboloid, we can extend any $f$ defined on the upper hyperboloid to the double 
hyperboloid by defining $f(x,-t) = f(x,t)$, so that $f$ is even in $t$. We can then consider the Fourier orthogonal 
expansion of $f$ by using only orthogonal polynomials even in $t$ variable. 

\begin{thm} \label{pop:FourierEvenSolid}
Let $w(t) = t w_0(t^2-1)$ on $\RR_+$ in \eqref{eq:W(x,t)}. If $f \in L^2(\VV_+^{d+1}, W)$, then
\begin{equation}\label{eq:FourierEvenSolid}
 f=  \sum_{n=0}^\infty  \sum_{k=0}^{\lfloor\frac{n}{2}\rfloor} \sum_{|\kb| = n-2k}  \wh f_{n-2k,\kb}^n 
     \Qb_{n-2k,\kb}^n
\end{equation}
in $L^2$ sense, where $\Qb_{n-2k,\kb}^n$ are defined in \eqref{eq:solidOPeven}.
\end{thm}

\begin{proof}
For $f$ defined on the upper surface $\VV_{+}^{d+1}$, we extend it to $\VV^{d+1}$ evenly in $t$ by
considering $f(x,|t|)$. By symmetry, it follows readily that 
$$
  \int_{\VV^{d+1}} |f(x,|t|)|^2 W(x,t) \d x \d t  = 2  \int_{\VV_{+}^{d+1}} | f(x,t)|^2 W(x,t) \d x \d t.  
$$
Since $\Qb_{m,\ell}^n$ is odd in $t$ when $n-m$ is an odd integer, we see that 
$$
      \la f, \Qb_{m,\ell}^n\ra_W =0, \qquad  n - m = \mathrm{odd}. 
$$ 
The rest of the proof follows exactly as in Proposition \ref{pop:FourierEven}. 
\end{proof}

Our remarks on the Fourier orthogonal series on the surface of the hyperboloid at the end of the previous 
subsection apply equally well for the solid hyperboloid. 
 
\section{Gegenbauer polynomials on a hyperboloid}\label{sec:Gegen}
\setcounter{equation}{0}

In this section we consider compact hyperboloids, for which $0< \varrho\le |t| \le b$ and $b$ is a finite positive
number. Without loss of generality, we assume $b = \sqrt{\varrho^2 + 1}$. We discuss orthogonal polynomials 
on the compact hyperboloid surface $\VV_0^{d+1}$ and the solid $\VV^{d+1}$. In both cases, the hyperboloid 
degenerates to the double cone when $\varrho =0$. Our weight function on the cone contains the Gegenbauer 
weight function $(1-t^2)^{\g-\f12}$ as a multiplicative factor.

\subsection{Gegenbauer polynomials on the surface of a hyperboloid}
We consider orthogonal polynomials on the bounded hyperboloid 
$$
  \VV_0^{d+1} = \left \{(x,t): \|x\|^2 = t^2 - \varrho^2, \, x \in \RR^d, \, \varrho \le |t| \le \sqrt{\varrho^2 +1}\right\},
$$
which is a double hyperboloid when $\varrho > 0$ and a double cone when $\varrho = 0$. We choose
the weight function $w$ as 
\begin{equation}\label{eq:sf6weight}
   w_{\b,\g}(t) = |t| (t^2-\varrho^2)^{\b-\f12}( \varrho^2+1 - t^2)^{\g-\f12},  
         \quad \b, \g > -\tfrac12, \quad \varrho \ge 0,
\end{equation}
defined for $\varrho \le |t| \le \sqrt{\varrho^2 +1}$. Correspondingly, the inner product becomes 
$$
  \la f,g\ra_{w_{\b,\g}} = b_{\b,\g}  \int_{\VV_0^{d+1}} f(x,t) g(x,t)  |t| (t^2-\varrho^2)^{\b-\f12}
       (\varrho^2+1-t^2)^{\g-\f12} \d \s(x,t),
$$
where $b_{\b,\g} = \Gamma(\b+\g+ \frac{d+1}2)/(\Gamma(\b+\frac{d}2) \Gamma(\g+\f12)  \s_d)$ and 
$\s_d$ is the surface area of $\sph$.  The constant is verified using the following identity,
\begin{align} \label{eq:intHyp}
  & \int_{\VV_0^{d+1}} f(t^2-\varrho^2)g(\xi) w_{\b,\g}(x,t) \d \s(x,t) \\
     & \qquad\qquad\qquad = \int_{0}^{1}  f(s)s^{\b+\frac{d-2}{2}} (1-s)^{\g-\f12} \d s \int_{\sph}  g(\xi) \d\s(\xi), \notag
\end{align}
which follows from symmetry, changing variable $t \to s^{\f12}$ and then $s\to s + \varrho^2$. 

\subsubsection{Double cone} We consider the cases $\varrho = 0$ first. In this case, 
$$
      w_{\b,\g}(t) = |t|^{2\b}(1-t^2)^{\g-\f12},  \qquad \b > \tfrac{d+1}{2},\, \g > -\tfrac12, \quad -1 \le t \le 1. 
$$
Since $w_{\b,\g}$ is supported on the set $[-1,1]$, the inner product $\la f,g\ra_w$ is defined on the 
surface of the finite cone 
$$
    \VV_0^{d+1} = \left\{(x,t): \|x\| = |t|, \, x\in \RR^d, \, - 1 \le t \le 1\right \}. 
$$
The polynomials $q_k^{(m)}$ in Proposition \ref{prop:sfOPbasis} are orthogonal with respect to 
$$
      |t|^{2m+d-1} w_{\b,\g}(t) = |t|^{2m +2\b + d-1} (1-t^2)^{\g-\f12}, \qquad - 1 \le t \le 1, 
$$ 
and hence are given by the generalized Gegenbauer polynomials $C_n^{(\g,\a)}$ in the Appendix A. More
precisely, $q_k^{(m)} = C_{n-m}^{(\g,\a)}$ with $\a = m + \b + \frac{d-1}{2}$. The polynomial $C_n^{(\g,\a)}$ 
is given explicitly in \eqref{eq:gGegen} and the square of its norm, denoted by $h_n^{(\g,\a)}$, is given in \eqref{eq:gGegenNorm}. Consequently, the orthogonal polynomials given in \eqref{eq:sfOPbasis} 
are now specialized as follows: 

\begin{prop}\label{prop:sfOPconeG}
Let $\{Y_\ell^m: 1 \le \ell \le \dim \CH_n^d\}$ denote an orthonormal basis of $\CH_m^d$.
Then the polynomials 
\begin{equation}\label{eq:sfOPconeG}
  \sC_{m,\ell}^n (x,t) = C_{n-m}^{(\g, m+\b + \frac{d-1}{2})}(t) Y_\ell^{m}(x), 
  \quad 1 \le \ell \le \dim \CH_m^d, \quad 0 \le m \le n.
\end{equation}
form an orthogonal basis of $\CV_n(\VV_{0}^{d+1}, w_{\b,\g})$. Moreover,  
\begin{equation}\label{eq:sfOPconeGnorm} 
  h_{m,n}^\sC: = \la \sC_{m,\ell}^n, \sC_{m,\ell}^n\ra_{w_{\b,\g}}=  
     \frac{(\b+\tfrac{d}2)_m}{(\b+\g+\tfrac{d+1}2)_m} h_{n-m}^{(\g,m+ \b + \frac{d-1}{2})}.
\end{equation}
\end{prop}

\begin{proof}
The polynomials in \eqref{eq:sfOPconeG} consist of an orthogonal basis by 
Proposition \ref{eq:sfOPbasis}. We compute their norm using 
$$
\int_{\VV_0^{d+1}} f(x,t) |t|^{2\b} (1- t^2)^{\g-\f12} \d x \d t = \int_{-\infty}^\infty t^{d-1} \int_{\sph} 
     f(t\xi, t) |t|^{2\b} (1-t^2)^{\g-\f12} \d \s(\xi) \d t      
$$ 
and $\sC_{m,\ell}^n(x,t) = C_{n-m}^{(\g,m+\b+\frac{d-1}{2})}(t) t^m Y_\ell^m(\xi)$. It follows that 
\begin{align*}
  h_{m,n}^\sC &\, = b_{\b,\g}  \int_{-1}^1  \left|C_{n-m}^{(\g,m+\b+\frac{d-1}{2})}(t)\right|^2 
        t^{2m+2\b+d-1}(1-t)^{\g-\f12} \d t  \int_{\sph} \left|Y_\ell^m(\xi)\right|^2 \d \s(\xi) \\
   & \, = \frac{\Gamma(m+\b+\frac{d}2)\Gamma(\b+\g+\f{d+1}{2})}{\Gamma(\b+\frac{d}2)\Gamma(m+\b+\g+\f{d+1}{2})}
    h_{n-m}^{(\g,m+ \b + \frac{d-1}{2})}
    = \frac{(\b+\tfrac{d}2)_m}{(\b+\g+\tfrac{d+1}2)_m} h_{n-m}^{(\g,m+ \b + \frac{d-1}{2})}, 
\end{align*}
after adjusting the normalization constant of the integral for $h_{n-m}^{(\g,m+\b+\frac{d-1}{2})}$. 
\end{proof}

We shall call these polynomials {\it Gegenbauer polynomials on the cone} when $\b = 0$ and 
{\it generalized Gegenbauer polynomials on the cone} when $\b \ne 0$. 

\begin{rem} If $n- m$ is odd, then the polynomials $\sC_{m,\ell}^n(x,t)$ contains a factor $t$ 
by \eqref{eq:gGegen}. Consequently, the space $\CV_n^O(\VV^{d+1},w_{\b,\g})$ that contains 
these polynomials are well defined for $\b > - \frac{d+1}{2}$. 
In particular, it is well defined for $\b = -1$ for all $d \ge 1$.
\end{rem}

In order to explore if these polynomials are eigenfunctions of a second order differential operators,
we need to consider polynomials that are even in $t$ and those are odd in $t$ separately. Our next 
theorem shows that the polynomials in $\CV_n^E(\VV_0^{d+1}, (1-t^2)^{\g-\f12})$ are 
eigenfunctions of a second order differential operator, whereas the polynomials in 
$\CV_n^O(\VV_0^{d+1}, |t|^{-1} (1-t^2)^{\g-\f12})$ are eigenfunctions of another differential operator.

\begin{thm}\label{thm:sfConeGdiff}
For $n=0,1,2,\ldots$, every $u \in \CV_n^E(\VV_0^{d+1},(1-t^2)^{\g-\f12})$ satisfies the differential 
equation 
\begin{align} \label{eq:sfConeGdiff}
  & \left[ (1-t^2) \partial_t^2 - (2 \g+d)  t \partial_t + \frac{d-1}{t} \partial_t + \frac{1}{t^2} \Delta_0^{(x)} \right] u = - n(n+2\g+d-1)u.
\end{align}
Furthermore,  every $u \in \CV_n^O(\VV_0^{d+1}, |t|^{-1} (1-t^2)^{\g-\f12})$ satisfies the differential equation 
\begin{align} \label{eq:sfConeGdiffO}
   & \left[ (1-t^2) \partial_t^2 - (2 \g +d-2)  t \partial_t + \frac{d-3}{t} \left(\partial_t - \frac{1}{t}\right) 
     + \frac{1}{t^2} \Delta_0^{(x)}   \right] u \\
   & \qquad\qquad\qquad\qquad\qquad\qquad \qquad\qquad\qquad =  - n(n+2\g+d-3)u, \notag
\end{align}
where $\Delta_0^{(x)}$ denotes the Laplace--Beltrami operator acting on the variable $x \in \sph$. 
\end{thm}

\begin{proof}
We work with polynomials in \eqref{eq:sfOPconeG}. The polynomials $\sC_{m,\ell}^n$ with $n-m$ being 
nonnegative even integers consist of a basis for $\CV_n^E(\VV_0^{d+1}, w_{\b,\g})$. Under this assumption, 
the polynomial $g(t) = C_{n-m}^{(\g,m+\b+ \frac{d-1}{2})}(t)$ satisfies, by the identity \eqref{eq:gGegenDiff},  
\begin{align*}
   (1-t^2) g''(t) -&\, (2m+ 2\b+2\g+ d) t g'(t)   +   \frac{2m+ 2\b+ d-1}{t} g'(t) \\
       &\qquad\qquad  = - (n - m) (n+m+2\b+2\g+d-1) g(t). 
\end{align*}
Let $f(t) = t^m g(x)$. Using the differential equation satisfied by $g$, a straightforward computation 
shows that $f$ satisfies 
\begin{align*}
   (1-t^2) &\, f''(t) -   (2\b+2\g+d) t f'(t)  +  \frac{2\b+d-1}{t} f'(t) = m (m+ 2\b+d-2)t^{-2} f(t) \\
        &\quad     -(n-m)(n+m +2\b+2\g+d-1 ) f(t) - m(m+ 2\b+2\g + d-1) f(t)  \\
 & = - n(n +2\b +2\g + d-1) f(t) +   m(m+ 2\b+ d-2) t^{-2} f(t).
\end{align*}
Since $u = \sC_{m,\ell}^n(x,t) = f(t) Y_\ell^m(\xi)$, $\xi \in \sph$ and $Y_\ell^m$ are eigenfunctions of 
$\Delta_0$ with the eigenvalues $-m(m+d-2)$, we conclude that 
\begin{align*}
  & \left[ (1-t^2) \partial_t^2 - (2\b+ 2 \g +d)  t \partial_t + \frac{2\b+d-1}{t} \partial_t + \frac{1}{t^2} \Delta_0^{(x)} \right] u \\
    & =  - n(n+2\b+2\g+d-1)u  + m(m+ 2\b+ d-2) t^{-2} u - m(m+d-2) t^{-2} u.
\end{align*}
The last two terms cancel when $\b =0$, which gives \eqref{eq:sfConeGdiff}. 

Next we consider $\CV_n^O(\VV_0^{d+1}, w_{\b,\g})$, which has a basis given by $\sC_{m,\ell}^n$ with 
$n-m$ being nonnegative odd integers. Using $g$ as above but with $n-m$ odd, then $g$ satisfies the
equation, by \eqref{eq:gGegenDiff},
\begin{align*}
   (1-t^2) g''(t) -&\, (2m+ 2\b+2\g+ d) t g'(t)   +   \frac{2m+ 2\b+ d-1}{t} \left(g'(t) - \frac{g(t)}{t}\right) \\
       &\qquad\qquad  = - (n - m) (n+m+2\b+2\g+d-1) g(t). 
\end{align*}
Using this equation, we see that $f(t) = t^m g(x)$ satisfies, following the proof for $n-m$ being even, 
\begin{align*}
   (1-t^2) f''(t) - & (2\b+2\g+d) t f'(t)  +  \frac{2\b+d-1}{t} \left( f'(t) - \frac{f(t)}{t}\right)\\
 & = - n(n +2\b +2\g + d-1) f(t) +   m(m+ 2\b+ d) t^{-2} f(t).
\end{align*}
Consequently, for $u = \sC_{m,\ell}^n(x,t) = f(t) Y_\ell^m(\xi)$, $\xi \in \sph$, we likewise obtain 
\begin{align*}
  & \left[ (1-t^2) \partial_t^2 - (2\b+ 2 \g +d)  t \partial_t + \frac{2\b+d-1}{t} \left(\partial_t - \frac{1}{t}\right) + \frac{1}{t^2} \Delta_0^{(x)} \right] u \\
    & =  - n(n+2\b+2\g+d-1)u  + m(m+ 2\b+ d) t^{-2} u - m(m+d-2) t^{-2} u,
\end{align*}
which gives \eqref{eq:sfConeGdiffO} when $\b =  -1$. This completes the proof.
\end{proof}
 
 A couple of remarks are in order. First, the term $t^{-2} u $ in \eqref{eq:sfconeHdiffO} comes 
from applying the difference operator defined by 
$$
  D_t u(x,t) = \frac{u(x,t) - u(x, -t)}{t}
$$
on the function that is odd in $t$ variable and this term is zero if $u$ is even in $t$. However,
we cannot combine \eqref{eq:sfConeGdiff} and \eqref{eq:sfConeGdiffO}, since they apply to polynomials
orthogonal with respect to different weight functions. Second, it should be emphasized that \eqref{eq:sfConeGdiff}
holds for the Gegenbuaer polynomials on the cone with $w_0(t) = (1-t^2)^{\g-\f12}$, whereas 
\eqref{eq:sfconeHdiffO} holds for the generalized Gegenbauer polynomials on the cone with 
$w_{-1}(t) =|t|^{-1}(1-t^2)^{\g-\f12} $. While $w_0$ is analytic on the domain, $w_{-1}$ has a singularity
at the origin. The latter one should be compared with the Jacobi polynomials on the upper cone studied
in \cite{X19}; see next subsection.  Finally, the above proof shows that the 
polynomials in $\CV_n^E(\VV_0^{d+1}, |t|^\b (1-t^2)^{\g-\f12})$ also satisfy a differential equation but 
the equation and ``eigenvalues" depend on both $n$ and $m$ and, as a result, it does not imply an 
differential operator that has all polynomials in the space as eigenfunctions. 

\subsubsection{Jacobi polynomials on the upper cone}
These polynomials are studied in \cite{X19} and they are orthogonal on the surface of the upper cone 
with respect to 
$$
  \la f, g\ra_{\b} = b_{\b} \int_{\VV_{0,+}^{d+1}} f(x,t) g(x,t) t^\b (1-t)^\g \d x \d t, \quad \b > -d, \, \g > -1,
$$
where $\VV_{0,+}^{d+1}$ is a compact surface with $0 \le t \le 1$. An orthogonal basis for the space 
$\CV_n(\VV_{+}^{d+1}, t^\b (1-t)^\g)$ is given by 
\begin{equation}\label{eq:coneJsf}
   \sJ_{m,\ell}^n(x,t) =  P_{n-m}^{(2m + \b + d-1,\g)} (1-2t) Y_\ell^m (x), \quad 0 \le m \le n, \,\, 1 \le \ell \le \dim \CH_m^d,
\end{equation}
in terms of the Jacobi polynomials and spherical harmonics  $\{Y_\ell^m\}$ of $\CH_m^d$.

It is shown in \cite{X19} that the Jacobi polynomials in $\CV_n(\VV_{+}^{d+1}, t^{-1}(1-t)^\g)$, with $\b = -1$,
satisfy the differential equation 
\begin{align} \label{eq:diffJsf}
   \left(t(1-t)\partial_t^2 + \big( d-1 - (d+\g)t \big) \partial_t+ t^{-1} \Delta_0^{(x)}\right) u = 
      -n (n+\g+d-1) u.
\end{align}
Note that the weight function $t^{-1} (1-t)^\g$ also has a singularity at the origin. 

Although we can extend the weight function $t^\b (1-t)^\g$ evenly to the double cone $\VV_0^{d+1}$ by 
considering $|t|^\b (1-|t|)^\g$, the Jacobi polynomials in \eqref{eq:coneJsf} are not related to orthogonal 
polynomials for this even weight function on the double cone. In fact, since $P_n^{(\a,b)}$ does not possess 
parity, the orthogonal polynomials with respect to $|t|^\a (1-|t|)^{\g-\f12}$ on $[-1,1]$ are not the 
Jacobi polynomials. 

The space $\CV_n(\VV_{0,+}^{d+1}, t^\b (1-t)^\g)$ has the dimension of 
$\CV_n(\VV_{0,+}^{d+1}, |t|^{2\b} (1-t^2)^{\g-\f12})$, which is nearly twice as that of 
$\CV_n^E(\VV_{0}^{d+1}, |t|^{2\b} (1-t^2)^{\g-\f12})$. This comparison is particularly 
meaningful in view of the Fourier orthogonal series discussed at the end of the Subsection \ref{sec:sfOPseries}. 
The Fourier orthogonal expansions in the Jacobi polynomials on the surface of the upper cone are studied 
in \cite{X19}, and we shall study the Fourier expansions in the generalized Gegenbauer polynomials in 
$\CV_n^E(\VV_0^{d+1}, |t|^{2\b} (1-t^2)^{\g-\f12})$ in the next two sections. 
 
\subsubsection{Gegenbauer polynomials on the hyperboloid} \label{sect:4.1.3}
Here $\rho >0$ and the weight function $w_{\b,\g}$ 
can be written as $w_{\b,\g}(t) = |t| w_0 (t^2-\varrho^2)$ with $w_0(t) = t^{\b-\f12}(1-t)^{\g-\f12}$. The 
polynomials $q_k^{(m)}(t)$ in Proposition \ref{prop:sfOPbasis} are given in terms of $p_k(w_0^{(m)}; t)$ 
with 
$$
    w_0^{(m)}(t) = t^{m+ \f{d-1}{2}}w_0(t) =  t^{m+ \b+ \f{d-2}{2}} (1-t)^{\g-\f12},
$$ 
which are the Jacobi polynomials; that is, $p_k(w_0^{(m)};s)= P_k^{(\g-\f12,m+\b+\f{d-2}{2})}(s)$. 
However, as we have seen in Proposition \ref{prop:op-1d}, $q_{2k}^{(m)}(t) = p_k(w_0^{(m)};t^2-\varrho^2)$ 
but $q_{2k+1}^{(m)}$, which is odd in $t$, is more complicated. Hence, we only consider orthogonal polynomials 
that are even in $t$, that is, those in $\CV_n^E(\VV_0^{d+1}; w_{\b,\g})$.  
 
The orthogonal polynomials given in Corollary \ref{cor:sfOPeven} are now specialized as follows: 

\begin{prop}\label{prop:sfOPhypG}
Let $\{Y_\ell^m: 1 \le \ell \le \dim \CH_n^d\}$ denote an orthonormal basis of $\CH_m^d$.
Then the polynomials 
\begin{equation}\label{eq:sfOPhypG}
 {}_\varrho\sC_{n-2k,\ell}^n (x,t) = P_k^{(\g-\f12,n-2k+\b+\f{d-2}{2})}(2t^2-2\varrho^2-1) Y_\ell^{n-2k}(x) 
\end{equation}
with $1 \le \ell \le \dim \CH_{n-2k}^d$ and $0 \le k\le n/2$, form an orthogonal basis of 
$\CV_n^E(\VV_{0}^{d+1}, w_{\b,\g})$. 
Furthermore, in terms of polynomials $\sC_{n-2k,\ell}^n$ on the cone in \eqref{eq:sfOPconeG} and $\xi \in \sph$, 
\begin{equation}\label{eq:sfOPhypG2}
 {}_\varrho \sC_{n-2k,\ell}^n(x,t) = \frac{(n-2k+\b+ \frac{d-1}{2})_k}{(n-2k+\b+\g+\frac{d-1}{2})_k}\sC_{n-2k,\ell}^n \left(x, \sqrt{t^2-\varrho^2}\right).
\end{equation}
\end{prop}

\begin{proof}
By Corollary \ref{cor:sfOPeven}, we only need to verify the relation \eqref{eq:sfOPhypG2}. 
Using \eqref{eq:gGegen} to write $P_k^{(\g-\f12,\a-\f12)}(2s^2-1)$ as $C_{2k}^{\g,\a}(s)$, we see that
this follows from \eqref{eq:sfOPhypG} and the expression of $\sC_{n-2k,\ell}^n$ in \eqref{eq:sfOPconeG}.
\end{proof}

We shall call these polynomials {\it Gegenbauer polynomials on the hyperboloid} when $\b =0$ and 
{\it generalized Gegenbauer polynomials on the hyperboloid} when $\b > 0$.
 
Our next theorem shows that the Gegenbauer polynomials on the hyperboloid are eigenfunctions of a
differential operator. 

\begin{thm}\label{thm:sfHypGdiff}
Let  $\rho > 0$ and $w_{0,\g} (t) = |t| (t^2-\varrho^2)^{-\f12} (1-t^2)^{\g-\f12}$, $\g > -\f12$. For 
$n=0,1,2,\ldots$, every $u \in \CV_n^E(\VV_0^{d+1}, w_{0,\g})$ satisfies the differential 
equation 
\begin{align} \label{eq:sfHypGdiff}
  & \left[ (1+\varrho^2-t^2) \left(1- \frac{\varrho^2}{t^2} \right) \partial_t^2 + 
      \left ( (1+\varrho^2-t^2) \frac{\rho^2}{t^2} - (2 \g+d) (t^2-\rho^2) \right)\frac{1}{t} \partial_t  \right.\\
   &\qquad\qquad\qquad\qquad\qquad \left.  + \frac{d-1}{t} \partial_t + \frac{1}{t^2- \varrho^2} \Delta_0^{(x)} \right] u = - n(n+2\g+d-1)u.
     \notag
\end{align}
\end{thm}

\begin{proof}
We work with the basis ${}_\varrho\sC_{n-2k,\ell}^n$ of the space ${}_\varrho\CV_n^E(\VV_0^{d+1}, w_{0,\g})$. 
Using the expression \eqref{eq:sfOPhypG2}, we can derive the differential 
equation satisfied by ${}_\varrho \sC_{n-2k,\ell}^n$ from the one satisfied by $\sC_{n-2k,\ell}^n$. We write 
\eqref{eq:sfOPhypG2} as 
$$
  {}_\varrho\sC_{n-2k,\ell}^n(x,t) = c_k f\left(\sqrt{t^2-\varrho^2}\right) Y_\ell^{n-2k}(\xi), 
       \qquad f(t) = P_{2k}^{(\g-\f12,n-2k+\frac{d-2}{2})}(t) t^{n-2k}, 
$$
where $x= \sqrt{t^2-\rho^2} \xi$, $\xi \in \sph$, and $c_k$ is the constant in \eqref{eq:sfOPhypG2}.
Let $F(t) = f\big(\sqrt{t^2-\varrho^2}\big)$. Then
\begin{align}\label{eq:f'F'}
 f'\left(\sqrt{t^2-\varrho^2}\right) = \frac{\sqrt{t^2-\varrho^2}}{t} F'(t), \quad 
 f''\left(\sqrt{t^2-\varrho^2}\right) = \frac{t^2-\varrho^2}{t^2} F''(t)+\frac{\varrho^2}{t^3} F'(t). 
\end{align}
The polynomials $\sC_{n-2k,\ell}^n(t\xi ,t) = f(t) Y_\ell^{n-2k}(\xi)$ satisfies the differential equation 
\eqref{eq:sfConeGdiff}, which shows that 
\begin{align*}
  & \left[(1-T^2) f''(T) - (2\g+d) T f'(T) + \frac{d-1}{T} f'(T) + \frac{1}{T^2} f(T)  \Delta_0^{(x)} \right] 
   Y_\ell^{n-2k}(\xi)\\
   &\qquad\qquad\qquad\qquad \qquad\qquad\qquad\qquad\qquad = - n(n+2\g+d-3) f(T) Y_\ell^m(\xi), 
\end{align*}
where $T= \sqrt{t^2-\varrho^2}$. Replacing the derivative on $f$ by the derivatives on $F$ according to
\eqref{eq:f'F'}, we obtain a differential equation for $F(t) Y_\ell^{n-2k}(\xi) = {}_\varrho\sC_{n-2k,\ell}^n(x,t)$, 
which simplifies to \eqref{eq:sfHypGdiff}. 
\end{proof}

For $\varrho > 0$, there may not be a differential operator that has orthogonal polynomials in 
${}_\varrho\CV_n^O(\VV_0^{d+1}, w_{\b,\g})$ as eigenfunctions for any $\b > 0$. Indeed, if such a 
differential operator exists, it should agree with \eqref{eq:sfConeGdiffO} when $\varrho =0$. However, 
the equation on the double cone holds for $\b = -1$ which is well defined when $\varrho = 0$, but not
for $\varrho > 0$ since we require $\b > -\f12$ when $\varrho  > 0$. 

\subsection{Gegenbauer polynomials on a solid hyperboloid}
On the domain 
$$
  \VV^{d+1} = \left\{(x,t): \|x\|^2 \le t^2 - \varrho^2, \, x \in \RR^d, \, \varrho \le |t| \le \sqrt{\varrho^2 +1}\right\},
$$
bounded by the hyperboloid and the planes $t = \sqrt{\varrho^2 +1}$ and $t=- \sqrt{\varrho^2 +1}$, 
we choose the weight function $w$ as 
$$
  w (t) = |t| (t^2-\varrho^2)^{\b-\f12}( \varrho^2+1 - t^2)^{\g-\f12}, 
         \quad \b, \g > -\tfrac12, 
$$
which is an even function in $t$, so that $W$ in \eqref{eq:W(x,t)} with $c =1$ is given by 
\begin{equation}\label{eq:6Weight}
  W_{\b,\g,\mu}(x, t) = |t| (t^2-\varrho^2)^{\b-\f12}( \varrho^2+1 - t^2)^{\g-\f12}(t^2-\varrho^2 - \|x\|^2)^{\mu-\f12},
\end{equation}
where $\b, \g, \mu > -\tfrac12$, defined for $\varrho \le |t| \le \sqrt{\varrho^2 +1}$. 
The corresponding inner product is defined by
$$
 \la f,g\ra_{\b,\g,\mu} = b_{\b,\mu} \int_{\VV^{d+1}} f(x,t)g(x,t) W_{\b,\g, \mu}(x,t) \d x \d t, 
$$
where the constant $b_{\b,\g,\mu} = 1/\int_{\VV^{d+1}} W_{\b,\g,\mu}(x,t) \d x \d t 
=  b_{\mu}^\BB \frac{\Gamma(\b+\mu+\g+\f{d+1}{2})} {\Gamma(\b+\mu+\f{d}{2}) \Gamma(\g+\f12)}$ with
$b_{\mu}^\BB$ being the normalization constant of $\varpi_\mu$ on $\BB^d$. The constant is computed using
\begin{align} \label{eq:intHypSolid}
  & \int_{\VV^{d+1}} f(t^2-\varrho^2)g\bigg(\frac{x}{\sqrt{t^2-\varrho^2}}\bigg) w_{\b,\g,\mu}(x,t) \d x \d t =    \\
     & \qquad\qquad  = \int_{0}^{1}  f(s)s^{\b+\mu+ \frac{d-2}{2}} (1-s)^{\g-\f12} \d s
              \int_{\BB^d}  g(x) (1-\|x\|^2)^{\mu-\f12} \d x, \notag
\end{align}
which follows from symmetry, changing variable $t \to s^{\f12}$ and then $s\to s + \varrho^2$. 
\subsubsection{Double cone}
In this case $\varrho = 0$ and the weight function becomes 
$$
  W_{\b,\g, \mu}(x,t) = |t|^{2\b}(1-t^2)^{\g-\f12}(t^2-\|x\|^2)^{\mu - \f12}, 
         \quad \b> -\tfrac{d+1}2,\, \g > -\tfrac12, \, \mu  > - \tfrac12.
$$
The polynomial $q_k^{(m)}$ in Proposition \ref{prop:solidOPbasis} are orthogonal with respect to 
$$
     |t|^{2m+ 2\mu +d -1} w_\b(t) =  |t|^{2m+2\b+2 \mu+d-1} (1-t^2)^{\g-\f12},
$$ 
so that they are given by the generalized Gegenbauer polynomials defined in \eqref{eq:gGegen}, that is,
$g_k^{(m)}(t) = C_k^{(\g, m+\b+\mu+ \frac{d-1}{2})}(t)$. Thus, the orthogonal polynomials in 
$\CV_n(\VV^{d+1},W_{\b,\g,\mu})$ given by \eqref{eq:solidOPbasis} are now specialized as follows: 

\begin{prop}\label{prop:solidOPGegen}
Let $\{P_\kb^m: |\kb| = m, \, \kb\in \NN_0^d\}$ be an orthonormal basis of $\CV_m^d(\varpi_\mu)$.
Then the polynomials 
\begin{equation}\label{eq:solidOPconeG}
  \Cb_{m,\kb}^n (x,t) = C_{n-m}^{(\g, m+\b+\mu+\frac{d-1}{2})}(t) t^{m} P_\kb^m\left( \frac{x}{t}\right),  
      \qquad 0 \le m \le n, \quad |\kb| = m,
\end{equation}
form an orthogonal basis of $\CV_n(\VV^{d+1}, W_{\b,\g, \mu})$. Moreover, 
\begin{equation}\label{eq:solidOPconeGnorm}
  h_{m,n}^\Cb: = \la \Cb_{m,\kb}^n, \Cb_{m,\kb}^n\ra_{W_{\b,\g, \mu}} = 
  \frac{(\b+\mu+\tfrac{d}2)_m}{(\b+\mu+\g+\tfrac{d+1}2)_m} 
       h_{n-m}^{(\g, m+\b+\mu+ \frac{d-1}{2})}.
\end{equation}
\end{prop}

\begin{proof}
The polynomials in \eqref{eq:solidOPconeG} consist of an orthogonal basis by 
Proposition \ref{eq:solidOPbasis}. The norm $h_{m,n}^\Cb$ is computed using 
$$
\int_{\VV^{d+1}} f(x,t) W_{\b,\g,\mu}(x,t) \d x \d t = \int_{-\infty}^\infty \int_{\BB^d} 
         f(t y, t) (1-\|y\|^2)^{\mu-\f12} \d y |t|^{2\b} (1-t^2)^{\g-\f12} \d t     
$$ 
and the orthonormality of $P_{\kb}^m$. The detail is similar to the proof of 
Proposition \ref{prop:sfOPconeG}.
\end{proof}
 
We shall call these polynomials {\it Gegenbauer polynomials on the solid cone} when $\b = \f12$ and 
 {\it generalized Gegenbauer polynomials on the solid cone} when $\b > \f12$. The choice of $\b =\f12$
instead of $\b =0$ is due to the differential equation in the next theorem, where we show that the 
elements of $\CV_n^E(\VV^{d+1}, W_{\f12,\g,\mu})$ are eigenfunctions of a second order differential operator, 
whereas the elements of $\CV_n^O(\VV^{d+1}, W_{-\f12,\g,\mu})$ are eigenfunctions of another second 
order differential operator. 
 
\begin{thm}\label{thm:solidConeGdiff}
For $n=0,1,2,\ldots$, every $u \in \CV_n^E(\VV^{d+1}, W_{\f12,\g,\mu})$ satisfies 
the differential equation 
\begin{align} \label{eq:solidConeGdiff}
&   \Big [(1-t^2) \partial_{t}^2 + \Delta_x - \la x,\nabla_x \ra^2 +\frac{2}{t} (1-t^2)\la x, \nabla_x\ra \partial_t 
    + (2\mu+d)\frac{1}{t} \partial_t  \\
   &   \qquad  -  t \partial_t - (2\g+2\mu+d)\left( t \partial_t +  \la x ,\nabla_x\ra \right) \Big] u  
       = -n(n + 2 \g + 2 \mu+d) u. \notag
\end{align}
Furthermore, every $u \in \CV_n^O(\VV^{d+1}, W_{-\f12,\g,\mu})$ satisfies 
the differential equation 
\begin{align} \label{eq:solidConeGdiffO}
&   \Big [(1-t^2) \partial_{t}^2 + \Delta_x - \la x,\nabla_x\ra^2 - \la x,\nabla_x\ra 
       + \frac{2}{t} (1-t^2) \la x,\nabla_x \ra \partial_t \\
   &  \qquad\qquad +  \frac{2\mu+d-2}{t} \Big(\partial_t -\frac{1}{t}\Big)  - 
      (2\g+2\mu+d-1) \big(t \partial_t + \la x ,\nabla_x\ra \big) \Big] u   \notag   \\ 
   &  \qquad\qquad \qquad\qquad \qquad\qquad \qquad  = -n(n + 2 \g + 2 \mu+d) u, \notag
\end{align}
where $\Delta_x$ and $\nabla_x$ indicate that the operators are acting on $x$ variable.
\end{thm}

\begin{proof}
We work with the orthogonal basis \eqref{eq:solidOPconeH}. Let $u = \Cb_{m,\kb}^n$ and we write
$$
  u(x,t) = g(t) H(x,t) \quad \hbox{with} \quad g(t) = C_{n-m}^{(\g,m+\b+\mu+\f{d-1}{2})}(t), \quad H(x,t) 
      = t^m P_{\kb}^m\left(\frac{x}{t}\right). 
$$
For $\CV_n^E(\VV_0^{d+1}, W_{\b,\g,\mu})$,  we assume that $n-m$ are nonnegative even integers. 
As it is shown in the proof of Theorem \ref{thm:sfConeGdiff}, with $\b$ replaced by $\b+\g$, the function 
$g$ satisfies 
\begin{align}\label{eq:partial_g2}
   (1-t^2) g''(t) - (2m+ 2\b+ &\, 2\g + 2\mu + d) t g'(t) + \frac{2m+ 2\b +2\mu+ d-1}{t} g'(t) \\
       &  = - (n - m) (n+m+2\b+2\g+2\mu+d-1) g(t).  \notag
\end{align}
Moreover, using the differential equation \eqref{eq:diffBall} satisfied by $P_\kb(x)$, it is shown in \cite[(3.9)]{X19} 
that $H$ satisfies 
\begin{align}\label{eq:tmOPball}
  \left( t^2 \Delta_x  - \la x, \nabla_x\ra^2  - (2 \mu +d-1) \la x,\nabla_x\ra \right) H 
      = -m (m+2\mu+d-1) H.
\end{align}
Furthermore, the function $H$ is homogeneous in variables $(x,t)$, so that it satisfies, as shown in \cite[(3.8)]{X19}, 
\begin{equation} \label{eq:partialH}
   \frac{\partial}{\partial t} H(x,t) =  \frac{1}{t} ( m   -  \la x, \nabla_x\ra )H(x,t) , 
\end{equation}
from which we also deduce that 
\begin{equation} \label{eq:partialH2}
  \frac{\partial^2}{\partial t^2} H(x,t) = \frac{1}{t^2} \big(  m (m-1) - (2m-1) \la x,\nabla_x \ra + \la x,\nabla_x \ra^2 \big)H(x,t).
\end{equation}
Using these identities to replace $\partial_t H$ and $\partial_t^2 H$, a tedious computation shows that, 
\begin{align*}
 &  \left ((1-t^2) \partial_{t}^2 + (1-t^2)\Delta_x + 2\Big(\frac{1}{t} -t\Big )\la x,\nabla_x \ra \partial_t  \right. \\
   & \left. \qquad\qquad\qquad \qquad\qquad+ (2\b+2\mu+d-1)\Big(\frac{1}{t} -t\Big )\partial_t - (2\g+1) 
         \big( t \partial_t + \la x ,\nabla_x\ra \big) \right) u  \\
 & = \left((1-t^2) g''(t) - (2m+2\b+2\g+2\mu + d)  t g'(t) + \frac{2m + 2\b+ 2\mu + d-1}{t} g'(t)\right)H \\
      &\quad+ \Big(\frac{g(t)}{t^2} - g(x)\Big)\left(t^2 \Delta_x - \la x,\nabla_x\ra^2 - (2\b+ d-1) \la x,\nabla_x\ra
        + m(m+  2\mu + d-1) \right)H \\
      &\quad + (2 \b-1) \frac{g(t)}{t^2} \big(- \la x,\nabla_x\ra + m\big)H + g(x) \big( (2\b-1)\la x,\nabla_x\ra
        + (2\b + 2 \g)m \big)H \\
 & = -(n-m)(n+m+2\b+2\g+2\mu +d-1) u \\
     &\qquad\qquad  +  (2 \b-1)\Big(1- \frac{1}{t^2}\Big) \la x,\nabla_x\ra u 
     + \Big( \frac{2\b-1}{t^2}- (2\b + 2 \g)\Big) m \, u \\
 & =  -n(n +2\b+ 2 \g + 2 \mu+d-1) u + m (m + 2 \mu + d - 1) u \\
    &\qquad\qquad \qquad\qquad \qquad\qquad   + (2 \b-1)\Big(1- \frac{1}{t^2}\Big) \la x,\nabla_x\ra u  
        +  \frac{2\b-1}{t^2}  m \, u,
\end{align*}
where we have used \eqref{eq:partial_g2} and \eqref{eq:tmOPball} in the second step. In particular, 
when $\b = \f12$, the righthand side becomes $-n(n + 2 \g + 2 \mu+d) u + m (m + 2 \mu + d - 1) u$,
in which the second term is $-1$ times of the eigenvalue in \eqref{eq:tmOPball}. Hence, adding the 
above identity with \eqref{eq:tmOPball} and simplifying the result, we conclude that the Gegenbauer 
polynomials on the solid cone satisfy 
\begin{align*}
&   \Big ((1-t^2) \partial_{t}^2 + \Delta_x - \la x,\nabla\ra^2 + 2\Big(\frac{1}{t} -t\Big)\la x,\nabla_x \ra \partial_t + \frac{2\mu+d}{t} \partial_t  \\
   &  \qquad\qquad   - (2\g+2\mu+d+1) t \partial_t - (2\g+2\mu+d) \la x ,\nabla_x\ra  \Big) u  
       = -n(n + 2 \g + 2 \mu+d) u, \notag
\end{align*}
which become \eqref{eq:solidConeGdiff} after rearranging terms.

Next we consider $\CV_n^O(\VV^{d+1}, W_{\b,\g,\mu})$, for which we assume that $n-m$ are nonnegative 
odd integers. As it is shown in the proof of Theorem \ref{thm:solidConeHdiff}, with $\b$ replaced by $\b+\g$, the 
function now satisfies 
\begin{align*}
   (1-t^2) g''(t) - (2m+ 2\b+ &\, 2\g + 2\mu + d) t g'(t) + \frac{2m+ 2\b +2\mu+ d-1}{t} \left(g'(t) - \frac{g(t)}{t} \right)\\
       &  = - (n - m) (n+m+2\b+2\g+2\mu+d-1) g(t).  
\end{align*}
Using this identity instead of \eqref{eq:partial_g2}, the proof for the even case shows that 
\begin{align*}
 &  \left ((1-t^2) \partial_{t}^2 + (1-t^2)\Delta_x + 2\left(\frac{1}{t} -t\right )\la x,\nabla_x \ra \partial_t  \right. \\
   & \left. \qquad\qquad\qquad \qquad\qquad+ (2\mu+d) \left(\frac{1}{t} -t\right)\partial_t - (2\g+1) 
         \big( t \partial_t + \la x ,\nabla_x\ra \big) \right) u  \\
 & =  -n(n +2\b+ 2 \g + 2 \mu+d-1) u + m (m + 2 \mu + d - 1) u \\
    &\qquad + \frac{2m +2\b+2\mu +d-1}{t^2} u + (2 \b-1)\Big(1- \frac{1}{t^2}\Big) \la x,\nabla_x\ra u  +  
      \frac{2\b-1}{t^2}  m \, u.
\end{align*}
If $\b = - \f12$, then the second line of the righthand side becomes independent of $m$. Hence, 
adding the resulted equation and \eqref{eq:tmOPball}, we conclude that 
\begin{align*}
&   \Big ((1-t^2) \partial_{t}^2 + \Delta_x - \la x,\nabla\ra^2 - \la x,\nabla\ra 
       + 2\Big(\frac{1}{t} -t\Big)\la x,\nabla_x \ra \partial_t + 
       \frac{2\mu+d-2}{t} \Big(\partial_t -\frac{1}{t}\Big) \\
   &  \qquad\qquad   - (2\g+2\mu+d-1) \big(t \partial_t + \la x ,\nabla_x\ra \big) \Big) u  
       = -n(n + 2 \g + 2 \mu+d) u, \notag
\end{align*}
from which \eqref{eq:solidConeGdiffO} follows. The proof is completed. 
\end{proof}

The remarks that we made just below Theorem \ref{thm:sfConeGdiff} apply to the solid cone as well. 
In particular, we can consider our results with the Jacobi polynomials on the upper cone, which
are discussed in the next subsection. 

\subsubsection{Jacobi polynomials on the solid upper cone}
These polynomials are studied in \cite{X19} and they are orthogonal on the solid upper cone 
with respect to 
$$
  \la f, g\ra_{\b} = b_{\b} \int_{\VV_+^{d+1}} f(x,t) g(x,t)   W_{\mu,\b,\g}^J(x,t)\d x \d t, 
$$
where $W^J_{\b,\g,\mu}(x,t): = (t^2-\|x\|^2)^{\mu-\f12} t^\b (1-t)^\g$, $\mu > -\tfrac12, \, \b> -1,\, \g > -1$ 
is defined inside the compact cone $\VV_{+}^{d+1}$ with $0 \le t \le 1$. An orthogonal basis for the space 
$\CV_n(\VV_{+}^{d+1},W_{\b,\g,\mu}^J)$ is given by 
\begin{equation}\label{eq:coneJ}
   \Jb_{m,\ell}^n(x,t) =  P_{n-m}^{(2\a+2m, \g)}(1- 2t) t^m P_\kb^m\left(\varpi_\mu; \frac{x}{t}\right),
\end{equation}
in terms of the Jacobi polynomials and orthogonal polynomials of $\CV_n^d(\varpi_\mu)$ on $\BB^d$.

It is shown in \cite{X19} that the Jacobi polynomials in $\CV_n(\VV_{+}^{d+1}, W_{0,\g,\mu}^J )$ with $\b = 0$
satisfy a differential equation 
\begin{align} \label{eq:diffJ}
 & \Big[ t(1-t)\partial_t^2 + 2 (1-t) \la x,\nabla_x \ra \partial_t + t \Delta_x - \la x, \nabla_x\ra + (2\mu+d)\partial_t  \\ 
 &  \qquad - (2\mu+\g+d+1)( \la x,\nabla_x\ra + t \partial_t)+\la x \nabla_x\ra \Big]u =  -n (n+2\mu+\g+d) u.\notag
\end{align}
 
In parallel to the Jacobi polynomials on the surface of the cone, the Jacobi polynomials on the solid upper cone
are not related to orthogonal polynomials for the even extension, in $t$ variable, of $W_{\b,\g,\mu}$ on the 
double cone. We shall study the Fourier expansions in the generalized Gegenbauer polynomials in 
$\CV_n^E(\VV^{d+1}, W_{\b,\g,\mu}^J)$ in the next two sections. 

\subsubsection{Double hyperboloid} 
With $\varrho > 0$ the weight function $W_{\b,\g,\mu}$ can be written as $W_{\b,\g, \mu}(t) = 
|t| w_0 (t^2-\varrho^2) (t^2-\varrho^2)^{\mu-\f12} \varpi_\mu(y)$, where $y= x/\sqrt{t^2-\varrho^2}$,
with $w_0(t) =  t^{\b-\f12}(1-t^2)^{\g-\f12}$. Hence, the polynomials $q_k^{(m)}(t)$ in 
Proposition \ref{prop:solidOPbasis} are given in terms of $p_k(w_0^{(m)}; t)$ with 
$$
   w_0^{(m)}(t) = t^{m+ \f{d-1}{2}}w_0(t) =   t^{m+ \b+ \f{d-2}{2}} (1-t^2)^{\g-\f12},
$$ 
which are the Jacobi polynomials; that is, $p_k(w_0^{(m)};s)= P_k^{(\g-\f12, m+\b+\f{d-2}{2})}(s)$. 
As in Subsection \ref{sect:4.1.3}, we only consider orthogonal polynomials that are even in $t$, that is, 
those in $\CV_n^E(\VV_0^{d+1}; W_{\b,\g,\mu})$. The orthogonal polynomials given in Corollary \ref{cor:solidOPeven} 
are now specialized as follows: 

\begin{prop}\label{prop:solidOPhypG}
Let $\{P_\kb^{n-2k}: |\kb| = n-2k, \, \kb \in \NN_0^d\}$ denote an orthonormal basis of $\CV_{n-2k}^d(\varpi_\mu)$.
Then the polynomials 
\begin{align}\label{eq:solidOPhypG}
 {}_\varrho \Cb_{n-2k,\kb}^n (x,t) = &\,  P_k^{(\g-\f12,n-2k+\b+\mu+\f{d-2}{2})}(2 t^2-2\varrho^2-1) \\
  &\times  (t^2-\varrho^2)^{\frac{n-2k}{2}} P_\kb^{n-2k}\bigg(\frac{x}{\sqrt{t^2-\varrho^2}} \bigg) \notag
\end{align}
with $|\kb| = n-2k$ and $0 \le k\le n/2$ form an orthogonal basis of $\CV_n^E(\VV^{d+1}, W_{\b,\g,\mu})$. 
Furthermore, in terms of orthogonal polynomials $\Cb_{n-2k,\ell}^n$ in \eqref{eq:solidOPconeG}, 
\begin{equation}\label{eq:solidOPhypG2}
 {}_\varrho \Cb_{n-2k,\ell}^n(x,t) =  \frac{(n-2k+\b+ \mu+ \frac{d-1}{2})_k}{(n-2k+\b+\g+\mu+\frac{d-1}{2})_k}
     \Cb_{n-2k,\ell}^n \left(x, \sqrt{t^2-\varrho^2}\right).
\end{equation}
\end{prop}

\begin{proof}
By Corollary \ref{cor:solidOPeven}, we only need to verify the relation \eqref{eq:solidOPhypG2}, which 
follows from \eqref{eq:gGegen} as in the proof of \eqref{eq:sfOPhypG2}. 
\end{proof}
 
We call these polynomials {\it Gegenbauer polynomials on the solid hyperboloid} when $\b = \frac12$, for
which the weight function is $W_{\f12,\g,\mu}(x,t) = w_{\f12,\g}(t) (t^2-\rho^2 - \|x\|^2)^{\mu-\f12}$ with 
$w_{\f12,\g}(t) = |t| (1-t^2)^{\g-\f12}$, and {\it generalized Gegenbauer polynomials on the solid hyperboloid} 
when $\b \ne \f12$. 

We now show that the elements of $\CV_n^E(\VV^{d+1}, W_{\b,\g,\mu})$ are eigenfunctions 
of a second order differential operator when $\b = \f12$. 
 
\begin{thm}\label{thm:solidHypGdiff}
For  $\rho \ge 0$ and $n=0,1,2,\ldots$, every $u \in \CV_n^E(\VV^{d+1}, W_{\f12,\g,\mu})$ satisfies 
the differential equation 
\begin{align} \label{eq:solidHypGdiff}
  & \bigg [(1+\varrho^2-t^2)  \left(1- \frac{\varrho^2}{t^2} \right) \partial_t^2 + 
  \Delta_x - \la x, \nabla_x\ra^2 + \la x, \nabla\ra \\
  & \,  +\frac{2}{t}(1+\varrho^2-t^2) \la x, \nabla_x\ra \partial_t  + 
       \left( (1+\varrho^2-t^2) \frac{\varrho^2}{t^2} + 2\mu+d \right)  \frac{1}{t}\partial_t  \notag \\
   & \, - (2\g+2\mu+d+1) \left( \left( 1- \frac{\varrho^2}{t^2}\right) t \partial_t + \la x,\nabla_x \ra \right)
      \bigg]  u = -n(n + 2 \g + 2 \mu+d) u. \notag 
\end{align}
\end{thm}

\begin{proof}
Here $\VV^{d+1}$ is a solid hyperboloid. The proof in this case is similar to that of 
Theorem \ref{thm:sfHypGdiff}. Indeed, let $u(x,t) =  \Cb_{n-2k,\ell}^n (x, t)$ and set 
$U(x,t) = u(x,\sqrt{t^2 -\varrho^2})$. By \eqref{eq:solidOPhypG2}, we need to find the differential
equation for $U$. We know that $u(x,\sqrt{t^2 -\varrho^2})$ satisfies \eqref{eq:solidConeGdiff} with 
$t$ replaced by  $T = \sqrt{t^2-\varrho^2}$, $\partial_t u$ and $\partial_t^2 u$ replaced by derivatives 
with respect to $T$. Now, replacing the derivative on $u$ by derivatives on $U$ by the relations in
\eqref{eq:f'F'} with $u= f$ and $U=F$, we obtain the differential equation satisfied by $U$, which is \eqref{eq:solidHypGdiff}.  
\end{proof}

\section{Addition formula of orthogonal polynomials}
\setcounter{equation}{0}
 
We derive addition formulas for the generalized Gegenbauer polynomials in and on either the surface 
or solid hyperboloids and double cones. These are closed formulas for the reproducing kernels of the 
space of orthogonal polynomials and their applications in the Fourier orthogonal series will be discussed
in the next section. 

\subsection{Addition formula on the surface of a hyperboloid}
We derive addition formula for the generalized Gegenbauer polynomials on the surface of the hyperboloid
$\VV_0^{d+1}$. The weight function is $w_{\b,\g}$ defined in \eqref{eq:sf6weight} for 
$\varrho \le |t| \le \sqrt{\varrho^2 +1}$. Using the connection between these polynomials on the 
hyperboloid and those on the cone, given in \eqref{eq:sfOPhypG2}, our main task is to establish the addition
formula for the Gegenbauer polynomials on the surface of the cone, which corresponds to the case 
$\varrho =0$. 

\subsubsection{Addition formula on the surface of the cone}
Here $w_{\b,\g}(t) = |t|^{2\b} (1-t^2)^{\g-\f12}$. In terms of the orthogonal polynomial 
$\sC_{m,\ell}^n$ defined in \eqref{eq:sfOPconeG} and its norm $h_{m,n}^\sC$, the reproducing 
kernel $\sP_n(w_{\b,\g})$ of $\CV_n(\VV_0^{d+1}, w_{\b,\g})$ on the double cone is given by
\begin{align} \label{eq:sfPbnCone1}
\sP_n\big (w_{\b,\g}; (x,t),(y,s) \big)& = \sum_{m=0}^n \sum_{\ell=1}^{\dim \CH_m^d}
         \frac{\sC_{m,\ell}^n(x,t) \sC_{m,\ell}^n(y,s)} {h_{m,n}^\sC}.
\end{align}
By \eqref{eq:sfOPconeG} and the addition formula for the spherical harmonics, we then obtain
\begin{align} \label{eq:sfPbnCone}
\sP_n\big (w_{\b,\g}; (x,t),(y,s) \big) =  
       \sum_{m=0}^n  &  \frac{(\a+\g+1)_m} {(\a +\f{1}2)_m} \frac{C_{n-m}^{(\g,m+\a)}(t) 
           C_{n-m}^{(\g,m+\a)}(s)} {h_{n-m}^{(\g,m+\a)}}  \\
       &  \times |t|^m |s|^m Z_m^{\frac{d-2}{2}}\Big (\Big \langle \frac{x}{|t|}, \frac{y}{|s|}\Big \rangle \Big),  \notag
\end{align}
where $\a = \b+\frac{d-1}{2}$. In order to derive a closed formula for this kernel, we work with the reproducing 
kernel $\sP^E_n(w_{\b,\g})$ of $\CV_n^E(\VV_0^{d+1}, w_{\b,\g})$, which consist of the generalized 
Gegenbauer polynomials that are even in $t$ variable, and the kernel $\sP^O_n(w_{\b,\g})$ of 
$\CV_n^O(\VV_0^{d+1}, w_{\b,\g})$, which consist of the generalized Gegenbauer polynomials that are odd
in $t$ variable. 
 
\begin{thm} \label{thm:sfPbEInt}
Let $\b, \g > -\f12$ and let $\a = \b + \frac{d-1}{2}$. For $n =1,2,\ldots$, 
\begin{align} \label{eq:sfPOadd}
  \sP_n^O \big (w_{\b,\g}; (x,t),(y,s) \big) = \frac{\a+\g+1}{\a+\f12} s t \, \sP_{n-1}^E \big (w_{\b+1,\g}; (x,t),(y,s) \big).
\end{align}
Furthermore, for $\b \ge 0$, $\g \ge 0$ and $n =0,1,2,\ldots$, 
\begin{align} \label{eq:sfPEadd}
  \sP_n^E \big (w_{\b,\g}; (x,t),(y,s) \big)
     &=  c \int_{[-1,1]^3}
      Z_n^{\a+\g} \big(  \xi(x,t,y,s; v,z)  \big)\\
       &  \times (1-z_1)^\frac{d-2}{2}(1+z_1)^{\b-1}  (1-z_2^2)^{\b-\f12} (1-v^2)^{\g-1} \d v  \d z, \notag
\end{align}
where  $c=  c_{\frac{d-2}{2},\b-1} c_{\b} c_{\g-\f12}$ in terms of the constants in \eqref{eq:c_l} and \eqref{eq:c_ab}, \begin{align*} 
  \xi(x,t,y,s; v,z) =  \frac{1-z_1}{2} \la x,y\ra \mathrm{sign}(st) + \frac{1+z_1}2 z_2 st + v \sqrt{1-s^2}\sqrt{1-t^2}. 
\end{align*}
The formula holds under the limit $\g \to 0$ and/or $\b \to 0$. 
\end{thm}

\begin{proof}
We prove \eqref{eq:sfPEadd} first. By \eqref{eq:sfOPconeG}, the polynomials $\sC_{m,\ell}^n(x,t)$ is even 
in $t$ when $n-m$ is an even integer and odd in $t$ if $n-m$ is an odd integer. Hence, by \eqref{eq:sfPbnCone} 
and \eqref{eq:sfPbE}, we obtain by setting $n -m = 2k$ that 
\begin{align} \label{eq:PbnConeE}
\sP_n^E \big (w_{\b,\g}; (x,t),(y,s) \big) = 
      \sum_{k=0}^{\lfloor \frac{n}{2}\rfloor} & \frac{(\a+\g+1)_{n-2k}} {(\a +\f{1}2)_{n-2k}} 
     \frac{C_{2k}^{(\g,\a+n-2k)}(t) 
            C_{2k}^{(\g,\a+n-2k)}(s)} {h_{2k}^{(\g, \a+n-2k)}} \\
       &  \times |t|^{n-2k} |s|^{n-2k} Z_{n-2k}^{\frac{d-2}{2}}\Big (\Big \langle \frac{x}{|t|}, \frac{y}{|s|}\Big \rangle \Big). \notag
\end{align}
For $\b =0$, the righthand side of this identity agrees with \eqref{eq:additionGG} in Appendix A with 
$\l = \g$ and $\mu = \frac{d-2}{2}$, from which the limiting case of \eqref{eq:sfPEadd}, see \eqref{eq:sfPEadd0}
below, follows. For $\b > 0$, we need to increase the index of $Z_{n-2k}^{\f{d-2}{2}}$ in \eqref{eq:PbnConeE} in 
order to apply \eqref{eq:additionGG}. For this we need the following identity \cite{X15} 
\begin{align}\label{eq:ZtoZ}
  Z_m^\l(t) =  c_{\l,\s-1} c_{\s}  
      \int_{-1}^1 \int_{-1}^1 & Z_m^{\l+\s} \left( \frac{1-z_1}{2} t + \frac{1+z_1}2 z_2\right)\\
      & \times (1-z_1)^\l (1+z_1)^{\s-1}  (1-z_2^2)^{\s-\f12}  \d z. \notag
\end{align}
Since $\a = \b+\frac{d-1}{2}$, using \eqref{eq:ZtoZ} with $\l = \frac{d-2}{2}$ and $\s = \b$, we can 
write \eqref{eq:PbnConeE} as 
\begin{align*} 
\sP_n^E & \big (w_{\b,\g}; (x,t),(y,s) \big) = 
    c_{\frac{d-2}{2},\b-1} c_{\b} \int_{-1}^1\int_{-1}^1 
         \sum_{k=0}^{\lfloor \frac{n}{2}\rfloor}   \frac{(\a+\g+1)_{n-2k}} {(\a +\f{1}2)_{n-2k}} |t|^{n-2k} |s|^{n-2k} \\
       &  \qquad \times \frac{C_{2k}^{(\g,\a+n-2k)}(t) 
    C_{2k}^{(\g,\a+n-2k)}(s)} {h_{2k}^{(\g, \a+n-2k)}}  Z_{n-2k}^{\a - \f12}\left( \frac{1-z_1}{2}
           \frac{\langle x,y \rangle}{|s t |} + \frac{1+z_1}2 z_2 \right) \notag \\
       &  \qquad \times (1-z_1)^\frac{d-2}{2}(1+z_1)^{\b-1}  (1-z_2^2)^{\b-\f12}  \d z. \notag
\end{align*}
In particular, setting $\l =\g$, $\mu = \a$ and comparing with \eqref{eq:additionGG}, we see
that \eqref{eq:sfPEadd} holds for all $\g > 0$ and $\b > 0$. Finally, the case $\g = 0$, see the
identity \eqref{eq:sfPEadd00} below, follows from taking the limit $\g \to 0$ by \eqref{eq:limit-int}. 

Next we consider the kernel $\sP_n^O \big (w_{\b,\g})$. By \eqref{eq:sfPbnCone} and \eqref{eq:sfPbO}, 
we obtain by setting $n -m = 2k+1$ that 
\begin{align*}
\sP_n^O \big (w_{\b,\g}; (x,t),(y,s) \big) = 
      \sum_{k=0}^{\lfloor \frac{n-1}{2}\rfloor} & \frac{(\a+\g+1)_{n-1-2k}} {(\a +\f{1}2)_{n-1-2k}} 
     \frac{C_{2k+1}^{(\g,\a+n-1-2k)}(t) 
            C_{2k+1}^{(\g,\a+n-1-2k)}(s)} {h_{2k+1}^{(\g, \a+n-1-2k)}} \\
       &  \times |t|^{n-1-2k} |s|^{n-1-2k} Z_{n-2k-1}^{\frac{d-2}{2}}\Big (\Big \langle \frac{x}{|t|}, \frac{y}{|s|}\Big \rangle \Big). \end{align*}
From \eqref{eq:gGegen} and \eqref{eq:gGegenNorm}, it is easy to see that 
$$
  \frac{C_{2k+1}^{(\l,\mu)}(s)C_{2k+1}^{(\l,\mu)}(t)}{h_{2k+1}^{(\l,\mu)}} = \frac{\l+\mu+1}{\mu+\f12}
       s t \frac{C_{2k}^{(\l,\mu+1)}(s)C_{2k}^{(\l,\mu+1)}(t)}{h_{2k}^{(\l,\mu+1)}}.
$$
Using this identity with $\l = \g$ and $\mu = \a+n-1-2k$, combining the constants by using 
$(c)_{n-2k-1} (c+n-2k-1) =  (c)_{n-2k}$, we obtain 
\begin{align*} 
\sP_n^O \big (w_{\b,\g}; (x,t),(y,s) \big) =  
      \sum_{k=0}^{\lfloor \frac{n-1}{2}\rfloor} & \frac{(\a+\g+1)_{n-2k}} {(\a +\f{1}2)_{n-2k}} 
     \frac{C_{2k}^{(\g,\a+n-2k)}(t)  C_{2k}^{(\g,\a+n-2k)}(s)} {h_{2k}^{(\g, \a+n-2k)}} \\
       &  \times st |t|^{n-2k-1} |s|^{n-2k-1} Z_{n-2k-1}^{\frac{d-2}{2}}\Big (\Big \langle \frac{x}{|t|}, \frac{y}{|s|}\Big \rangle \Big). \end{align*}
Writing $\a + n-2k = (\a +1) + n-1-2k$ and using $(\a+c)_{n-2k} = (\a+c) (\a+1+c)_{n-1-2k}$, we see that 
the sum in the righthand side can be written, up to a multiple constant, in terms of 
$\sP_{n-1}^E\big (w_{\b+1,\g}; (x,t),(y,s) \big)$, which establishes \eqref{eq:sfPOadd}. 
\end{proof} 

As a corollary, we state two limiting cases of the reproducing kernel explicitly. The first one is for the
Gegenbauer weight $w_{0,\g}$, which is the limiting case $\b \to 0$ of \eqref{eq:sfPEadd},
whereas the second one is for the Chebushev weight $w_{0,0}$, which follows from the first one
by taking the limit $\g \to 0$. 

\begin{cor}
For the Gegenbauer weight function $w_{0,\g}(t) = (1-t^2)^{\g-\f12}$,
\begin{align} \label{eq:sfPEadd0}
& \sP_n^E \big (w_{0,\g}; (x,t),(y,s) \big)  \\
& \qquad\ = c_{\g-\f12}  \int_{-1}^1 Z_n^{\g+\frac{d-1}{2}}
 \bigg(\la x,y\ra  \mathrm{sign}(st) + v \sqrt{1-s^2}\sqrt{1-t^2}\bigg) (1-v^2)^{\g-1} \d v. \notag
\end{align}
Furthermore, for the Chebyshev weight function $w_{0,0}(t) = (1-t^2)^{-\f12}$, 
\begin{align} \label{eq:sfPEadd00}
\sP_n^E \big (w_{0,0}; (x,t),(y,s) \big) = \f12 & \left[
      Z_n^{\frac{d-1}{2}} \bigg(\la x,y\ra  \mathrm{sign}(st) + \sqrt{1-s^2}\sqrt{1-t^2}\bigg)  \right. \\
      & + \left. Z_n^{\frac{d-1}{2}} \bigg(\la x,y\ra  \mathrm{sign}(st) - \sqrt{1-s^2}\sqrt{1-t^2}\bigg)  \right]. \notag
\end{align}
\end{cor}
 
These formulas resemble the closed form formula for the reproducing kernel of the classical orthogonal 
polynomials on the unit ball $\BB^d$. Indeed, when $t = \|x\|$, they agree with $\sP_n(\varpi_\g; x,y)$
with $\varpi_\g(x) = (1-\|x\|^2)^{\g-\f12}$. A moment reflection shows that this should be the case, since 
$(x,\|x\|) \in \VV_0^{d+1}$, the finite cone in this setting, implies that $\|x\|\le 1$ or $x\in \BB^d$. 

\subsubsection{Addition formula on the surface of hyperboloid}
Here $\rho > 0$ and $\VV_0^{d+1}$ is the hyperboloid surface with the weight function 
$w_{\b,\g}(t) = |t| (t^2-\varrho^2)^{2\b-1} (1+\varrho^2-t^2)^{\g-\f12}$ for $\varrho \le |t| \le \sqrt{1+\varrho^2}$. 
We denote by ${}_\varrho \sP_n(w_{\b,\g})$ the reproducing kernel of the 
space $\CV_n(\VV_0^{d+1}, w_{\b,\g})$. Similarly, we denote by ${}_\varrho \sP_n^E(w_{\b,\g})$ and
${}_\varrho \sP_n^O(w_{\b,\g})$ the reproducing kernels for $\CV_n^E(\VV_0^{d+1}, w_{\b,\g})$ and
for $\CV_n^O(\VV_0^{d+1}, w_{\b,\g})$, respectively.  

In this case, however, we can derive a closed formula only for ${}_\varrho \sP_n^E(w_{\b,\g})$ because of the
complexity of orthogonal polynomials that are odd in $t$ variable. 

\begin{prop}
Let $\b,\g > -\f12$. For $n=0,1,2,\ldots$, the reproducing kernel ${}_\varrho \sP_n^E(w_{\b,\g})$ of 
$\CV_n^E(\VV_0^{d+1}, w_{\b,\g})$ on the hyperboloid satisfies,  
\begin{align}\label{eq:sfPEhyp}
   {}_\varrho \sP_n^E \big(w_{\b,\g}; (x,t), (y,s)\big) 
    = \sP_n^E \left(w_{\b,\g}; \Big(x,\sqrt{t^2-\varrho^2}\Big),  \Big(y,\sqrt{s^2-\varrho^2} \Big)\right). 
\end{align}
\end{prop}

\begin{proof}
In terms of the orthogonal polynomials ${}_\varrho\sC_{n-2k,\ell}^n$ defined in \eqref{eq:sfOPhypG} and its norm 
${}_\varrho h_{n-2k,n}^\sC$, the reproducing kernel ${}_\varrho \sP_n^E(w_{\b,\g})$ on the hyperboloid is given by
\begin{align*}
  {}_\varrho\sP_n\big (w_{\b,\g}; (x,t),(y,s) \big)  = \sum_{k=0}^{\lfloor \frac{n}{2} \rfloor} 
     \sum_{\ell=1}^{\dim \CH_{n-2k}^d}
        \frac{{}_\varrho\sC_{n-2k,\ell}^n(x,t) {}_\varrho\sC_{n-2k,\ell}^n(y,s)} {{}_\varrho h_{n-2k,n}^\sC}.
\end{align*}
By the relation \eqref{eq:sfOPhypG2} and the relation between ${}_\varrho h_{n-2k,n}^\sC$ and 
$h_{n-2k,n}^\sC$ derived by the integral \eqref{eq:intHyp}, it is easy to see that
\begin{align}\label{eq:C-rhoC}
 \frac{{}_\varrho\sC_{n-2k,\ell}^n(x,t) {}_\varrho\sC_{n-2k,\ell}^n(y, s)} {{}_\varrho h_{n-2k,n}^\sC} = \frac{\sC_{n-2k,\ell}^n(x,\sqrt{t^2-\varrho^2})  \sC_{n-2k,\ell}^n(y, \sqrt{s^2-\varrho^2})} {h_{n-2k,n}^\sC},  
\end{align}
from which \eqref{eq:sfPEhyp} follows readily. 
\end{proof}

Using \eqref{eq:sfPEhyp}, we can derive an addition formula for $\CV_n^E(\VV_0^{d+1}, w_{\b,\g})$ on the 
hyperboloid from Theorem \ref{thm:sfPbEInt}. We state the counterpart of \eqref{eq:sfPEadd0} as an example. 

\begin{cor}
For the Gegenbauer weight function $w_{0,\g}(t) = (1-t^2)^{\g-\f12}$,
\begin{align} \label{eq:sfPEadd0Hyp}
& {}_\varrho \sP_n^E \big (w_{0,\g}; (x,t),(y,s) \big)= c_{\g-\f12}  \int_{-1}^1   (1-v^2)^{\g-1}  \\
& \quad\times Z_n^{\g+\frac{d-1}{2}}
 \bigg(\sqrt{1-\frac{\varrho^2}{t^2}} \sqrt{1-\frac{\varrho^2}{s^2}} \, \la x,y\ra
   + v \sqrt{1+\varrho^2-s^2} \sqrt{1+ \varrho^2 -t^2}\bigg) \d v. \notag
\end{align}
\end{cor}

\subsection{Addition formula on a solid hyperboloid}
In this subsection we derive addition formula for the generalized Gegenbauer polynomials on the solid
hyperboloid $\VV^{d+1}$, bounded by the compact surface $\VV_0^{d+1}$ with $\varrho  \le t \le \sqrt{1+\varrho^2}$ 
and by the hyperplanes $t =\pm  \sqrt{1+\varrho^2}$. The weight function $W_{\b,\g,\mu}$ is defined in 
\eqref{eq:6Weight} for $\b, \g, \mu > -\tfrac12$. 

As in the case of the surface, our main task is to establish the addition formula for Gegenbauer polynomials 
on the solid cone, which corresponds to the case $\varrho =0$. 

\subsubsection{Addition formula on the solid cone}
With $\varrho =0$, the weight function becomes 
$W_{\b,\g,\mu}(x,t) = |t|^{2\b} (1-t^2)^{\g-\f12}(1-\|x\|^2)^{\mu-\f12}$. In terms of the orthogonal polynomial 
$\Cb_{m,\ell}^n$ defined in \eqref{eq:solidOPconeG} and its norm $h_{m,n}^\Cb$, the reproducing 
kernel $\Pb_n(W_{\b,\g,\mu})$ of $\CV_n(\VV^{d+1}, W_{\b,\g,\mu})$ on the solid cone is given by
\begin{align} \label{eq:PbnCone1}
\Pb_n\big (W_{\b,\g,\mu}; (x,t),(y,s) \big)& =  \sum_{m=0}^n \sum_{|\kb|=m}     
  \frac{\Cb_{m,\kb}^n(x,t) \Cb_{m,\kb}^n(y,s)} {h_{m,n}^\Cb}.
\end{align}
By \eqref{eq:solidOPconeG} and the addition formula for the classical orthogonal polynomials on the 
unit ball, we then obtain
\begin{align} \label{eq:PbnCone}
 & \Pb_n\big (W_{\b,\g,\mu}; (x,t),(y,s) \big) = c_{\mu - \f12} \! \int_{-1}^1 \!
       \sum_{m=0}^n \!  \frac{(\a+\g+1)_m} {(\a +\f{1}2)_m} \frac{C_{n-m}^{(\g,m+\a)}(t) 
           C_{n-m}^{(\g,m+\a)}(s)} {h_{n-m}^{(\g,m+\a)}}  \\
       &\quad   \times 
        |t|^m |s|^m Z_m^{\mu+\frac{d-1}{2}}
           \left (\Big \langle \frac{x}{|t|}, \frac{y}{|s|}\Big\rangle
              + u \sqrt{1- \frac{\|x\|^2}{t^2}} \sqrt{1- \frac{\|y\|^2}{s^2}} \right) (1-u^2)^{\mu -1} \d u,  \notag
\end{align}
where $\a = \b+\mu + \frac{d-1}{2}$. Comparing the sum inside the integral with the righthand 
side of \eqref{eq:sfPbnCone}, we see that the difference is essentially on the index $\l$ 
of $Z_m^\l$. Consequently, we could follow the same method used for the surface of the cone
to obtain addition formulas for the Gegenbauer polynomials on the solid cone. Thus, we shall be
brief with the proof. 

Let $\Pb^E_n \big (W_{\b,\g,\mu})$ and $\Pb^O_n \big (W_{\b,\g,\mu})$ be the reproducing 
kernels for $\CV_n^E(\VV^{d+1}; W_{\b,\g,\mu})$ and  $\CV_n^O(\VV^{d+1}; W_{\b,\g,\mu})$.
The following theorem is an analogue of Theorem \ref{thm:sfPbEInt}.

\begin{thm} \label{thm:PbEInt}
Let $\b, \g, \mu > -1/2$ and let $\a = \b + \mu+ \frac{d-1}{2}$. For $n =1,2,\ldots$, 
\begin{align} \label{eq:POadd}
  \Pb_n^O \big (W_{\b,\g,\mu}; (x,t),(y,s) \big) = \frac{\a+\g+1}{\a+\f12} s t \, 
      \Pb_{n-1}^E \big (W_{\b+1,\g,\mu}; (x,t),(y,s) \big).
\end{align}
Furthermore, for  $\b \ge \f12$, $\g, \mu \ge 0$ and $n =0,1,2,\ldots$, 
\begin{align} \label{eq:PEadd}
  \Pb_n^E & \big  (W_{\b,\g,\mu}; (x,t),(y,s) \big)
     =   c \int_{[-1,1]^4}  Z_m^{\a+\g} \big(  \zeta(x,t,y,s; u,v,z)  \big)  \\
 &  \times (1-z_1)^{\mu+\frac{d-1}{2}}(1+z_1)^{\b-\f32}  (1-z_2^2)^{\b-\f12} (1-v^2)^{\g-1}(1-u^2)^{\mu-1} 
      \d u \d v \d z, \notag
\end{align}
where $c = c_{\mu+\frac{d-1}{2},\b-\f32} c_{\b-\f12} c_{\g-\f12}c_{\mu-\f12}$ and 
\begin{align*}
  \zeta(x,t, y,s; u, v,z) =  \frac{1-z_1}{2} \left(\la x,y\ra  + u \sqrt{t^2-\|x\|^2}\sqrt{s^2-\|y\|^2} \right) \mathrm{sign}(st) \\
         + \frac{1+z_1}2 z_2 st + v \sqrt{1-s^2}\sqrt{1-t^2}. \notag
\end{align*}
The formula holds under the limit if $\b = \f12$ or either one of $\b$ and $\mu$ is 0. 
\end{thm}

\begin{proof}
The proof follows along the same line of the proof of Theorem \ref{thm:sfPbEInt}. For the proof of 
\eqref{eq:PEadd}, we use \eqref{eq:PbnCone} to obtain 
\begin{align*}
& \Pb_n^E \big (W_{\b,\g,\mu}; (x,t),(y,s) \big) = c \int_{-1}^1
      \sum_{k=0}^{\lfloor \frac{n}{2}\rfloor} \frac{(\a+\g+1)_{n-2k}} {(\a +\f{1}2)_{n-2k}} 
     \frac{C_{2k}^{(\g,\a+n-2k)}(t) 
            C_{2k}^{(\g,\a+n-2k)}(s)} {h_{2k}^{(\g, \a+n-2k)}} \\
& \quad \times |t|^{n-2k} |s|^{n-2k} Z_{n-2k}^{\mu+\frac{d-1}{2}}
           \left (\Big \langle \frac{x}{|t|}, \frac{y}{|s|}\Big\rangle
              + u \sqrt{1- \frac{\|x\|^2}{t^2}} \sqrt{1- \frac{\|y\|^2}{s^2}} \right) (1-u^2)^{\mu -1} \d u,
\end{align*}
then apply \eqref{eq:ZtoZ}, with $\s = \b - \f12$, to increasing the index of $Z_{n-2k}^\l$ from $\mu+\frac{d-1}{2}$ 
to $\a -\frac12$ with a cost of a double integral, so that the identity \eqref{eq:additionGG} can be applied to 
give the final formula, which introduces one more layer of integral. The proof of \eqref{eq:POadd} follows 
that of \eqref{eq:sfPOadd} almost verbatim.
\end{proof}

Comparing with \eqref{eq:sfPEadd}, we see that the identity \eqref{eq:POadd} coincide with
$\sP_n(w_{\b,\g})$ when $(x,t), (y,s) \in \VV_0^{d+1}$ and $\mu = 0$. In other words, the 
restriction of the reproducing kernel $\Pb_n\big(W_{\b,\g,0}\big)$ on the boundary 
$\VV_0^{d+1} \times \VV_0^{d+1}$ is equal to $\sP_n\big(w_{\b,\g}\big)$. 
 
We state two limiting cases of \eqref{eq:PEadd} as a corollary. The first one is for the Gegenbauer weight 
function $W_{\f12,\g,\mu}$ and the second one is the Chebyshev weight $W_{\f12,0,0}$ on the solid cone.

\begin{cor} \label{cor:SolidCone}
For the Gegenbauer weight $W_{\f12,\g,\mu}(t) = |t| (1-t^2)^{\g-\f12} (1-\|x\|)^{\mu-\f12}$,
\begin{align*} 
& \Pb_n^E \big (W_{\f12,\g,\mu}; (x,t),(y,s) \big) = c_{\g-\f12} c_{\mu-\f12} \int_{-1}^1 \int_{-1}^1  (1-v^2)^{\g-1} 
   (1-u^2)^{\mu-1} \\
& \times Z_n^{\g+\mu+\frac{d-1}{2}} 
 \left( \Big(\la x,y\ra+ u \sqrt{t^2-\|x\|^2}\sqrt{s^2-\|y\|^2}\Big) \mathrm{sign}(st) + v \sqrt{1-s^2}\sqrt{1-t^2}\right)\d u \d v. \notag
\end{align*}
Furthermore, for the Chebyshev weight function $W_{\f12,0,0}(t) =|t|(1-t^2)^{-\f12}(1-\|x\|^2)^{-\f12}$, 
\begin{align*} 
\Pb_n^E & \big (W_{\f12,0,0}; (x,t),(y,s) \big) \\
   & = \f14 \sum   Z_n^{\frac{d-1}{2}} \left(
   \Big(\la x,y\ra \pm \sqrt{t^2-\|x\|^2}\sqrt{s^2-\|y\|^2}\Big) \mathrm{sign}(st)\pm \sqrt{1-s^2}\sqrt{1-t^2}\right), 
\end{align*}
where the sum is over four possible combinations of signs in the two $\pm$.
\end{cor}

In the case of $s = 1$ and $t =1$, these two identities in the corollary coincide with the closed form formulas for 
the reproducing kernel of the classical orthogonal polynomials with respect to $\varpi_\mu$ and 
$\varpi_0$, respectively, on $\BB^d$. This is what it should be since $\BB^d$ is the intersection 
of $\VV^{d+1}$ and the hyperplane $t =1$. 

\subsubsection{Addition formula on solid hyperboloid}
We consider the case $\rho > 0$ and $\VV^{d+1}$ is the solid hyperboloid bounded by the hyperboloid
surface and the hyperplane $t =1$. For the weight function $W_{\b,\g,\mu}$ given in \eqref{eq:6Weight},
we denote by ${}_\varrho \Pb_n(W_{\b,\g,\mu})$ the reproducing kernel of the space 
$\CV_n(\VV^{d+1}, W_{\b,\g,\mu})$ and by ${}_\varrho \Pb_n^E(W_{\b,\g,\mu})$ and
${}_\varrho \Pb_n^O(W_{\b,\g,\mu})$ the reproducing kernels for $\CV_n^E(\VV^{d+1}, W_{\b,\g,\mu})$ and 
for $\CV_n^O(\VV^{d+1}, W_{\b,\g,\mu})$, respectively. 
 
As in the case of the surface hyperboloid, we can derive an addition formula only for the reproducing
kernel ${}_\varrho \Pb_n^E(W_{\b,\g,\mu})$.

\begin{prop}
Let $\b,\g, \mu > -\f12$. For $n=0,1,2,\ldots$, the reproducing kernel 
of $\CV_n^E\big({}_\varrho\VV^{d+1}, W_{\b,\g,\mu}\big)$ on the solid hyperboloid satisfies
\begin{align}\label{eq:PEhyp}
   {}_\varrho \Pb_n^E\big(W_{\b,\g,\mu}; (x,t), (y, s)\big) 
   = \Pb_n^E \left(W_{\b,\g,\mu}; \Big(x,\sqrt{t^2-\varrho^2}\Big),  \Big(y,\sqrt{s^2-\varrho^2} \Big) \right). 
\end{align}
\end{prop}

\begin{proof}
In terms of the orthogonal polynomials ${}_\varrho\Cb_{n-2k,\ell}^n$ defined in \eqref{eq:solidOPhypG} 
and its norm ${}_\varrho h_{n-2k,n}^\Cb$, the reproducing kernel ${}_\varrho \Pb_n^E(W_{\b,\g,\mu})$ on the 
solid hyperboloid is given by
\begin{align*}
  {}_\varrho\Pb_n^E\big (W_{\b,\g,\mu}; (x,t),(y,s) \big)  = \sum_{m=0}^{\lfloor \frac{n}{2} \rfloor} 
     \sum_{|\kb|=n}  \frac{{}_\varrho\Cb_{n-2m,\kb}^n(x,t) {}_\varrho\Cb_{n-2m,\kb}^n(y,s)}
         {{}_\varrho h_{n-2m,n}^\Cb}.
\end{align*}
By the relation \eqref{eq:solidOPhypG2}, which also leads to the relation between 
${}_\varrho h_{n-2k,n}^\Cb$ and $h_{n-2k,n}^\Cb$ by the integral \eqref{eq:intHypSolid}, we see that
an analogue of \eqref{eq:C-rhoC} holds for $\Cb_{n-2m,\kb}^n$ and ${}_\varrho\Cb_{n-2m,\kb}^n$,
from which \eqref{eq:PEhyp} follows readily. 
\end{proof}

Using \eqref{eq:PEhyp}, addition formula for $\CV_n^E(\VV^{d+1}, W_{\b,\g,\mu})$ on the solid hyperboloid 
can then be derived from those in \eqref{eq:PEadd} for the solid cone.  

\subsubsection{Poisson kernel}
Let $P_n^E$ denote either $\sP_n(w_{\b,\g})$ on the surface $\VV_0^{d+1}$ or $\Pb_n(W_{\b,\g,\mu})$
on the solid $\VV^{d+1}$. We denote by $P^E\big(r; \cdot, \cdot \big)$  the Poisson kernel defined by 
\begin{equation} \label{eq:Poisson}
  P^E \big(r; (x,t),(y,s)\big) := \sum_{n=1}^\infty P_n^E \big ((x,t),(y,s) \big) r^n,   \quad 0 \le r < 1.
\end{equation}
Since our addition formulas for $\sP_n^E \big (w_{\b,\g}; \cdot,\cdot)$ in Theorem \ref{thm:sfPbEInt} and
for $\Pb_n^E \big (W_{\b,\g,\mu}; \cdot,\cdot)$ in Theorem \ref{thm:PbEInt} are all given in terms of integrals
of $Z_n^\l$, we can use the well-known identity (cf. \cite[p. 19]{DX}), 
$$
    \sum_{n=0}^\infty Z_n^\l(t) r^n = \frac{1-r^2}{(1-2r t + r^2)^{\l+1}}, \qquad 0 \le r < 1,
$$
to derive closed form formulas for the Poisson kernel in both cases. It is straightforward to write down
these closed form formulas. We shall leave them to interested readers. 

\section{Addition formula and Fourier orthogonal series}
\setcounter{equation}{0}

In this section we show how addition formulas can be used for studying the Fourier orthogonal series. 
There are several cases, since we have established addition formulas for the generalized Gegenbauer
polynomials on both the surface and the solid double cones and on hyperboloids. We shall state the 
result first in a general setting that includes all these cases. 

\subsection{Fourier orthogonal series} \label{sect:7FourierOS}
We consider the setting of $L^2(\UU, W)$, where $\UU$ is a domain in $\RR^D$ and $W$ is 
a weight function on $\UU$ such that $\int_{\UU} W(\xb) \d \xb =1$. For our purpose, $\UU$ can be
the surface $\VV_0^{d+1}$ or its upper part $\VV_{0,+}^{d+1}$, or the solid $\VV^{d+1}$ or its upper 
part $\VV_+^{d+1}$, and $W$ is the corresponding generalized Gegenbauer weight on $\UU$
for either the cone or the hyperboloid.

We assume that $L^2(\UU, W)$ has an orthogonal basis of polynomials. Let $\CV_n(\UU, W)$ denote 
the space of orthogonal polynomials of degree $n$ in $L^2(\UU, W)$. The Fourier orthogonal series of
$f \in L^2(\UU, W)$ is defined by 
\begin{equation} \label{eq:7Fourier}
  f = \sum_{n=0}^\infty \proj_n f(W), \qquad \proj_n(W): L^2(\UU, W) \mapsto \CV_n(\UU, W), 
\end{equation}
where $\proj_n(W)$ is the orthogonal projection operator. Let $P_n(\cdot,\cdot)$ denote the reproducing kernel 
of $\CV_n(\UU, W)$. Then the projection operator can be written as 
$$
  \proj_n f(W;\xb) = \int_{\UU} f(\yb) P_n (\xb, \yb) W(\yb) \d \yb. 
$$
In terms of an orthonormal basis $\{P_k^n: 1 \le k \le \dim \CV_n(\UU,W)\}$ of $\CV_n(\UU,W)$, the 
reproducing kernel satisfies 
$$
   P_n (\xb, \yb) = \sum_{k=1}^{\dim \CV_n(\UU,W)} \frac{P_k^n(\xb)P_k^n (\yb)}{H_k^n}, 
$$
where $H_k^n$ denotes the norm square of $P_k^n$. Our main assumption is the existence of an 
addition formula for the reproducing kernel. Recall that we use $p_n(\varpi)$ to denote an orthogonal 
polynomial of degree $n$ with respect to the weight function $\varpi$, and $h_n(\varpi)$ to denote the
norm square of $p_n(\varpi)$. 

\begin{defn}
The reproducing kernel $P_n$ is said to satisfy an addition formula if there is a weight function $\varpi$ 
on $[-1,1]$, $\int_{-1}^1 \varpi(t) \d t =1$, such that 
\begin{equation} \label{eq:7Pn}
   P_n (\xb, \yb) = \int_{[-1,1]^m} Z_n \big(\xi(\xb, \yb; \ub) \big) \d \tau (\ub), \qquad
          Z_n(t) = \frac{p_n(\varpi;1) p_n(\varpi;t)}{h_n(\varpi)},
\end{equation}
where $m$ is a positive integer; $\xi(\xb, \yb; \ub)$ is a function of $\ub \in [-1,1]^m$, symmetric 
in $\xb$ and $\yb$, and $\xi(\xb, \yb; \ub) \in [-1,1]$; and $\d \tau$ is a probability measure on 
$[-1,1]^m$, which can degenerate to have a finite support. 
\end{defn}

This assumed form of the addition formula includes, as examples, \eqref{eq:sfPEadd} and \eqref{eq:PEadd} and
their various special cases that we derived for cones and hyperboloids, as well as \eqref{eq:PnBall} 
for the unit ball and, in a totally degenerated case, the original addition formula \eqref{eq:additionF} for 
spherical harmonics. In these examples, the polynomial $Z_n$ is given by $Z_n^\l$ defined in \eqref{eq:Zn} 
for various $\l$. For the addition formulas for the Jacobi polynomials on the cone in \cite{X19} as well as those
for the the simplex in $\RR^d$ (cf. \cite[p. 275]{DX}), the polynomial $Z_n$ is given in terms of the Jacobi 
polynomial $P_n^{(\l-\f12,-\f12)}$, using the quadratic transform from $P_n^{(\l-\f12,-\f12)}(2t^2-1) = a_n 
C_{2n}^\l(t)$ if necessary.  Further examples of \eqref{eq:7Pn} are also held for reflection invariant weight 
functions on the unit sphere and on the unit ball (cf. \cite[p. 221 and p. 265]{DX}). 

Motivated by the addition formula, we define an operator $T$  such that $P_n = T Z_n$. 

\begin{defn}\label{defn:7T} 
Assume the addition formula \eqref{eq:7Pn}. For $g\in L^1([-1,1],\varpi)$, we define the 
operator $T$ on $\UU$ by 
\begin{align} \label{eq:7T}
   T g\big(\xb,\yb \big) :=\, & \int_{[-1,1]^m}  g \big( \xi (\xb, \yb; \ub)\big)  \d \tau(\ub).
\end{align}
\end{defn}

\begin{lem} \label{lem:translateT}
Let $g \in L^1([-1,1],\varpi)$. Then, for each $Q_n \in \CV_n (\UU, W)$, 
\begin{equation}\label{eq:7FH}
     \int_{\UU}  T g (\xb,\yb) Q_n (\yb) W (\yb) \d \yb  =  \Lambda_n (g) Q_n(\xb),
\end{equation}
where 
$$
 \Lambda_n (g) =  \int_{-1}^1 g(t) \frac{p_k(\varpi; t)}{p_k(\varpi;1)} \varpi(t) \d t.
$$
Furthermore, for $1\le p \le \infty$ and $\xb \in \UU$, 
\begin{equation}\label{eq:Tbd}
  \left \| T g (\xb, \cdot)\right\|_{L^p(\UU, W)} \le \|g \|_{L^p([-1,1],\varpi)}.
\end{equation}
\end{lem}

\begin{proof}
The proof is standard by now and we shall be brief. If $g$ is a polynomial of degree at most $n$, then 
by the orthogonality of $p_k(\varpi)$ and the definition of $Z_n$ in \eqref{eq:7Pn}, 
$$
g(t) = \sum_{m=0}^n \Lambda_k Z_{k}(t), \qquad   \Lambda_k p_k(\varpi; 1) =  \int_{-1}^1 g(t) p_k(\varpi; t) \varpi(t) \d t.
$$ 
Using \eqref{eq:7Pn} and \eqref{eq:7T}, we obtain 
$$
  T g (\xb, \yb) = \sum_{k=0}^n \Lambda_k T Z_n (\xb,\yb) =   \sum_{k=0}^n \Lambda_k P_k (\xb,\yb). 
$$
Consequently,  by the definition of the reproducing kernel, \eqref{eq:7FH} holds for all polynomials. 
The usual density argument then completes the proof for all $g \in L^1([-1,1],\varpi)$. 
Furthermore, the inequality \eqref{eq:Tbd} holds immediately for $p = \infty$ and $p =1$. The case
$1 < p < \infty$ follows from the Riesz-Thorin theorem. 
\end{proof}

For $\xb \in \UU$, the operator $g \mapsto T g (\xb,\cdot)$ defines a ``translation" of $g$ by $\xb$. 
We use this operator to define a convolution structure.

\begin{defn}\label{defn:7convol}
Let $f \in L^1(\UU, W)$ and $g \in L^1([-1,1],\varpi)$, define the convolution of $f$ and $g$ on $\UU$ 
by 
$$
  (f \ast  g)(\xb)  :=   \int_{\UU} f(\yb) T g (\xb,\yb) W(\yb) \d \yb.
$$
\end{defn}

This convolution satisfies the classical Young's inequality:

\begin{thm} \label{thm:Young}
Let $p,q,r \ge 1$ and $p^{-1} = r^{-1}+q^{-1}-1$. For $f \in L^q(\UU, W)$ and
$g \in L^r([-1,1]; \varpi)$ with $g$ an even function, 
\begin{equation} \label{eq:Young}
  \|f \ast g\|_{L^p (\UU, W)} \le \|f\|_{L^q(\UU, W)}\|g\|_{L^r( [-1,1];\varpi)}.
\end{equation}
\end{thm}

\begin{proof}
The proof is standard. Using Minkowski's inequality and \eqref{eq:Tbd}, 
\begin{align*}
     \|f  \ast  g\|_{L^r ( \UU, W)} \le \int_{\UU} |f(\xb)|  \| T g(\xb,\cdot)\|_{L^r ( \UU, W)}  \d \xb 
        \le  \|f \|_{L^1 ( \UU, W)} \|g\|_{L^r([-1,1],\varpi)}. 
\end{align*}
Furthermore, by H\"older's inequality and \eqref{eq:Tbd}, we see that 
$$
\|f \ast g\|_\infty \le  \|f \|_{L^{r'}( \UU, W)}  \|g\|_{L^r([-1,1],\varpi)}, 
$$
where $\frac{1}{r'} + \frac1{r} = 1$. The inequality \eqref{eq:Young} follows from these two
inequalities by the Riesz-Thorin theorem. 
\end{proof}

The $n$-th partial sum of the Fourier orthogonal series of \eqref{eq:7Fourier} is defined by 
$$
      S_n(W;f) = \sum_{k=0}^n \proj_k (W; f).
$$
Recall the kernel $k_n(\varpi; \cdot,\cdot)$, defined in \eqref{eq:1dkernel}, of the $n$-th partial sum $s_n(\varpi; g)$. 

\begin{prop}
Assume the addition formula \eqref{eq:7Pn}. For $f \in L^1(\UU, W)$, 
\begin{equation} \label{eq:SnW}
    S_n(W; f) =  f \ast k_n(\varpi; 1, \cdot). 
\end{equation}
\end{prop}

The identity \eqref{eq:SnW} shows that the Fourier orthogonal series has a one-dimension structure if the 
orthogonal polynomials in $L^2(\UU,W)$ admits an addition formula. Furthermore, the identity allows us to
derive results on the summability of the Fourier orthogonal series on $\UU$ from that of Fourier orthogonal
series with respect to $\varpi$ in one variable. To be more precise, we consider the applications to the cones 
and hyperboloid in the next subsection. 

\subsection{Ces\`aro summability in and on hyperboloids}
 
For $\delta > 0$, the Ces\`aro $(C,\delta)$ means $S_n^\delta (W;f)$ of the Fourier orthogonal series 
is defined by 
$$
 S_n^\delta (W;f) := \f{1}{\binom{n+\delta}{n}} \sum_{k=0}^n \binom{n-k+\delta}{n-k} \proj_k(W; f).
$$
By th addition formula \eqref{eq:7Pn} and \eqref{eq:SnW}, we can write $S_n^\delta(W)$ as  
\begin{equation} \label{eq:(C,d)S_nW}
     S_n^\delta (W; f) =  \int_{\UU} f(y) K_n^\delta(W;\cdot, \yb) W(\yb)\d \yb = f \ast k_n^\delta (\varpi; 1, \cdot),
\end{equation}
where $K_n^\delta(W;\xb,\yb) = T [k_n^\delta(\varpi; 1,\cdot)](\xb,\yb)$ and $k_n^\delta(\varpi)$ is the kernel of the 
Ces\`aro $(C,\delta)$ means $s_n^\delta(\varpi)$ for $\varpi$ on $[-1,1]$, 
$$
  k_n^\delta (\varpi; u,v) = \frac{1}{\binom{n+\delta}{n}} \sum_{k=0}^n \binom{n-k+\delta}{n-k} 
        \frac{p_k(\varpi; u)p_k(\varpi; v)}{h_k(\varpi)}.
$$
To prove the convergence of $S_n^\delta (W;f)$ to $f$ in $L^1(\UU,W)$ or $C(\UU,W)$, it suffices to
show 
$$
  \sup_{\xb \in \UU} \int_{\UU} \left |K_n^\delta(W;\xb,\yb) \right | W(\yb) \d \yb < \infty. 
$$
Similar convergence result holds for $s_n^\delta (\varpi)$ if $\sup_{x\in [-1,1]}\int_{-1}^1 |k_n^\delta(x,y)| \varpi(y) \d y
< \infty.$

\subsubsection{Ces\`aro summability on the surface of hyperboloid}
Our first theorem is about the Fourier orthogonal series on the surface of the hyperboloid, for which
$\UU = \VV_0^{d+1}$, $W(x,t) = w_{\b,\g}(t)$, defined in \eqref{eq:sf6weight}, and $\varpi =\varpi_\l$,
where $\varpi_\l(u) = (1-u^2)^{\l-\f12}$ is the Gegenbauer weight function and $\l= \b+\g+ \frac{d -1}{2}$ 
by \eqref{eq:sfPEadd}. 

\begin{thm} \label{thm:sfConvE}
Let $\b,\g \ge 0$ and define $\l_{\b,\g}: = \b+\g+ \frac{d -1}{2}$. If $f \in L^1({}_\varrho\VV_0^{d+1}, w_{\b,\g})$, 
$\varrho \ge 0$, is even in the variable $t$, then the Ces\`aro $(C,\delta)$ means for the Fourier orthogonal 
series with respect to $w_{\b,\g}$ on ${}_\varrho\VV_{0}^{d+1}$ satisfy 
\begin{enumerate} [\quad 1.]
\item if $\delta \ge 2 \l_{\b,\g}+1$, then $\sS_n^\delta(w_{\b,\g}; f)$ is nonnegative if $f$ is nonnegative;
\item $\sS_n^\delta (w_{\b,\g}; f)$ converge to $f$ in $L^1({}_\varrho\VV_0^{d+1}, w_{\b,\g})$ or 
$C({}_\varrho\VV_0^{d+1})$ norm if $\delta > \l_{\b,\g}$, and $\delta > \l_{0,\g}$ is also necessary for $\b =0$.
\end{enumerate}
Furthermore, if $\rho = 0$, and $f\in L^1(\VV_0^{d+1}, w_{\b,\g})$ is odd in the variable $t$, then the 
same conclusion holds with $\l_{\b,\g}$ replaced by $\l_{\b+1,\g}$. 
\end{thm}
 
\begin{proof}
Let us denote by $\ast_{\b,\g}$ the convolution defined in Definition \ref{defn:7convol} with respect to 
$W = w_{\b,\g}$ on $\VV_0^{d+1}$. By \eqref{eq:(C,d)S_nW}, it follows from \eqref{eq:sfPEhyp} and 
\eqref{eq:sfPEadd} that 
$$
\sS_n^\delta \big(w_{\b,\g}; f\big) = f \ast_{\b,\g} k_n^\delta \left(\varpi_{\l_{\b,\g}};\cdot,1\right).
$$
It is known \cite{Ask} that $k_n^\delta (\varpi_\l; \cdot, \cdot)$ is nonnegative on $[-1,1]$ if and only if 
$\delta \ge 2 \l +1$, form which the first assertion follows. Furthermore, $\int_{-1}^1 |k_n^\delta (\varpi_\l;1,y)| \d y$ 
is finite if and only if $\delta > \lambda$ \cite[Theorem 9.1.3]{Sz}, from which the sufficiency of the second 
assertion follows from \eqref{eq:(C,d)S_nW} and Theorem \ref{thm:Young}. To prove that $\delta > \l_{0,\g}$ 
is also necessary when $\b = 0$, we consider the convergence of $\sS_n^\delta(w_{0,\g}; f)$ at $(\xi, 1) \in
 \VV_0^{d+1}$, where $\xi \in \sph$, which converges if and only if 
$$
 I_n^\delta :=  \int_{\VV_0^{d+1}} \left| \sK_n^\delta \big(w_{0,\g}; (\xi,1), (y,s) \big) \right|
      w_{0,\g}(s) \d \s (y,s) < \infty.
$$
By \eqref{eq:sfPEhyp} and \eqref{eq:intHyp}, we only need to consider the case of $\varrho = 0$, for which
we obtain by \eqref{eq:sfPEadd0} with $x= \xi \in \sph$, $t =1$, and $y = s\eta$, $\eta \in \sph$ that 
$$
   \sP_n^E\left(w_{0, \g}; (\xi,1), (s\eta,s)\right) = Z_n^{\g+\frac{d-1}{2}} \big(\la \xi,s\eta \ra \mathrm{sign} (s)\big)
   = Z_n^{\g+\frac{d-1}{2}} \big(\la \xi, |s|\eta\ra\big), 
$$
which implies, by \eqref{eq:(C,d)S_nW}, that the kernel of $(C,\delta)$ means of $\sS_n^\delta(w_{0,\g}; f)$
satisfies
$$
\sK_n^\delta \big(w_{0,\g}; f, (\xi,1), (s\eta,s)\big) =  k_n^\delta (\varpi_{\l_{0,\g}}; |s|\la \xi, \eta\ra,1).
$$
The integral of a zonal function $g(\la \xi,y\ra)$ over $\sph$ is known to satisfy (cf. \cite[p. 412]{DaiX})
$$
  \int_{\sph} g(\la \xi,\eta\ra) \d \s(\eta) = \s_{d-1} \int_{-1}^1 g(u) (1-u^2)^{\f{d-3}{2}}\d u.
$$
Hence, recall that $\varpi_{\l_{0,\g}} = \varpi_{\l+\f{d-1}{2}}$, it follows that 
$$
 I_n^\delta  =    \s_{d-1} \int_{-1}^1 |s|^{d-1} \int_{-1}^1 \left|k_n^\delta \big(\varpi_{\l+\f{d-1}{2}}; |s| u \big)\right|
   (1-u^2)^{\f{d-3}{2}}\d u (1-s^2)^{\g-\f12} \d s.
$$
Changing variable $t = su$ and exchanging the order of the integrals, we obtain 
\begin{align*}
 I_n^\delta  =  &\, 2  \s_{d-1} \int_{0}^1 \int_{-s}^s \left|k_n^\delta \big(\varpi_{\l+\f{d-1}{2}}; t,1 \big)\right|
   s(s^2-t^2)^{\f{d-3}{2}} (1-s^2)^{\g-\f12}  \d t \d s\\
    = &\, 2  \s_{d-1} \int_{-1}^1 \left|k_n^\delta \big(\varpi_{\l+\f{d-1}{2}}; t,1 \big)\right|
    \int_{|t|}^1 s(s^2-t^2)^{\f{d-3}{2}} (1-s^2)^{\g-\f12}  \d s  \d t \\
    = &\, \s_{d-1} B\left(\tfrac{d-1}2, \g+\tfrac12\right)
     \int_{-1}^1 \left|k_n^\delta \big(\varpi_{\l+\f{d-1}{2}}; t,1 \big)\right| (1-t^2)^{\g+\f{d-2}{2}} \d t,
\end{align*}
where $B(\a,\b)$ denotes the beta function and the last step follows from evaluating the integral over $s$ by
beta integral.  By \cite[Theorem 9.1.3]{Sz}, the last integral is bounded if and only if $\delta >\l+\f{d-1}{2}$. 
This completes the proof of the second assertion. 

Next we assume $\varrho = 0$ and consider the case when $f$ is odd in $t$ variable. For such an $f$, 
its orthogonal projection satisfies 
$$
  \proj_n\left(w_{\b,\g}; f\right)   = b_{\b,\g}\int_{\VV_0^{d+1}} 
     f(y,s) \sP_n^O\left(w_{\b,\g};  (\cdot,\cdot),(y,s)\right) w_{\b,\g}(y,s) \d \s(y,s),
$$
where $b_{\b,\g}$ is the normalization constant of $w_{\b,\g}$ over $\VV_0^{d+1}$. By \eqref{eq:sfPOadd},
we obtain
\begin{align*}
  \proj_n\left(w_{\b,\g}; f\right) & = b_{\b+1,\g} t \int_{\VV_0^{d+1}} 
     f(y,s) \sP_{n-1}^E\left(w_{\b+1,\g};  (\cdot,\cdot),(y,s)\right) w_{\b+1,\g}(y,s) \d \s(y,s) \\
        & = t  \left(f \ast_{\b+1,\g} Z_n^{\b+\g+ \f{d+1}{2}}\right)(x,t),
\end{align*}
where we have used $\frac{\b+\g+\f{d}2}{\a+\f12} b_{\b,\g} = b_{\b+1,\g}$. In particular, it follows that
$$
  \sS_n^\delta \left(w_{\b,\g}; f\right) =  t \left( f \ast_{\b+1,\g} k_n^\delta(\varpi_{\l_{\b+1,\g}}; \cdot,1) \right).
$$
The additional $t$ in the right hand side shows that $L^1(\VV_0^{d+1}, w_{\b,\g})$ norm of 
$\sS_n^\delta \left(w_{\b,\g}; f\right)$ becomes $L^1(\VV_0^{d+1}, w_{\b+1,\g})$ norm of
$f \ast_{\b+1,\g} k_n^\delta(\varpi_{\b+\g+1}; \cdot,1)$. The rest of the proof then follows exactly as
in the even case. 
\end{proof}
 
For functions that are even in $t$ variable, we could state Theorem \ref{thm:sfConvE} equivalently in 
terms of functions defined on $\VV_{0,+}^{d+1}$; see Theorem \ref{pop:FourierEven}. In \cite{X19}, we
studied the $(C,\delta)$ means for the Fourier orthogonal series on the upper cone $\VV_{0,+}$ with respect 
to the Jacobi weight $t^\b (1-t)^\g$ for $0 \le t \le 1$ and proved, in particular, that the $(C,\delta)$ means 
converge if $\delta > \b+\g+d$ and that this condition is sharp if $\g = -\f12$. To prove that $\delta > \b+ d -\f12$
is sharp when $\g = -\f12$, we considered the convergence at the apex  $(x,t) = (0,0)$ of the cone in \cite{X19}. 

In contrast, we used the convergence at $(x,t) = (\xi,1)$, $\xi \in \sph$, a point at the brink of the cone surface,
to show that $\delta > \g + \frac{d-1}{2}$ is sharp for $\b = 0$ in the proof of Theorem \ref{thm:sfConvE}. 
Given the peculiarity of the point $(0,0)$, one may ask if we could also use $(x,t) = (0,0)$ instead in the proof.

\begin{prop} \label{prop:Cd@(0,0)}
Let $f \in C(\VV_0^{d+1})$ be even in $t$ variable. For $\g \ge 0$, $\sS_n^\delta(w_{0,\g}; f)$ converges to 
$f$ at $(x,t) = (0,0)$ if $\delta > \min \{\frac{d-1}{2},\g\}$ and it is sharp if $\g \ge \f{d-1}{2}$. 
\end{prop}

\begin{proof}
By \eqref{eq:sfPEadd0}, we obtain 
\begin{align*}
  \sP_n^E\left(w_{0,\g}; (0,0); (y,s)\right)\, & = c_{\g-\f12} \int_{-1}^1 Z_n^{\g+\f{d-1}{2}}\left(v \sqrt{1-s^2} \right)
       (1-v^2)^{\g-\f12}d\g  \\
     & = \frac{C_n^{(\f{d-1}{2}, \g)}(1)C_n^{(\f{d-1}{2}, \g)}\left(\sqrt{1-s^2} \right)}{h_n^{(\f{d-1}{2}, \g)}},
\end{align*}
where we have used \eqref{eq:gGegenIntertw} and \eqref{eq:gGegen@1} of the generalized Gegenbauer 
polynomials. It follows readily that the kernel of the $(C,\delta)$ means satisfies 
$$
  \sK_n^\delta \left(w_{0,\g}; (0,0); (y,s)\right) = k_n^\delta\left(\varpi_{\f{d-1}{2}, \g}; \sqrt{1-s^2}, 1\right) 
$$
in terms of the kernel for the weight function $\varpi_{\f{d-1}{2}, \g}(u) = |u|^{d-1}(1-u^2)^{\g -\f12}$. Hence,
\begin{align*}
   \int_{\VV_0^{d+1}} \left| \sK_n^\delta \left(w_{0,\g}; (0,0); (y,s)\right)\right| w_{0,\g}(s) \d \s 
    &  =  \s_d \int_{-1}^1 \left| k_n^\delta\left(\varpi_{\f{d-1}{2}, \g}; \sqrt{1-s^2}, 1\right) \right| w_{0,\g}(s)\d s \\
    & = \s_d \int_{-1}^1 \left| k_n^\delta\left(\varpi_{\f{d-1}{2}, \g}; u, 1\right) \right| \varpi_{\f{d-1}{2}, \g}(u)\d u,
\end{align*}
where $\s_d$ is the surface area of $\sph$. It follows that $S_n^\delta(w_{\b,\g}; f, (0,0))$ converges if 
and only if $s_n^\delta (\varpi_{\f{d-1}{2}, \g}; g)$ converges at $t=1$ for any continuous function $g$, which
holds if $\delta > \min \{\f{d-1}{2}, \g\}$ and only if for $\g > \f{d-1}{2}$, as shown in \cite[Section 8.5]{DaiX}. 
\end{proof}

Together with \cite{X19}, our study shows that while the origin $(0,0)$ is a critical point for the Fourier 
orthogonal series in the Jacobi polynomials on the upper cone, it is no longer the case for the Fourier 
orthogonal series in the Gegenbauer polynomials on the double cone. 

\subsubsection{Ces\`aro summability on the solid hyperboloid}
Our next theorem is about the Fourier orthogonal series on the solid hyperboloid, for which
$\UU = \VV^{d+1}$, $W(x,t) = W_{\b,\g,\mu}(x,t)$, defined in \eqref{eq:6Weight}, and $\varpi =\varpi_\l$
is the Gegenbauer weight function with $\l= \b+\g+ \mu+ \frac{d -1}{2}$ by \eqref{eq:PEadd}. 

\begin{thm} \label{thm:ConvE}
Let $\b \ge \f12$, $\g, \mu \ge 0$ and define $\l_{\b,\g,\mu}: = \b+\g+ \mu+ \frac{d -1}{2}$. If 
$f \in L^1({}_\varrho\VV^{d+1}, W_{\b,\g,\mu})$, $\varrho \ge 0$, is even in the variable $t$, then the Ces\`aro 
$(C,\delta)$ means for the Fourier orthogonal series with respect to $w_{\b,\g}$ on ${}_\varrho\VV^{d+1}$ satisfy 
\begin{enumerate} [\quad 1.]
\item if $\delta \ge 2 \l_{\b,\g}+1$, then $\Sb_n^\delta(W_{\b,\g,\mu}; f)$ is nonnegative if $f$ is nonnegative;
\item $\Sb_n^\delta (W_{\b,\g,\mu}; f)$ converge to $f$ in $L^1({}_\varrho\VV^{d+1}, W_{\b,\g,\mu})$ or 
$C({}_\varrho\VV^{d+1})$ norm if $\delta > \l_{\b,\g,\mu}$, and $\delta > \l_{0,\g,\mu}$ is also necessary for
$\b = \f12$. 
\end{enumerate}
Furthermore, if $\rho = 0$, and $f\in L^1({}_0\VV^{d+1}, W_{\b,\g,\mu})$ is odd in the variable $t$, then the 
same conclusion holds with $\l_{\b,\g,\mu}$ replaced by $\l_{\b+1,\g,\mu}$. 
\end{thm}

\begin{proof}
Using \eqref{eq:PEhyp}, \eqref{eq:PEadd} and \eqref{eq:POadd}, the proof of this theorem follows 
almost verbatim as that of Theorem \ref{thm:sfConvE}. We state necessary formulas for proving
that $\delta > \l_{0,\g,\mu}$ is necessary when $\b = \f12$ below and omit rest of the proof. 

To prove that $\delta > \l_{\f12,\g,\mu}$ is necessary, we consider the convergence of 
$\Sb_n^\delta(W_{\f12,\g,\mu}; f)$ at $(\xi, 1) \in  \VV^{d+1}$, where $\xi \in \sph$. By \eqref{eq:PEhyp} and \eqref{eq:intHypSolid}, we only need to consider the case of $\varrho = 0$, for which
we obtain by \eqref{eq:sfPEadd0} with $x= \xi \in \sph$, $t =1$, and $y = s u$, $u \in \BB^d$ that 
$$
   \Pb_n^E\left(W_{\f12, \g,\mu}; (\xi,1), (s u,s)\right) = Z_n^{\g+\mu+\frac{d}{2}} \big(\la \xi,s u \ra \mathrm{sign} (s)\big)
   = Z_n^{\g+\mu + \frac{d}{2}} \big(|s| \la \xi, u\ra\big). 
$$
The integral of a zonal function $g(\la \xi,u\ra)$ over $\BB^d$ is known to satisfy (cf. \cite[p. 412]{DaiX})
$$
  \int_{\BB^d} g(\la \xi,u\ra) (1-\|u\|^2)^{\mu-\f12} \d u = c_\mu \int_{-1}^1 g(t) (1-t^2)^{\mu+ \f{d-2}{2}} \d t.
$$
Using these identities, the proof can be carried out as in the proof of Theorem \ref{thm:sfConvE}.
\end{proof}

For functions that are even in $t$, we can also state Theorem \ref{thm:ConvE} equivalently for functions 
defined on $\VV_{+}^{d+1}$; see Theorem \ref{pop:FourierEvenSolid}. We could also compare our result 
with the Jacobi polynomials on the cone in \cite{X19}, where it was shown that $(C,\delta)$ means for the
Fourier orthogonal series on the upper cone $\VV_{+}$ with respect to the Jacobi weight 
$t^\b (1-t)^\g (t^2-\|x\|^2)^{\mu-\f12}$ converge if $\delta > 2\mu + \b+\g+d$ and that this condition 
is sharp if $\g = -\f12$.  

Like the case on the surface of the cone, the origin $(0,0)$ is no longer a critical point of the
Fourier orthogonal series in Gegenbauer polynomials on the solid double cone. 

\begin{prop}
Let $f \in C(\VV^{d+1})$ be even in $t$ variable. For $\g \ge 0$ and $\mu \ge 0$,
$\Sb_n^\delta(W_{\f12,\g,\mu}; f)$ converges to $f$ at $(x,t) = (0,0)$ if and only if 
$\delta > \min \{\mu+\frac{d}{2},\g\}$.
\end{prop}

\begin{proof}
Using the identity in Proposition \ref{cor:SolidCone}, we can follow the proof of Proposition \ref{prop:Cd@(0,0)}
to show that the kernel of the $(C,\delta)$ means satisfies 
$$
  \Kb_n^\delta \left(W_{\f12,\g,\mu}; (0,0); (y,s)\right) = k_n^\delta\left(\varpi_{\mu+\f{d}{2}, \g}; \sqrt{1-s^2}, 1\right), 
$$
where $\varpi_{\mu+\f{d}{2}, \g}(s) = |s|^{2\mu +d}(1-s^2)^{\g-\f12}$, from which the stated result follows as in 
the proof of Proposition \ref{cor:SolidCone}. 
\end{proof}
 
 \section{Hermite polynomials in and on a hyperboloid} \label{sec:sfHermiteHyp}
\setcounter{equation}{0}

Assuming that the weight function contains $e^{-t^2}$ in the $t$ variable, we discuss the corresponding 
orthogonal polynomials on the surface of a hyperboloid $\VV_0^{d+1}$ and on the domain bounded by the
surface $\VV^{d+1}$ in two consecutive subsections. In both cases, our discussion includes the double 
cone with $\varrho =0$. These polynomials can be treated as the limits of the generalized Gegenbauer 
polynomials studied in Section \ref{sec:Gegen}.

\subsection{Hermite polynomials on the surface of a hyperboloid} 
We consider orthogonal polynomials on the surface of hyperboloid
$$
  {}_\varrho\VV_0^{d+1} = \left\{(x,t): \|x\|^2 = t^2 - \varrho^2, \, x \in \RR^d, \, |t| \ge \varrho\right\},
$$
which is a double hyperboloid when $\varrho > 0$ and degenerates to a double cone when $\varrho = 0$.
We choose the weight function $w$ as 
$$
 w_{\b}(t) = |t| (t^2-\varrho^2)^{\b-\f12} e^{-t^2},  \quad \b > -\tfrac12, \quad |t| \ge \varrho \ge 0,
$$
which is an even function defined on $\VV_0^{d+1}$. Correspondingly, the inner product becomes 
$$
  \la f,g\ra_{w_\b} =  b_\b  \int_{\VV_0^{d+1}} f(x,t) g(x,t) |t| (t^2-\varrho^2)^{\b-\f12} e^{-t^2} \d \s(x,t),
$$
where $b_{\b} =e^{\varrho^2}/(\Gamma(\b+\frac{d}2) \s_d)$ and $\s_d$ is the surface area of $\sph$.

\subsubsection{Double cone}
We consider the cases $\varrho = 0$ first, for which our construction give a full basis of orthogonal 
polynomials explicitly. In this case
$$
w_\b(t) = |t|^{2\b} e^{- t^2}, \qquad \b > - \tfrac{d}2,
$$ 
where the condition $\b > - \tfrac{d}2$ guarantees that $w_\b$ is integrable over $\VV_0^{d+1}$. The 
polynomials $q_k^{(m)}$ in Proposition \ref{prop:sfOPbasis} are orthogonal with respect to 
$$
         |t|^{2m+d-1} w_\b(t) = |t|^{2m +2\b + d-1} e^{-t^2},
$$ 
and, hence, are the generalized Hermite polynomials $H_n^{\a}$, with $\a = m + \b + \frac{d-1}{2}$,
that are given in the Appendix \ref{sect:append}. In particular, $H_n^\a$ is given explicitly in
\eqref{eq:gHermite} and the square of its $L^2$ norm, $h_n^\a$, is given in \eqref{eq:gHermiteNorm}. 

The orthogonal polynomials given in \eqref{eq:sfOPbasis} are now specialized as follows: 

\begin{prop}\label{prop:sfOPHermite}
Let $\{Y_\ell^m: 1 \le \ell \le \dim \CH_n^d\}$ denote an orthonormal basis of $\CH_m^d$.
Then the polynomials 
\begin{equation}\label{eq:sfOPconeH}
  \sH_{m,\ell}^n (x,t) := H_{n-m}^{m+\b + \frac{d-1}{2}}(t) Y_\ell^{m}(x), 
  \quad 1 \le \ell \le \dim \CH_m^d, \quad 0 \le m \le n,
\end{equation}
form an orthogonal basis of $\CV_n(\VV_{0}^{d+1}, w_\b)$. Moreover, 
\begin{equation}\label{eq:sfOPconeHnorm}
  h_{m,n}^\sH: = \la \sH_{m,\ell}^n, \sH_{m,\ell}^n\ra_{w_\b} =  (\b+\tfrac{d}2)_m h_{n-m}^{m+\frac{\b+d-1}{2}}.
\end{equation}
\end{prop}

\begin{proof}
The polynomials in \eqref{eq:sfOPconeH} consist of an orthogonal basis by 
Proposition \ref{eq:sfOPbasis}. The norm is computed as in Proposition \eqref{prop:sfOPconeG}.
\end{proof}

\begin{rem}\label{rem:4.1}
If $n-m$ is odd, then the polynomials $\sH_{m,\ell}^n(x,t)$ contains a factor $t$ by \eqref{eq:gHermite}.
Consequently, the space $\CV_n^O(\VV_0^{d+1}, w_\b)$ that contains these polynomials are well 
defined for $\b > - \tfrac{d + 2}2$. In particular, it is well defined for $\b = -1$ for all $d \ge 1$. 
\end{rem}

We shall call these polynomials {\it Hermite polynomials on the cone} when $\b = 0$ and 
{\it generalized Hermite polynomials on the cone} when $\b \ne 0$.  

\begin{thm}\label{thm:sfGtoH}
Let $\sC_{m,\ell}^n$ be the generalized Gegenbauer polynomials in $\CV_n(\VV_0^{d+1}, w_{\b,\g})$
defined at \eqref{eq:sfOPconeG}. Then
\begin{equation}\label{eq:sfGtoH}
\lim_{\g \to \infty} \frac{\sC_{m,\ell}^n\left(\frac{x}{\sqrt{\g}}, \frac{t}{\sqrt{\g}} \right)}{\sqrt{h_{m,n}^\sC}}
    = \frac{\sH_{m,\ell}^n(x,t)}{\sqrt{h_{m,n}^\sH}}.
\end{equation}
\end{thm}

\begin{proof}
From \eqref{eq:gGegen} and \eqref{eq:gHermite}, we see that \eqref{eq:sfGtoH} holds trivially if $m =n$ 
since $\sC_{n,\ell}^n = \sH_{n,\ell}^n = Y_\ell^n(x)$. We assume $0 \le m \le n -1$. By \eqref{eq:gGegenNorm} 
and \eqref{eq:gHermiteNorm}, it is easy to see that $\l^{-n} h_n^{(\l,\mu)} \to (\k_n^\mu)^2 h_n^\mu$ as 
$\l \to \infty$, where $\k_n^\mu$ is the coefficient defined in \eqref{eq:limgC=gH}. Consequently, by \eqref{eq:sfOPconeGnorm} and \eqref{eq:sfOPconeHnorm}, we obtain
$$
  \lim_{\g \to \infty} \frac{ h_{m,n}^\sC}{\g^{n-2m}} =  \frac{(\b + \frac{d}{2})_m}{1+\CO(\g^{-1})}
      \frac{ h_{n-m}^{(\g,\mu)}}{\g^{n-m}} 
       = [\k_{n-m}^{\mu} ]^2(\b + \tfrac{d}{2})_m h_{n-m}^{\mu} = 
       [\k_{n-m}^{\mu} ]^2h_{m,n}^\sH.
$$
where $\mu = m+\b+\frac{d-1}{2}$. For $(x,t) \in \VV_0^{d+1}$, we let $x = t \xi$, $\xi \in \sph$. Since $Y_\ell^m$ is homogeneous, it follows from \eqref{eq:limgC=gH} that 
\begin{align*}
 \lim_{\g \to \infty}  \frac{1}{\left(\sqrt{\g}\right)^{n-2m}}\sC_{m,\ell}^n\left(\frac{x}{\sqrt{\g}}, \frac{t}{\sqrt{\g}} \right) 
 &  =   \lim_{\g \to \infty}  \frac{1}{\left (\sqrt{\g}\right)^{n-m}} C_{n-m}^{(\g,\mu)} \left(\frac{t}{\sqrt{\g}} \right)  
         t^m Y_\ell^{m}(\xi) \\
  & = \k_{n-m}^{\mu} H_{n-m}^{\mu}(t)  t^m Y_\ell^{m}(\xi) = \k_{n-m}^\mu \sH_{m,\ell}^n(x,t).
\end{align*}
Together, the above two displayed limits imply \eqref{eq:sfGtoH}. 
\end{proof}

The limit relation \eqref{eq:sfGtoH} allows us to derive the differential equation satisfied by the Hermite
polynomials on the cone from those satisfied by the Gegenbauer polynomials on the cone.

\begin{thm}\label{thm:sfconeHdiff}
For $n \ge 0$, every $u \in \CV_n^E(\VV_0^{d+1},  e^{-t^2})$ satisfies the differential equation 
\begin{equation} \label{eq:sfconeHdiff}
  \left[  \partial_t^2  - 2t \partial_t  + \frac{d-1}{t} \partial_t + \frac{1}{t^2} \Delta_0^{(x)} \right] u = - 2n u;
 \end{equation}
moreover, every $u \in \CV_n^O(\VV_0^{d+1},  |t|^{-1} e^{-t^2})$ satisfies the differential equation 
\begin{equation} \label{eq:sfconeHdiffO}
  \left[  \partial_t^2 - 2t \partial_t  + \frac{d-3}{t}  \left(\partial_t -  \frac{1}{t}\right) + \frac{1}{t^2} \Delta_0^{(x)} \right] u 
     = - 2n u,
\end{equation}
where $\Delta_0^{(x)}$ is the Laplace-Beltrami operator in variable $x \in \sph$.   
\end{thm}

\begin{proof}
For $\b = 0$, let $u (x,t)= \sC_{m,\ell}^n(x,t) = C_{n-m}^{(\g, m+\f{d-1}{2})}(t)  t^m  Y_\ell^m(\xi)$,
$\xi \in \sph$, and let $U (x,t) = u(\frac{x}{\sqrt{\g}}, \frac{t}{\sqrt{\g}})$. 
Replacing $\partial_t $ by $\sqrt{\g} \partial_t $ in
\eqref{eq:sfConeGdiff}, it follows that $U$ satisfy 
\begin{align*}
\left[ \left(1- \frac{t^2}{\g}\right)  \g \partial_t^2 - (2 \g+d) t \partial_t + \frac{d-1}{t} \g \partial_t + 
     \frac{\g}{t^2}   \Delta_0^{(x)} \right] U = - n(n+2\g+d-1)U.
\end{align*}
Dividing the above equation by $\g \sqrt{h_{m,n}^\sC}$ and taking the limit $\g \to \infty$, we obtain
by \eqref{eq:sfGtoH} that $\sH_{m,\ell}^n$ satisfies \eqref{eq:sfconeHdiff}. The proof of \eqref{eq:sfconeHdiffO}
is similar. 
\end{proof}

The remarks that we made right below Theorem \ref{thm:sfConeGdiff} on the Gegenbauer polynomials 
apply to our setting here. We note particularly that \eqref{eq:sfconeHdiff} holds for the Hermite
polynomials on the cone with $w_{0,\g}(t) = e^{-t^2}$, whereas \eqref{eq:sfconeHdiffO} 
holds for the generalized Hermite polynomials on the cone with $w_{-1}(t) = |t|^{-1}e^{-t^2}$,
which has a singularity at the origin. 

\subsubsection{Laguerre polynomials on the upper cone}
Let us compare our results with those developed on the upper cone studied recently in \cite{X19}. 
The Laguerre polynomials on the upper cone $\VV_{0,+}^{d+1}$ are orthogonal with respect to
$$
  \la f, g\ra_{\b} = b_\b \int_{\VV_{0,+}^{d+1}} f(x,t) g(x,t) t^\b e^{-t} \d \s(x,t), \qquad \b > -d.
$$
An orthogonal basis for the space $\CV_n(\VV_{0,+}^{d+1}, t^\b e^{-t})$ is given by 
\begin{equation}\label{eq:coneLsf}
    \sL_{m,\ell}^n(x,t) = L_{n-m}^{2m + \b+d-1}(t) Y_\ell^m(x), \quad 0 \le m \le n, \quad 1\le \ell \le \dim \CH_m^d,
\end{equation}
where $L_n^\a$ are the Laguerre polynomials and $\{Y_\ell^m\}$ is an orthonormal basis of $\CH_m^d$. It
is shown in \cite{X19} that, when $\b = -1$, the Laguerre polynomials in $\CV_n(\VV_{0,+}^{d+1}, t^{-1} e^{-t})$ 
satisfy a differential equation 
\begin{equation} \label{eq:sfConeLaguerrediff}
 \left( t \partial_t^2 + (d-1 - t) \partial_t + t^{-1} \Delta_0^{(x)} \right) u = - n u.
\end{equation}

Like the Jacobi polynomials on the cone surfaces, the Laguerre polynomials in \eqref{eq:coneLsf} are not 
related to orthogonal polynomials for the even weight function $|t|^\b e^{-|t|}$ on the double cones.

\subsubsection{Double hyperboloid} \label{subsec:sfHypHermtie}
We now consider the case $\varrho > 0$. The weight function $w_\b$ can be written as 
$$
w_\b(t) = |t| w_0 (t^2-\varrho^2), \qquad w_0(t) = e^{\rho^2} t^{\b-\f12} e^{-t}, \quad \b > -\tfrac12.
$$ 
Hence, by Proposition \ref{prop:op-1d},  the polynomials $q_k^{(m)}(t)$ in Proposition \ref{prop:sfOPbasis} 
are given in terms of $p_k(w_0^{(m)}; t)$ with 
$$
    w_0^{(m)}(t) = t^{m+ \f{d-1}{2}}w_0(t) = e^{-\rho^2} t^{m+ \b+ \f{d-2}{2}} e^{-t},
$$ 
which are the Laguerre polynomials; that is, $p_k(w_0^{(m)};s)= L_k^{m+\b+\f{d-2}{2}}(s)$. As it is for
the generalized Gegenbauer polynomials, we consider only orthogonal polynomials that are even in $t$, 
that is, those in $\CV_n^E(\VV_0^{d+1}; w_\b)$. 

The orthogonal polynomials given in Corollary \ref{cor:sfOPeven} are now specialized as follows: 

\begin{prop}\label{prop:sfOPhypH}
Let $\{Y_\ell^m: 1 \le \ell \le \dim \CH_n^d\}$ denote an orthonormal basis of $\CH_m^d$.
Then the polynomials 
\begin{equation}\label{eq:sfOPhypH}
 {}_\varrho\sH_{n-2k,\ell}^n (x,t) := L_k^{n-2k+\b+\f{d-2}{2}}(t^2-\varrho^2) Y_\ell^{n-2k}(x) 
\end{equation}
with $1 \le \ell \le \dim \CH_{n-2k}^d$ and $0 \le k\le n/2$, form an orthogonal basis of $\CV_n^E(\VV_{0}^{d+1}, w_\b)$. Furthermore, in terms of orthogonal polynomials $\sH_{n-2k,\ell}^n$ in \eqref{eq:sfOPconeH}, 
\begin{equation}\label{eq:sfOPhypH2}
 {}_\varrho\sH_{n-2k,\ell}^n(x,t) =  \frac{(-1)^k}{2^{2k} k!}  \sH_{n-2k,\ell}^n \left(x, \sqrt{t^2-\varrho^2}\right).
\end{equation}
\end{prop}

\begin{proof}
By Corollary \ref{cor:sfOPeven}, we only need to verify the relation \eqref{eq:sfOPhypH2}. Using
\eqref{eq:gHermite} to write $L_k^{\a-\f12}(s^2)$ as $H_{2k}^{\a}(s)$, this follows from the expression of  
$\sH_{n-2k,\ell}^n$ given in \eqref{eq:sfOPconeH}.
\end{proof}
 
We shall call these polynomials {\it Hermite polynomials on the hyperboloid} when $\b = 0$ and 
{\it generalized Hermite polynomials on the hyperboloid} when $\b \ne 0$. 

By \eqref{eq:sfOPhypG2} and \eqref{eq:sfOPhypH2}, it is easy to see that the limit relation \eqref{eq:sfGtoH} 
also holds for ${}_\varrho\sC_{n-2k,\ell}^n(t\xi,t)$ and ${}_\varrho\sH_{n-2k,\ell}^n(t\xi,t)$. 

\begin{thm}\label{thm:sfHypHdiff}
Let $\varrho > 0$ and $w_0(t) = |t| (t^2-\varrho^2)^{-\f12} e^{-t^2}$. For $n =0,1,2,\ldots$,  
every $u \in \CV_n^E(\VV_0^{d+1}, w_0)$ satisfies the differential equation 
\begin{equation} \label{eq:sfHypHdiff}
  \left[ \left(1- \frac{\varrho^2}{t^2} \right) \partial_t^2 + \left ( \frac{\rho^2}{t^2} - 2 (t^2-\rho^2) \right)\frac{1}{t} \partial_t
 + \frac{d-1}{t} \partial_t + \frac{1}{t^2- \varrho^2} \Delta_0^{(x)} \right] u = - 2n u,
\end{equation}
where $\Delta_0^{(x)}$ is the Laplace-Beltrami operator in variable $x \in \sph$. 
\end{thm}
 
\begin{proof}
Since the limit \eqref{eq:sfGtoH} extents to the hyperboloid by \eqref{eq:sfOPhypH2}, this follows from 
taking limit in Theorem \ref{thm:sfHypGdiff}. It can also be deduced from Theorem \ref{thm:sfconeHdiff} 
by following the proof of Theorem \ref{thm:sfHypGdiff}.
\end{proof}

\subsubsection{Poisson kernel on the surface of the cone}
In terms of the orthogonal basis \eqref{eq:sfOPconeH}, the reproducing kernel of the generalized 
Hermite polynomials on the cone satisfy 
$$
     \sP_n \big(w_\b; (x,t),(y,s)\big) = \sum_{m=0}^n \sum_{\ell =1}^{\dim \CH_{m}^n} 
         \frac{\sH_{m,\ell}^n (x,t)\sH_{m,\ell}^n (y,s)}{h^\sH_{m,n}}.
$$
The kernel $\sP_n^E(w_\b)$ has a similar expression, summing over $m$ such that $n-m$ is even. 
By \eqref{eq:sfPbnCone1}, the limit relation \eqref{eq:sfGtoH} leads immediately to the relation  
\begin{equation}\label{eq:sfPlimit}
   \lim_{\g \to \infty}  \sP_n^E\left(w_{\b,\g}; \left(\frac{x}{\sqrt{\g}},\frac{t}{\sqrt{\g}}\right),
       \left(\f{y}{\sqrt{\g}},\f{s}{\sqrt{\g}}\right)\right) = \sP_n^E\big(w_\b; (x,t),(y,s)\big). 
\end{equation}
However, we cannot derive an addition formula for the generalized Hermite polynomials from this relation
and the addition formula \eqref{eq:sfPEadd}. What we can derive is a Mehler-type formula for the 
Poisson kernel $\sP^E(w_{\b}; r, \cdot,\cdot)$ of the generalized Hermite polynomials, 
$$
\sP^E\big(w_{\b}; r,  (x,t),(y,s)\big) = \sum_{n=0}^\infty \sP_n^E(w_{\b}; (x,t),(y,s)\big) r^n, \quad 0 \le r < 1.
$$

\begin{thm}\label{thm:sfPoissonH}
Let $\b \ge 0$. Then, for $(x,t), (y,s) \in \VV_0^{d+1}$, 
\begin{align}\label{eq:sfPoissonH}
 \sP^E\big(w_{\b}; r,  (x,t),(y,s)\big)& = \frac{1}{(1-r^2)^{\b + \f{d}{2}}} 
       e^{-\frac{s^2+t^2}{1-r^2} r^2}   \int_{-1}^1 \int_{-1}^1 e^{\frac{2 r}{1-r^2} \eta(x,t,y,s;z)}\\
        & \times c_{\frac{d-2}{2},\b-1} c_\b (1-z_1)^{\frac{d-2}{2}} (1+z_1)^{\b-1}(1-z_2^2)^{\b-\f12} \d z, \notag
\end{align}
where $\eta(x,t,y,s;z)$ is defined by 
$$
  \eta (x,t,y,s; z) = \frac{1-z_1}{2}\la x,y\ra \mathrm{sign}(st) + \frac{1+z_1}{2} z_2 st.
$$
\end{thm}

\begin{proof}
We start with \eqref{eq:sfPEadd}. Its integrant is $ \eta(x,t,y,s; z) + \sqrt{1-s^2} \sqrt{1-t^2}v$ by our
definition of $\eta$. Changing variable $v \mapsto \xi(x,t,y,s,v,z)$ for the integral with resect to $\d v$ in 
\eqref{eq:sfPEadd}, we can write
\begin{equation}\label{eq:sfPossion1}
  \sP_n^E\big(w_{\b,\g};(x,t),(y,s)\big) = c_{\g-\f12}\int_{-1}^1  Z_n^{\a+\g}(v) G_\g(x,t,y,s;r) (1-r^2)^{\a+\g-\f12}\d r,
\end{equation}
where $\a = \b+ \f{d-1}{2}$ and, with $c' = c_{\frac{d-2}{2},\b-1} c_\b$,
$$
 G_\g(x,t,y,s;r)  = c' \int_{-1}^1 \int_{-1}^1 
    F_\g(x,t,y,s;r,z) (1-z_1)^{\frac{d-2}{2}} (1+z_1)^{\b-1}(1-z_2^2)^{\b-\f12} \d z
$$
with the function $F_\g$ defined by 
$$
F_\g(x,t,y,s;z,r) = \left(1- \frac{\big(r-\eta(x,t,y,s;z)\big)^2}{(1-s^2)(1-t^2)}\right)^{\g-1}
       \frac{1}{\sqrt{1-s^2}\sqrt{1-t^2} (1-r^2)^{\a+\g-\f12}}
$$
if $\big(r-\eta(x,t,y,s;z)\big)^2 \le (1-s^2)(1-t^2)$ and $F_\g(x,t,y,s;z,r) =0$ otherwise. Up to a constant, 
the righthand side of \eqref{eq:sfPossion1} is the Fourier-Gegenbauer coefficient of the function 
$r \mapsto G_\l(\cdot; r)$. By $h_n^\l = \frac{\l}{n+\l} C_n^\l(1)$, we then obtain
\begin{equation}\label{eq:sfPossion2}
 G_\g(x,t,y,s;r) = \frac{c_{\a+\g}}{c_{\g-\f12}} \sum_{n= 0}^\infty\sP_n^E\big(w_{\b,\g};(x,t),(y,s)\big) 
  \frac{Z_n^{\a+\g}(r)}{C_n^{\a+\g}(1)}.  
\end{equation}
Since the leading coefficient of $C_n^\l$ is $k_n^\l = 2^n (\l)_n /n!$, so that $C_n^{\l}(r) /C_n^{\l}(1) \to r^n$ 
if $\l \to \infty$, and $\frac{c_{\a+\g}}{c_{\g-\f12}} \to 1$ if $\g \to \infty$, it follows from \eqref{eq:sfPlimit} that 
$$
  \lim_{\g \to \infty} G_\g \left(\frac{x}{\sqrt{\g}},\frac{t}{\sqrt{\g}},\frac{y}{\sqrt{\g}},\frac{s}{\sqrt{\g}};r\right) 
    =  \sP^E\big(w_{\b}; r, (x,t),(y,s)\big). 
$$
Since $\eta(\frac{x}{\sqrt{\g}},\frac{t}{\sqrt{\g}},\frac{y}{\sqrt{\g}},\frac{s}{\sqrt{\g}}; z) = \frac{1}{\sqrt{\g}} 
\eta(x,t,y,s;z)$, we see that $F_\l$ satisfies
\begin{align*}
& F_\g \left(\frac{x}{\sqrt{\g}},\frac{t}{\sqrt{\g}},\frac{y}{\sqrt{\g}},\frac{s}{\sqrt{\g}}; z, r\right)
    =  \frac{1}{(1-r^2)^{\a+\f12} (1-\frac{s^2}{\g})^{\g-\f12} (1-\frac{t^2}{\g})^{\g-\f12}} \\
& \qquad\qquad\qquad
  \times \left(1- \frac{t^2+s^2 -2 r \eta(x,t,y,s;z)}{\g(1-r^2)} + \frac{t^2 s^2 - \eta(x,t,y,s;z)^2}{\g^2(1-r^2)} \right)^{\g-1},
\end{align*}
which has the limit, as $\g \to \infty$, 
$$
\frac{1}{(1-r^2)^{\a+\f12}} e^{t^2+s^2} e^{\frac{- (t^2+s^2 - 2 r \eta(x,t,y,s;z))}{1-r^2}} 
 = \frac{1}{(1-r^2)^{\a+\f12}}   e^{- \frac{(t^2+s^2)r^2 +2 r \eta(x,t,y,s;z)}{1-r^2}}. 
$$
Taking limit under the integral sign of $G_\g$, we obtain \eqref{eq:sfPoissonH} from \eqref{eq:sfPossion2}. 
\end{proof}

When $\b = 0$, the formula \eqref{eq:sfPoissonH} holds under the limit \eqref{eq:limit-int} and it reduces
to the following formula for the Poisson kernel of the Hermite polynomials on the cone: 

\begin{cor}
For $\b = 0$,  
\begin{align}
 \sP^E\big(w_0; r,  (x,t),(y,s)\big) = \frac{1}{(1-r^2)^{\f{d}{2}}} 
      \exp \left [- \frac{ (t^2 +s^2)r^2 - 2r \la x,y\ra \mathrm{sign}(st)} {1-r^2} \right].
\end{align}
\end{cor}
 
\subsubsection{Poisson kernel on the surface of the hyperboloid} 
For $\varrho > 0$, we can define the Poisson kernel on the hyperboloid by 
$$
{}_\varrho \sP^E\big(w_{\b}; r,  (x,t),(y,s)\big) = \sum_{n=0}^\infty 
    {}_\varrho\sP_n^E(w_{\b}; (x,t),(y,s)\big) r^n, \quad 0 \le r < 1.
$$
By \eqref{eq:sfOPhypH2} and the analogue of the integral \eqref{eq:intHyp} for the Hermite weight, 
we have
$$
 {}_\varrho\sP_n^E(w_{\b}; (x,t),(y,s)\big) =  \sP_n^E\left(w_{\b}; \big(x,\sqrt{t^2-\varrho^2}\big),
 \big(y,\sqrt{s^2-\varrho^2}\big)\right) 
$$
and the similar relation for the Poisson kernel. Consequently, the Mehler-type formula for the 
Poisson kernel on the hyperboloid follows from a simple changing variable of the identity 
\eqref{eq:sfPoissonH} for the cone.
 
\subsection{Hermite polynomials on a solid hyperboloid}
On the solid domain $\VV^{d+1}$ bounded by the hyperboloid, we choose the weight function $w$ as 
$$
     w_{\b}(t) = |t| (t^2-\varrho^2)^{\b-\f12} e^{-t^2},  \qquad \b > -\tfrac12, \quad |t| \ge \varrho \ge 0,
$$
which is an even function, so that $W$ in \eqref{eq:W(x,t)} with $c =1$ is given by 
$$
  W_{\b,\mu}(x,t) = |t| (t^2-\varrho^2)^{\b-\f12} e^{-t^2} (t^2-\varrho^2 - \|x\|^2)^{\mu-\f12}, \qquad  \mu > -\tfrac12.
$$
The corresponding inner product on the solid hyperboloid is defined by 
$$
 \la f,g\ra_{\b,\mu} = b_{\b,\mu} \int_{\VV^{d+1}} f(x,t)g(x,t) W_{\b,\mu}(x,t) \d x \d t, 
$$
where $b_{\b,\mu} = 1/\int_{\VV^{d+1}} W_{\b,\mu}(x,t) \d x \d t = 
e^{\varrho^2}b_{\mu}^\BB / \Gamma(\b+\mu+\f{d}{2})$ with $b_{\mu}^\BB$ being the normalization 
constant of $\varpi_\mu(x) = (1-\|x\|^2)^{\mu-\f12}$ on $\BB^d$. 

\subsubsection{Double cone}
In this case $\varrho = 0$, the weight function becomes 
$$
  W_{\b,\mu}(x,t) = |t|^{2\b} e^{-t^2} (t^2-\|x\|^2)^{\mu - \f12}, \quad \b > - \tfrac{d+1}2, \, \mu > - \tfrac12.
$$
The polynomial $q_k^{(m)}$ in Proposition \ref{prop:solidOPbasis} are orthogonal with respect to 
$$
     |t|^{2m+ 2\mu +d -1} w_\b(t) =  |t|^{2m+2\mu+2\b+d-1} e^{-t^2},
$$ 
so that they are given by the generalized Hermite polynomials defined in \eqref{eq:gHermite}, that is,
$g_k^{(m)}(t) = H_k^{m+\mu+\b+ \frac{d-1}{2}}(t)$. Thus, the orthogonal polynomials in 
$\CV_n(\VV^{d+1},W_{\b,\mu})$ given by \eqref{eq:solidOPbasis} are now specialized as follows: 

\begin{prop}\label{prop:solidOPHermite}
Let $\{P_\kb^m: |\kb| = m, \, \kb\in \NN_0^d\}$ be an orthonormal basis of $\CV_m^d(\varpi_\mu)$.
Then the polynomials 
\begin{equation}\label{eq:solidOPconeH}
  \Hb_{m,\kb}^n (x,t) = H_{n-m}^{m+\b + \mu+\frac{d-1}{2}}(t) t^{m} P_\kb^m\left( \frac{x}{t}\right),  
      \qquad 0 \le k \le n, \quad |\kb| = m,
\end{equation}
form an orthogonal basis of $\CV_n(\VV^{d+1}, W_{\b,\mu})$. Moreover, 
\begin{equation}\label{eq:solidOPconeHnorm}
  h_{m,n}^\Hb: = \la \Hb_{m,\kb}^n, \Hb_{m,\kb}^n\ra_{W_{\b,\mu}} = 
   (\b+\mu+\tfrac{d}2)_m h_{n-m}^{m+\b+\mu+ \frac{d-1}{2}}.
\end{equation}
\end{prop}

\begin{proof}
The polynomials in \eqref{eq:solidOPconeH} consist of an orthogonal basis by 
Proposition \ref{eq:solidOPbasis}. The norm $h_{m,n}^\Hb$ is computed using 
the identity \eqref{eq:intHypSolid} and the orthonormality of $P_{\kb}^m$, where the
ratio of the Pochhammer constant comes from $b_{\b,\mu}/b_{\b+m,\mu}$. 
\end{proof}
 
We shall call these polynomials {\it Hermite polynomials on the solid cone} when $\b = \f12$ 
and {\it generalized Hermite polynomials on the solid cone} when $\b \ne \f12$. The reason for
choose $\b = \f12$ instead of $\b =0$ lies in Theorem \ref{thm:solidConeHdiff} below. 

\begin{rem}\label{rem:4.2}
If $n-m$ is odd, then the polynomials $\Hb_{m,\kb}^n(x,t)$ contains a factor $t$ by \eqref{eq:gHermite}.
Consequently, the space $\CV_n^O(\VV^{d+1}, W_{\b,\mu})$ that contains these polynomials are well 
defined for $\b > - \tfrac{d + 2}2$ and, in particular, for $\b = -1$ for all $d \ge 1$.
\end{rem}

\begin{thm}\label{thm:solidGtoH}
Let $\Cb_{m,\kb}^n$ be the generalized Gegenbauer polynomials in $\CV_n(\VV^{d+1}, W_{\b,\g,\mu})$
defined at \eqref{eq:solidOPconeG}. Then
\begin{equation}\label{eq:GtoH}
\lim_{\g \to \infty} \frac{\Cb_{m,\kb}^n\left(\frac{x}{\sqrt{\g}}, \frac{t}{\sqrt{\g}} \right)}{\sqrt{h_{m,n}^\Cb}}
    = \frac{\Hb_{m,\kb}^n(x,t)}{\sqrt{h_{m,n}^\Hb}}.
\end{equation}
\end{thm}

\begin{proof}
The proof is similar to that of Theorem \ref{thm:sfGtoH}. We deduce analogously the limit of $h_{m,n}^{\Cb}$
from \eqref{eq:solidOPconeGnorm} and \eqref{eq:solidOPconeHnorm} and the limit of $\Cb_{m,\kb}^n$ 
from \eqref{eq:solidOPconeG} and \eqref{eq:solidOPconeH}. 
\end{proof}

Again, the limit relation \eqref{eq:GtoH} allows us to derive the differential equations satisfied by the 
Hermite polynomials on the solid cone. 

\begin{thm}\label{thm:solidConeHdiff}
Let $\mu > -\f12$ and $W_{\f12,\mu}(x,t) = |t| e^{-t^2} (t^2-\|x\|^2)^{\mu - \f12}$. 
For $n=0,1,2,\ldots$, every $u \in \CV_n^E\big(\VV_0^{d+1}, W_{\f12,\mu}\big)$ satisfies 
the differential equation 
\begin{align} \label{eq:solidConeHdiff}
    \left[ \partial_t^2 + \Delta_x +\frac{2}{t} \la x, \nabla_x\ra - 
   2 \big(t \partial_t + \la x, \nabla_x \ra\big)  + \frac{2\mu+d}{t} \partial_t \right] 
    u = - 2n \, u.  
\end{align}
Furthermore, every $u \in \CV_n^O\big(\VV_0^{d+1}, W_{- \f12,\mu}\big)$ satisfies 
the differential equation 
\begin{align} \label{eq:solidConeHdiffO}
     \left[ \partial_t^2 + \Delta_x +\frac{2}{t} \la x, \nabla_x\ra \left (\partial_t - \frac{1}{t}\right)
      - 2 \big(t \partial_t + \la x, \nabla_x \ra\big)   + \frac{2\mu+d-2}{t} \left (\partial_t - \frac{1}{t}\right) \right]u \\
     = - 2n \, u.  \notag
\end{align}
\end{thm}

\begin{proof}
The proof is similar to that of Theorem \ref{thm:sfconeHdiff}. We shall omit the details. 
\end{proof}

We note that the identity \eqref{eq:solidConeHdiff} holds for $\CV_n^E(\VV^{d+1}; W_{\b,\mu})$ with
$\b = \f12$, where \eqref{eq:solidConeHdiffO} holds for $\CV_n^O(\VV^{d+1}; W_{\b,\mu})$
with $\b = - \f12$. We have observed similar phenomenon for orthogonal polynomials on the surface 
of the cone. 

\subsubsection{Laguerre polynomials on the solid upper cone}
We compare our results with those developed on the solid upper cone studied recently in \cite{X19}. 
The Laguerre polynomials on the solid upper cone $\VV_{+}^{d+1}$ are orthogonal with respect to
$$
  \la f, g\ra_{\b,\mu} = b_{\b,\mu} \int_{\VV_{+}^{d+1}} f(x,t) g(x,t) W_{\b,\mu}^L(x,t) \d x \d t, 
$$
where $W_{\b,\mu}^L(x,t)= (t^2-\|x\|^2)^{\mu-\f12} t^\b e^{-t}$, $\b > -d, \, \mu > -\tfrac12$.
An orthogonal basis for the space $\CV_n(\VV_{+}^{d+1}, t^\b e^{-t})$ is given by 
\begin{equation}\label{eq:coneL}
   \Lb_{m,\kb}^n(x,t) =  L_{n-m}^{(2\a+2m)}(t) t^m P_\kb^m\left(\varpi_\mu; \frac{x}{t}\right),
\end{equation}
in terms of the Laguerre polynomials and an orthogonal basis of $\CV_m^d(\varpi_\mu)$ on the unit ball.
It is shown in \cite{X19} that, when $\b = 0$, the Laguerre polynomials in $\CV_n(\VV_{+}^{d+1}, W_{0,\mu}^L)$
satisfy a differential equation 
\begin{align} \label{eq:ConeLaguerrediff}
  \Big[ t \left( \Delta_x+ \partial_t^2 \right) + &\,  2  \la x,\nabla_x \ra \partial_t   - \la x,\nabla_x\ra  \\
        + &\,  (2\mu + d -t) \partial_t \left( t \partial_t^2 + (d-1 - t) \partial_t + t^{-1} \Delta_0^{(x)} \right) \Big]u = - n u.\notag
\end{align}
Note that the weight function $t^{-1} e^{-t}$ also has a singularity at the origin. 

Like the case of cone surfaces, the Laguerre polynomials in \eqref{eq:coneL} are not related to orthogonal 
polynomials for the even extension, in $t$ variable, of $W_{\b,\mu}^L$ on the double cones.  
 
\subsubsection{Double hyperboloid} \label{subsec:hyperSolidH}
With $\varrho > 0$, the weight function $W_{\b,\mu}$ can be written as $W_{\b,\mu}(t) = 
|t| w_0 (t^2-\varrho^2) (t^2-\varrho^2)^{\mu-\f12} \varpi_\mu(x/\sqrt{t^2-\varrho^2})$ with $w_0(t) = e^{\rho^2} 
t^{\b-\f12} e^{-t}$. Hence, the polynomials $q_k^{(m)}(t)$ in Proposition \ref{prop:solidOPbasis} 
are given in terms of $p_k(w_0^{(m)}; t)$ with 
$$
   w_0^{(m)}(t) = t^{m+ \f{d-1}{2}}w_0(t) = e^{-\rho^2} t^{m+ \b+ \f{d-2}{2}} e^{-t},
$$ 
which are the Laguerre polynomials; that is, $p_k(w_0^{(m)};s)= L_k^{m+\b+\f{d-2}{2}}(s)$. As in 
Subsection \ref{subsec:sfHypHermtie}, we only consider orthogonal polynomials that are even in $t$, that is, 
those in $\CV_n^E(\VV_0^{d+1}; W_{\b,\mu})$, in view of Proposition \ref{prop:op-1d}. The orthogonal 
polynomials given in Corollary \ref{cor:sfOPeven} are now specialized as follows: 

\begin{prop}\label{prop:solidOPhypH}
Let $\{P_\kb^{n-2k}: |\kb| = n-2k, \, \kb \in \NN_0^d\}$ denote an orthonormal basis of $\CV_{n-2k}^d(\varpi_\mu)$.
Then the polynomials 
\begin{equation}\label{eq:solidOPhypH}
 {}_\varrho \Hb_{n-2k,\kb}^n (x,t) = L_k^{n-2k+\b+\mu+\f{d-2}{2}}(t^2-\varrho^2) (t^2-\varrho^2)^{\frac{n-2k}{2}}
      P_\kb^{n-2k}\bigg(\frac{x}{\sqrt{t^2-\varrho^2}} \bigg) 
\end{equation}
with $|\kb| = n-2k$ and $0 \le k\le n/2$ form an orthogonal basis of $\CV_n^E(\VV^{d+1}, W_{\b,\mu})$. 
Furthermore, in terms of orthogonal polynomials $\Hb_{n-2k,\ell}^n$ in \eqref{eq:solidOPconeH}, 
\begin{equation}\label{eq:solidOPhypH2}
 {}_\varrho \Hb_{n-2k,\ell}^n(x,t) =  \frac{(-1)^k}{2^{2k} k!} \Hb_{n-2k,\ell}^n \left(x, \sqrt{t^2-\varrho^2}\right).
\end{equation}
\end{prop}

\begin{proof}
By Corollary \ref{cor:solidOPeven}, we only need to verify the relation \eqref{eq:sfOPhypH2}, which 
follows from \eqref{eq:gHermite} as in the proof of \eqref{eq:sfOPhypH2}.
\end{proof}
 
We call these polynomials {\it Hermite polynomials on the solid hyperboloid} when $\b = \frac12$, for
which the weight function is $w_{\f12,\mu}(t) = e^{\varrho^2} |t| e^{-t^2} (t^2 - \varrho^2 - \|x\|^2)^{\mu-\f12}$, 
and {\it generalized Hermite polynomials on the solid hyperboloid} when $\b \ne \f12$. 

We now show that the elements of $\CV_n^E(\VV^{d+1}, W_{\b,\g})$ are eigenfunctions 
of a second order differential operator when $\b = \f12$. 
 
\begin{thm}\label{thm:solidHypHdiff}
For  $\rho > 0$ and $n=0,1,2,\ldots$, every $u \in \CV_n^E(\VV_0^{d+1}, W_{\f12,\mu})$ satisfies 
the differential equation 
\begin{align} \label{eq:solidHypHdiff}
  & \left[ \left(1- \frac{\varrho^2}{t^2} \right) \partial_t^2 + \Delta_x +\frac{2}{t} \la x, \nabla_x\ra - 
   \frac{2}{t}(t^2-\varrho^2) \partial_t - 2 \la x, \nabla_x \ra \right. \\
  & \qquad \qquad \qquad  \qquad \qquad \qquad \qquad
    \left. + \frac{1}{t}\left (\frac{\varrho^2}{t^2} + 2\mu+d\right) \partial_t \right] 
    u = - 2n \, u. \notag
\end{align}
\end{thm}

\begin{proof}
This can be deduced either from \eqref{eq:solidConeHdiff} or from taking limit of \eqref{eq:solidHypGdiff}. 
\end{proof}
 
\subsubsection{Poisson kernel on the solid cone}
In terms of the orthogonal basis \eqref{eq:solidOPconeH}, the reproducing kernel of the generalized 
Hermite polynomials on the cone satisfy 
$$
     \Pb_n \big(W_{\b,\mu}; (x,t),(y,s)\big) = \sum_{m=0}^n \sum_{|\kb| = m}
         \frac{\Hb_{m,\kb}^n (x,t)\Hb_{m,\kb}^n (y,s)}{h^\Hb_{m,n}}.
$$
The kernel $\Pb_n^E(W_{\b,\mu})$ has a similar expression, summing over $m$ such that $n-m$ is even. 
By \eqref{eq:PbnCone1}, the limit relation \eqref{eq:sfGtoH} leads immediately to the relation  
\begin{equation}\label{eq:Plimit}
   \lim_{\g \to \infty}  \Pb_n^E\left(W_{\b,\g,\mu}; \left(\frac{x}{\sqrt{\g}},\frac{t}{\sqrt{\g}}\right),
       \left(\f{y}{\sqrt{\g}},\f{s}{\sqrt{\g}}\right)\right) = \Pb_n^E\big(W_{\b,\mu}; (x,t),(y,s)\big). 
\end{equation}
As in the case of cone surface, we can derive a Mehler-type formula for the Poisson kernel 
$\Pb^E(W_{\b,\mu}; r, \cdot,\cdot)$ of the generalized Hermite polynomials, 
$$
\Pb^E\big(W_{\b,\mu}; r,  (x,t),(y,s)\big) = \sum_{n=0}^\infty \Pb_n^E(W_{\b,\mu}; (x,t),(y,s)\big) r^n, \quad 0 \le r < 1.
$$

\begin{thm} \label{thm:solidPoissonH}
Let $\b \ge \f12$. Then, for $(x,t), (y,s) \in \VV^{d+1}$, 
\begin{align}\label{eq:PoissonH}
& \Pb^E  \big(W_{\b,\mu}; r,  (x,t),(y,s)\big) = \frac{1}{(1-r^2)^{\b +\mu+ \f{d}{2}}} 
       e^{-\frac{s^2+t^2}{1-r^2} r^2}   \int_{[-1,1]^3} e^{\frac{2 r}{1-r^2} \phi(x,t,y,s;u,z)}\\
      \times & c_{\mu+\frac{d-1}{2},\b-\f32} c_{\b-\f12} c_{\mu -\f12} 
      (1-z_1)^{\mu+\frac{d-1}{2}} (1+z_1)^{\b-\f32}(1-z_2^2)^{\b-1}(1-u^2)^{\mu-1}  \d z \d u, \notag
\end{align}
where $\phi(x,t,y,s;u,z)$ is defined by 
$$
\phi (x,t,y,s; u,z) = 
   \frac{1-z_1}{2}\left(\la x,y\ra + u \sqrt{t^2-\|x\|^2}\sqrt{t^2-\|y\|^2}\right)\mathrm{sign}(st) + \frac{1+z_1}{2} z_2 st.
$$
\end{thm}

\begin{proof}
The integrant of \eqref{eq:PEadd} is $\phi(x,t,y,s;u,z) + \sqrt{1-s^2} \sqrt{1-t^2}v$ by our definition of $\phi$. 
Changing variable $v \mapsto \zeta(x,t,y,s; u,v,z)$ for the integral with resect to $\d v$ in \eqref{eq:PEadd}, 
we can write
\begin{equation*}
  \Pb_n^E\big(W_{\b,\g,\mu};(x,t),(y,s)\big) =
   c_{\g-\f12}\int_{-1}^1  Z_n^{\a+\g}(v) G_\g(x,t,y,s;u,r) (1-r^2)^{\a+\g-\f12}\d r,
\end{equation*}
where $\a = \b +\mu+\frac{d-1}{2}$,   
\begin{align*}
 G_\g(x,t,y,s;u,r)  = &\,  c_{\frac{d-1}{2},\b-\f32} c_{\b-\f12} c_{\mu-\f12} \int_{[-1,1]^3} 
    F_\g(x,t,y,s;u,z,r) \\ 
      & \, \times (1-z_1)^{\frac{d-1}{2}} (1+z_1)^{\b-\f32}(1-z_2^2)^{\b-1}(1-u^2)^{\mu-1}  \d z \d u
\end{align*}
and the function $F_\g$ is defined by 
$$
F_\g(x,t,y,s;u,z,r) = \left(1- \frac{\big(r-\phi(x,t,y,s;u,z)\big)^2}{(1-s^2)(1-t^2)}\right)^{\g-1} \!\!
       \frac{1}{\sqrt{1-s^2}\sqrt{1-t^2} (1-r^2)^{\a+\g-\f12}}
$$
if $\big(r-\eta(x,t,y,s;u,z)\big)^2 \le (1-s^2)(1-t^2)$ and $F_\g(x,t,y,s;u,z,r) =0$ otherwise. With this
set up, the proof follows exactly as in the proof of Theorem \ref{thm:sfPoissonH}.
\end{proof}

\subsubsection{Poisson kernel on the solid hyperboloid} 
For $\varrho > 0$, we define the Poisson kernel on the solid hyperboloid by 
$$
{}_\varrho \Pb^E\big(W_{\b,\mu}; r,  (x,t),(y,s)\big) = \sum_{n=0}^\infty 
    {}_\varrho\Pb_n^E(W_{\b,\mu}; (x,t),(y,s)\big) r^n, \quad 0 \le r < 1.
$$
As it is for the kernel on the surface hyperboloid, we can deduce from \eqref{eq:solidOPhypH2} that
$$
 {}_\varrho\Pb_n^E(W_{\b,\mu}; (x,t),(y,s)\big) =  \Pb_n^E\left(W_{\b,\mu}; \big(x,\sqrt{t^2-\varrho^2}\big),
 \big(y,\sqrt{s^2-\varrho^2}\big)\right) 
$$
and the similar relation for the Poisson kernel. Consequently, the Mehler-type formula for the 
Poisson kernel on the hyperboloid follows from the identity 
\eqref{eq:PoissonH} for the cone.

\section{Further extensions of  Gegenbauer and Hermite polynomials}
\setcounter{equation}{0}

In recent studies of orthogonal structure on the unit sphere and on the unit ball, a large portion of the
classical results are extended to the setting that the orthogonality is defined with respect to Dunkl's 
family of weight functions that are invariant under a reflection group. 

For the unit sphere, this means replacing the orthogonality with respect to $\d \s$ by 
\begin{equation} \label{eq:Dunkl-hk}
  \la f,g\ra_\k = a_\k \int_\sph f(x) g(x) h_\k^2(x) \d\s(x), \qquad h_\k(x) = \prod_{v\in R_+} |\la x,v\ra|^{\k_v},
\end{equation}
where $R_+$ is a set of positive roots for a reflection group $G$ with a reduced root system $R$ and 
$v \mapsto \k_v$ is a nonnegative multiplicity function defined on $R$ with the property that it is a 
constant on each conjugate class of $G$. In this setting, the partial derivative is replaced by the Dunkl
operator \cite{D89} defined by 
\begin{equation} \label{eq:Dunkl-Op}
    D_i f(x) = \partial_i  f(x) + \sum_{v \in R_+} \k_v \frac{f(x) - f (x \s_v)}{\la x,\s_v \ra} v_i,  \qquad i =1,2 , \ldots,d,
\end{equation}
where $x \s_v := x - 2\la x, v \ra v /\|v\|^2$, which commute in the sense that $D_i D_j = D_j D_i$ for $1 \le i, j \le d$. 
The operator 
$$
\Delta_h = D_1^2+ \cdots + D_d^2
$$
plays the role of the Laplace operator. Its restriction on the sphere, denoted by $\Delta_{h,0}$, is an analog
of the Laplace--Beltrami operator. An $h$-harmonics is a homogeneous polynomial that satisfies $\Delta_h Y =0$. 
The restriction of $h$-harmonics are orthogonal with respect to $\la \cdot, \cdot \ra_\k$. Let $\CH_n^d(h_\k^2)$ 
be the space of $h$-harmonics of degree $n$. These polynomials share the two characteristic properties of
ordinary spherical harmonics. They are the eigenfunctions of $\Delta_{h,0}$, 
\begin{equation} \label{eq:Delta_h0}
   \Delta_{h,0} Y = - n (n+2 \l_\k) Y, \qquad \l_\k := |\k| + \tfrac{d-2}{2},
\end{equation}
where $|\k| = \sum_{v\in R_+} \k_v$ and the reproducing kernel $\Pb_n(h_\k^2; \cdot,\cdot)$ of $\CH_n^d(h_\k^2)$
satisfies 
\begin{equation} \label{eq:Pb-hk}
\Pb_n(h_\k^2; x, y) = V_\k \left[Z_n^{\l_\k} (\la \cdot, y\ra)\right](x), \qquad x,y \in \sph,
\end{equation}
where $Z_n^\l$ is defined in \eqref{eq:Zn} and $V_\k$, called the intertwining operator, is a linear operator
that satisfies $ D_i V_\k  = V_\k \partial_i$, $1 \le i \le d$, and $V_\k 1 = 1$. 
The simplest nontrivial case is $G= \ZZ_2^d$, for which the weight function $h_\k$ is given by 
$h_\k(x) = \prod_{i=1}^d |x_i|^{\k_i}$, $\k_i \ge 0,
$
and the intertwining operator is a multiple beta integral so that \eqref{eq:Pb-hk} becomes \cite{X97}
\begin{equation} \label{eq:PbZ2d}
     \Pb_n(h_\k^2; x, y) = c_\k^h \int_{[-1,1]^d} C_n^{\l_\k} (x_1 y_1 t_1 + \cdots + x_d y_d t_d) 
        \prod_{i=1}^d (1+t_i)(1-t_i^2)^{\k_i-1} \d t,
\end{equation}
where $c_k^h = c_{\k_1-\f12} \cdots c_{\k_d-\f12}$, which holds under the limit \eqref{eq:limit-int} if any
$\k_i=0$.

\subsection{Orthogonal polynomials on the surface of a hyperboloid}
Using $h$-spherical harmonics to replace spherical harmonics, we can study orthogonal polynomials on
the surface of the double cone or the hyperboloid for the inner product 
$$
  \la f,g\ra = b_\k \int_{\VV_0^{d+1}} f(x,t) g(y,s) h_\k^2(x) w(t) \d x \d t.
$$
With slight modification, what we have done in the previous sections can be extended to this more general 
weight function. Most of the proof carries over with only slight modification. We state several key results  
below without proof. 

For $\b \ge 0$, $\g > -\f12$ and $h_\k$ being the reflection invariant weight function, define
$$
   w_{\b,\g,\k}(x,t) = h_\k^2(x) w_{\b,\g}(t) = h_\k^2(x) |t| (t^2-\varrho^2)^{\b-\f12} (1+\varrho^2 -t^2)^{\g-\f12}. 
$$
We then obtain a generalization of the generalized Gegenbauer polynomials. Since $h_\k$ is a homogeneous
function, $h_\k^2(t\xi) = |t|^{2|\k|} h_\k^2(\xi)$, $\xi \in \sph$; it is easy to see that an orthogonal basis
of $\CV_n(\VV_0^{d+1}, w_{\b,\g,\k})$ is given by, for $\varrho = 0$, 
\begin{equation}\label{eq:sfOPconeGhk}
  \sC_{m,\ell}^n (x,t) = C_{n-m}^{(\g, m+|\k|+\b + \frac{d-1}{2})}(t) Y_\ell^{m}(x), 
  \quad 1 \le \ell \le \dim \CH_m^d(h_\k^2), \quad 0 \le m \le n,
\end{equation}
where $\{Y_\ell^m: 1 \le \ell \le \dim \CH_m^d(h_\k^2)\}$ is an orthonormal basis $\CH_m^d(h_\k^2)$. 
We can define the basis for $\varphi > 0$ similarly and the relation \eqref{eq:sfOPhypG2} also holds.

In particular, these polynomials are even in the variable $t$ if $n-m$ is even. Replacing $\b$ by $\b + |\k|$
in the proof of Theorem \ref{thm:sfConeGdiff} and using \eqref{eq:Delta_h0} instead of \eqref{eq:sph-harmonics},
we obtain:

\begin{thm}\label{thm:sfConeGdiff-hk}
For $n=0,1,2,\ldots$, every $u \in \CV_n^E(\VV_0^{d+1},(w_{0,\g,\k})$ satisfies  
\begin{align} \label{eq:sfConeGdiff-hk}
  & \left[ (1-t^2) \partial_t^2 - (2 \g+2|\k|+ d)  t \partial_t + \frac{2|\k|+ d-1}{t} \partial_t 
    + \frac{1}{t^2} \Delta_{h,0}^{(x)} \right] u \\
    & \qquad\qquad  \qquad\qquad \qquad\qquad \qquad\qquad   = - n(n+2\g+ 2|\k|+d-1)u. \notag
\end{align}
\end{thm}

We can also stay analogues of \eqref{eq:sfConeGdiffO} and \eqref{eq:sfHypGdiff}. Likewise, we can
derive an addition formula for these Gegenbauer polynomials in this more general setting. 

\begin{thm} \label{thm:sfPbEInt-hk}
For $\b, \g \ge 0$, let $\a = \b+|\k| + \frac{d-1}{2}$. Then, for $n=0,1,2\ldots$, 
\begin{align} \label{eq:sfPEadd-hk}
  \sP_n^E \big (w_{\b,\g,\k}; & (x,t),(y,s) \big)
     =  c \int_{[-1,1]^3} V_\k \left[ Z_n^{\a+\g} \big(  \xi(\cdot,t,y,s; v,z)  \big)\right] (x) \\
       &  \times (1-z_1)^{|\k|+\frac{d-2}{2}}(1+z_1)^{\b-1}  (1-z_2^2)^{\b-\f12} (1-v^2)^{\g-1} \d v  \d z, \notag
\end{align}
where  $c=  c_{|\k|+\frac{d-1}{2},\b-1} c_{\b} c_{\g-\f12}$ and $\xi(x,t,y,s;v,z)$ is defined as in Theorem \ref{thm:sfPbEInt}.
\end{thm}

We can also consider generalized Hermite polynomials with $h_\k$ as part of the weight function. For 
$\b > -\f12$, let 
$$
   w_{\b,\k}(x) = h_\k^2(x) |t| (t^2 - \varrho^2)^{\b-\f12} e^{-t^2}, \qquad (x,t) \in \VV_0^d.
$$
The corresponding orthogonal polynomials in the space $\CV_n\big(\VV_0^{d+1},  w_{\b,\k}\big)$ are given by 
\begin{equation}\label{eq:sfOPconeHhk}
  \sH_{m,\ell}^n (x,t) = H_{n-m}^{m+|\k|+\b + \frac{d-1}{2}}(t) Y_\ell^{m}(x), 
  \quad 1 \le \ell \le \dim \CH_m^d(h_\k^2), \quad 0 \le m \le n,
\end{equation}
where $\{Y_\ell^m: 1 \le \ell \le \dim \CH_m^d(h_\k^2)\}$ is an orthonormal basis $\CH_m^d(h_\k^2)$. 
It follows readily that these polynomials are the limits of those in \eqref{eq:sfOPconeH}; that is,
\eqref{eq:sfGtoH} remains holds, from which we can derive an analogue of Theorem \ref{thm:sfconeHdiff}
as well as that of Theorem \ref{thm:sfHypHdiff}. We state only one result, an analogue of the Mehler-type 
formula. 

\begin{thm}\label{thm:sfPoissonHhk}
Let $\b \ge 0$. Then, for $(x,t), (y,s) \in \VV_0^{d+1}$, 
\begin{align*}
 \sP^E\big(w_{\b,\k}; r, (x,t),(y,s)\big)  = & \frac{1}{(1-r^2)^{\b + |\k|+ \f{d}{2}}} 
e^{-\frac{s^2+t^2}{1-r^2} r^2}   \int_{-1}^1 \int_{-1}^1  V_\k\Big[ e^{\frac{2 r}{1-r^2} \eta(\cdot,t,y,s;z)}\Big](x)\\
        & \times c_{\frac{d-2}{2},\b-1} c_\b (1-z_1)^{\frac{d-2}{2}} (1+z_1)^{\b-1}(1-z_2^2)^{\b-\f12} \d z, \notag
\end{align*}
where $\eta(x,t,y,s;z)$ is defined as in Theorem \ref{thm:sfPoissonH}.
\end{thm}

\subsection{Orthogonal polynomials on a solid hyperboloid}
Following the study in the previous subsection, we can study orthogonal polynomials on the solid 
double cone or the hyperboloid for the inner product 
$$
  \la f,g\ra = b_\k \int_{\VV^{d+1}} f(x,t) g(y,s) h_\k^2(x) W(x,t) \d x \d t,
$$
where $W(x,t) = w(t) (c(t^2-\varrho^2) - \|x\|^2)^{\mu-\f12}$. In this case, we will need results on 
generalized classical orthogonal polynomials on the unit ball $\BB^d$ that are orthogonal withe 
respect to 
$$
  \varpi_{\k,\mu}(x) = h_\k^2(x)(1-\|x\|^2)^{\mu-\f12}, \qquad \k \ge 0, \quad \mu > -\tfrac12.   
$$
A basis of $\CV_n^d(\varpi_{\k,\mu})$ is given explicitly in terms of the Jacobi polynomials and $h$-spherical
harmonics and a closed form formula of its reproducing kernel can be given in terms of an integral of 
$V_\k Z_n^\l$ \cite[Section 8.1.2]{DX}. Using orthogonal polynomials of $\CV_n^d(\varpi_{\k,\mu})$, we
can extend what we have done in the previous sections to $h_\k^2(x)W(x,t)$ with only slight modification. 
Again, we state several key results below without proof. 

For $\b \ge 0$, $\g, \mu > -\f12$ and $h_\k$ being the reflection invariant weight function, define
$$
   W_{\b,\g,\k,\mu}(x,t) = |t| (t^2-\varrho^2)^{\b-\f12} (1+\varrho^2 -t^2)^{\g-\f12}  h_\k^2(x) (1-\|x\|^2)^{\mu-\f12}. 
$$
The associated orthogonal polynomials are a generalization of the generalized Gegenbauer polynomials on the
solid domains. It is easy to see that an orthogonal basis of $\CV_n(\VV^{d+1}, W_{\b,\g,\k,\mu})$ is given by, 
for $\varrho = 0$, 
\begin{equation}\label{eq:OPconeGhk}
  \Cb_{m,\kb}^n (x,t) = C_{n-m}^{(\g, m+|\k|+\b + \mu+ \frac{d-1}{2})}(t) t^m P_\kb^{m} \left(\frac{x}{t}\right), 
  \quad |\kb| = m, \quad 0 \le m \le n,
\end{equation}
where $\{P_{\kb}^m:|\kb|=m\}$ denotes an orthonormal basis of $\CV_m^d(\varpi_{\k,\mu})$ on $\BB^d$. 
We can also state analogous basis of \eqref{eq:solidOPhypG} and see that an analogue of 
\eqref{eq:solidOPhypG2} holds. In particular, these polynomials are even in the variable $t$ if $n-m$ is even. 
The polynomials of 
$\CV_m^d(\varpi_{\k,\mu})$ satisfy the differential-difference equation
\begin{equation}\label{eq:diffBall-hk}
   \left(\Delta_h - \la x ,\nabla \ra^2 - 2 \l_{\k,\mu} \la x \nabla\ra \right) u = - m (m + 2\l_{\k,\mu}) u, \quad
     \l_{\k,\mu} = | \k|+\mu+\tfrac{d-1}2. 
\end{equation}
Replacing $\b$ by $\b + |\k|$ in the proof of Theorem \ref{thm:solidConeGdiff} and using \eqref{eq:diffBall-hk}
in stead of \eqref{eq:diffBall}, we can state a result on eigenfunctions for $\varrho = 0$:

\begin{thm}\label{thm:solidConeGdiff-hk}
For $n=0,1,2,\ldots$, every $u \in \CV_n^E(\VV^{d+1}, W_{\f12,\g,\k,\mu})$ satisfies   
\begin{align} \label{eq:solidConeGdiff-hk}
&   \Big [(1-t^2) \partial_{t}^2 + \Delta_h^{(x)} - \la x,\nabla^{(x)} \ra^2 +\frac{2}{t} (1-t^2)\la x, \nabla^{(x)}\ra \partial_t 
    + (2 \g+ 2\mu+d)\frac{1}{t} \partial_t  \\
   &     -  t \partial_t -(2\g+2|\k|+ 2\mu+d)\left( t \partial_t +\la x ,\nabla^{(x)}\ra\right)  \Big] u  
       = -n(n + 2 \g + 2 \mu+d) u. \notag
\end{align}
\end{thm}

We can also stay analogous results of \eqref{eq:solidConeGdiffO} and \eqref{eq:solidHypGdiff}. Moreover,
we can derive an addition formula for these Gegenbauer polynomials in this more general setting. 

\begin{thm} \label{thm:PbEInt-hk}
For $\b, \g \ge 0$, let $\a = \b+|\k| +\mu + \frac{d-1}{2}$. Then, for $n=0,1,2\ldots$, 
\begin{align} \label{eq:PEadd-hk}
 & \Pb_n^E \big (W_{\b,\g,\k,\mu};  (x,t),(y,s) \big)
     =  c \int_{[-1,1]^4} V_\k \left[ Z_n^{\a+\g} \big(  \zeta(\cdot,t,y,s;u, v,z)  \big)\right] (x) \\
       &  \times (1-z_1)^{|\k|+\mu+\frac{d-1}{2}} (1+z_1)^{\b-\f32}  (1-z_2^2)^{\b-1} (1-v^2)^{\g-1} (1-u^2)^{\mu-1} \d u 
       \d v  \d z, \notag
\end{align}
where  $c=  c_{|\k|+\frac{d-1}{2},\b-\f32} c_{\b-\f12} c_{\g-\f12}c_{\mu-\f12}$ and 
$\zeta$ is defined as in Theorem \ref{thm:PbEInt}.
\end{thm}

We can also consider generalized Hermite polynomials with $h_\k$ as part of the weight function. For 
$\b > -\f12$, let 
$$
   W_{\b,\k,\mu}(x) = |t| (t^2 - \varrho^2)^{\b-\f12} e^{-t^2}h_\k^2(x)(t^2- \varrho^2 -\|x\|^2)^{\mu-\f12},
    \qquad (x,t) \in \VV^{d+1}.
$$
The corresponding orthogonal polynomials in the space $\CV_n\big(\VV^{d+1},  W_{\b,\k,\mu}\big)$ are given by 
\begin{equation}\label{eq:OPconeHhk}
  \Hb_{m,\ell}^n (x,t) = H_{n-m}^{m+\b+ |\k|+\mu + \frac{d-1}{2}}(t) t^m P_\kb^m \left(\frac{x}{t} \right), 
  \quad |\kb|=m, \quad 0 \le m \le n,
\end{equation}
where $\{P_\kb^m: |\kb| =m\}$ is an orthonormal basis $\CV_m^d(\varpi_{\k,\mu})$.  
It follows readily that these polynomials are the limits of those in \eqref{eq:solidOPconeH}; that is,
\eqref{eq:GtoH} remains holds, from which we can derive an analogue of Theorem \ref{thm:solidConeHdiff}
as well as that of Theorem \ref{thm:solidHypHdiff}. We again state only one result for the Poisson kernel. 

\begin{thm}\label{thm:PoissonH-hk}
Let $\b, \mu \ge 0$ and $\l = \b +|\k|+\mu+ \f{d}{2}$. Then, for $(x,t), (y,s) \in \VV^{d+1}$, 
\begin{align*}
& \Pb^E  \big(W_{\b,\k,\mu}; r,  (x,t),(y,s)\big) = \frac{1}{(1-r^2)^\l} 
       e^{-\frac{s^2+t^2}{1-r^2} r^2}   \int_{[-1,1]^3} V_\k \Big[e^{\frac{2 r}{1-r^2} \phi(\cdot,t,y,s;u,z)}\Big](x)\\
     & \,\, \times c_{\mu+\frac{d-1}{2},\b-\f32} c_{\b-\f12} c_{\mu -\f12} 
      (1-z_1)^{\mu+\frac{d-1}{2}} (1+z_1)^{\b-\f32}(1-z_2^2)^{\b-1}(1-u^2)^{\mu-1}  \d z \d u, \notag
\end{align*}
where $\phi(x,t,y,s;u,z)$ is defined as in Theorem \ref{thm:solidPoissonH}. 
\end{thm}

\appendix
\section{Generalized  Gegenbauer and Hermite polynomials}\label{sect:append}
\setcounter{equation}{0}

We collect properties of these polynomials that we need in this section. Our main references 
are \cite{DX} and \cite{X97}. 

\subsection{Generalized Gegenbauer polynomials}
For $\l,\mu > -\f12$, the generalized Gegenbauer polynomials $C_n^{(\l,\mu)}$ satisfy the orthogonal relation
$$
     \frac{\Gamma(\l+\mu)}{\Gamma(\l+\f12)\Gamma(\mu+\f12)} \int_{-1}^1 C_n^{(\l,\mu)}(x) C_m^{(\l,\mu)}(y) 
        |x|^{2\mu} (1-x^2)^{\l -\f12} \d x  =  h_n^{(\l,\mu)} \delta_{m,n}.  
$$
The polynomials $C_n^{(\l,\mu)}$ are given explicitly by \cite[Section 1.5.2]{DX}  
\begin{align} \label{eq:gGegen}
\begin{split}
C_{2m}^{(\l,\mu)}(t) &\, = \frac{(\l+\mu)_m}{(\mu+\f12)_m} P_m^{(\l-\f12,\mu-\f12)}(2 t^2 -1),\\
C_{2m+1}^{(\l,\mu)}(t) &\, = \frac{(\l+\mu)_{m+1}}{(\mu+\f12)_{m+1}} t P_m^{(\l-\f12,\mu+\f12)}(2 t^2 -1),
\end{split}
\end{align}
where $P_n^{(a,b)}$ are the standard Jacobi polynomials. The norm square of these polynomials
are equal to \cite[p. 26]{DX}
\begin{align} \label{eq:gGegenNorm}
\begin{split}
h_{2m}^{(\l,\mu)} &\, = \frac{(\l+\f12)_m(\l+\mu)_m(\l+\mu)}{m!(\mu+\f12)_m(\l+\mu+2m)},\\  
 h_{2m+1}^{(\l,\mu)}&\,  =  \frac{(\l+\f12)_m(\l+\mu)_{m+1}(\l+\mu)}{m!(\mu+\f12)_{m+1}(\l+\mu+2m+1)}. 
\end{split}
\end{align}
Furthermore, in terms of the evaluation of the polynomials at $t=1$,  
\begin{equation} \label{eq:gGegen@1}
   h_{n}^{(\l,\mu)} = \frac{\l+\mu}{n+\l+\mu} C_{n}^{(\l,\mu)}(1). 
\end{equation}

The generalized Gegenbauer polynomial $C_n^{(\a,\g)}$ satisfies \cite[Thm 8.1.3]{DX} 
$$
  \Delta_h f - t (t f')' - (2\a+2\g) t f' = -n (n + 2 \a + 2 \g) f,
$$
where $\Delta_h$ is the Dunkl operator associated with the group $G=\ZZ_2$. Written explicitly, this 
equation becomes 
\begin{align} \label{eq:gGegenDiff}
(1-t^2)f''(t) - (2\a+2\g+1) t f'(t) +&\, \a \left( \frac {2 f'(t)}{t} - \, \frac{f(t) - f(-t)}{t^2} \right)  \\
  &  = - n (n+2\a+2\g)f(t).  \notag
\end{align}

The polynomial $C_n^{(\l,\mu)}$ is closely related to the Gagenbauer polynomial $C_n^\l$. Indeed,
\begin{equation} \label{eq:gGegenIntertw}
   C_n^{(\l,\mu)}(x) =   c_{\mu-\f12} \int_{-1}^1 C_n^{\l+\mu}(s t )(1+t)(1-t^2)^{\mu-1} \d t.
\end{equation}
Recall that $Z_n^\l(x) = \frac{n+\l}{\l} C_n^\l(x)$. The following lemma plays an essential role in deriving 
the addition formula on the double cone and hyperboloid: 

\begin{lem}
For $\l,\mu > 0$, 
\begin{align} \label{eq:additionGG}
  c_{\l-\f12} & \int_{-1}^1 Z_n^{\l+\mu} \Big(u s  t + v \sqrt{1-s^2}\sqrt{1-t^2}\Big) (1-v^2)^{\l-1} \d v \\
&  = \sum_{k=0}^{\lfloor \frac{n}{2} \rfloor} \frac{(\l+\mu+1)_{n-2k}}{(\mu+\f12)_{n-2k}} (s t)^{n-2k}
  \frac{C_{2k}^{(\l, \mu+n-2k)}(s) C_{2k}^{(\l, \mu+n-2k)}(t)}{h_{2k}^{(\l, \mu+n-2k)}}
               Z_{n-2k}^{\mu-\f12}(u). \notag
\end{align}
\end{lem}  

This lemma follows from the addition formula for the Gegenbauer polynomials 
\begin{align*}
   C_n^{\l+\mu} (u \cos \t \cos \phi + v \sin \t \sin \phi)& = 
     \sum_{k=0}^{\lfloor \frac{n}{2} \rfloor} \sum_{i+j= n-2k} b_{i,j}^n 
         (\cos \t \cos \phi)^i (\sin \t \sin \phi)^j \\
        \times & C_{n-i-j}^{(\l + j, \mu+i)}(\cos \t)C_{n-i-j}^{(\l + j, \mu+i)}(\cos \phi)
             C_i^{\mu-\f12}(u) C_j^{\l-\f12}(v), 
\end{align*}
where $b_{i,j}^n$ is a constant given by
$$
   b_{i,j}^n =\frac{1}{h_{n-i-j}^{(\l + j, \mu+i)}} \cdot \frac{(\l+\mu) (\l+\mu+1)_{i+j}}{ (n+\l+\mu) (\l-\f12)_j(\mu-\f12)_i}.
$$
This formula was first established in \cite[p. 242, (4.7)]{K73} using a group theoretic method but with the 
constants $b_{k,j}^n$ undetermined. The current form, written in terms of the generalized Gegenbauer polynomials, 
was derived in \cite[Thm. 2.3]{X97}. Integrating this identity in $v$ variable and rearranging the constant, we
obtain
\begin{align*} 
  c_{\l-\f12} & \int_{-1}^1  C_n^{\l+\mu} (u \cos \t \cos \phi + v \sin \t \sin \phi) (1-v^2)^{\l-1} \d v \\
  & = \sum_{k=0}^{\lfloor \frac{n}{2} \rfloor} b_{n-2k,0}^n (\cos \t \cos \phi)^{n-2k}
     C_{2k}^{(\l, \mu+n-2k)}(\cos \t) C_{2k}^{(\l, \mu+n-2k)}(\cos \phi) C_{n-2k}^{\mu-\f12}(u). 
\end{align*}
Setting $s = \cos \t$ and $t = \cos \phi$, and rearranging constants, this becomes \eqref{eq:additionGG}.

\subsection{Generalized Hermite polynomials}
For $\mu \ge 0$, the generalized Hermite polynomials $H_n^\mu$ satisfy the orthogonal relation
$$
 \frac{1}{\Gamma(\mu+\f12)} \int_\RR H_n^\mu(t) H_m^\mu(t) |t|^{2\mu} e^{-t^2} \d t
   = \delta_{m,n} h_n^\mu.
$$ 
The polynomials $H_n^\mu$ are given explicitly by \cite[Section 1.5.1]{DX}
\begin{align} \label{eq:gHermite}
\begin{split}
  H_{2m}^{\mu}(t) &\, =  (-1)^m 2^{2m} m! L_m^{\mu-\f12}(t^2),\\
  H_{2m+1}^{\mu}(t) &\, = (-1)^m 2^{2m+1} m! t L_m^{\mu+\f12}(t^2), 
\end{split}
\end{align}
where $L_n^\a$ are the Laguerre polynomials.  Furthermore, the norm square of these polynomials
are equal to \cite[p. 24]{DX}
\begin{align} \label{eq:gHermiteNorm}
   h_{2m}^\mu = 2^{4m} m! (\mu+\tfrac12)_m\quad\hbox{and}\quad h_{2m+1}^\mu
       = 2^{4m+2} m! (\mu+\tfrac12)_{m+1}. 
\end{align}
The generalized Hermite polynomials satisfy the differential-difference equation
\begin{align*}
   f''(t) - 2 t f'(t) + \mu \left( \frac {2 f'(t)}{t} - \, \frac{f(t) - f(-t)}{t^2} \right) = - 2n f(t).  
\end{align*}

When $\mu  = 0$, $H_n^\mu$ is the usual Hermite polynomial $H_n$. The generalized Hermite polynomial
$H_n^\mu$ are limits of the generalized Gegenbauer polynomials \cite[Section 1.5.1]{DX}
\begin{equation} \label{eq:limgC=gH}
 \lim_{\l \to 0}  \frac{1}{\left(\sqrt{\l}\right)^n} C_n^{(\l,\mu)}\left (\frac{x}{\sqrt{\l}}\right) = \k_n^\mu H_n^\mu(x),
\end{equation}
where $\k_{2n}^\mu = \left[ 2^{2n} n! (\mu + \f12)_n \right]^{-1}$ and $\k_{2n+1}^\mu = 
\left[ 2^{2n+1} n! (\mu + \f12)_{n+1}\right]^{-1}$.

\end{document}